\documentclass[11pt]{article}
\usepackage{amsmath,enumerate,amsfonts,amssymb,color,graphicx,amsthm}
\renewcommand{\raggedright}{\leftskip=0pt \rightskip=0pt plus 0cm}
\usepackage{multirow,caption,subcaption}
\usepackage[title]{appendix}
\usepackage{mathrsfs}
\usepackage{textcomp}
\usepackage{extarrows}
\usepackage{authblk}
\usepackage{titlesec}
\allowdisplaybreaks[4]
\usepackage{textcomp}
\usepackage{fancyhdr}
\usepackage{verbatim}
\usepackage{hyperref}
\usepackage{framed}
\usepackage[normalem]{ulem}
\usepackage{tikz}
\usetikzlibrary{patterns}
\usepackage[margin=1.25in]{geometry}
\usepackage{cite}
\titleformat{\section}{\sc\centering\Large}{\thesection}{1em}{}
\hypersetup{
	colorlinks,
	citecolor=blue,
	filecolor=blue,
	linkcolor=blue,
	urlcolor=black
}

\numberwithin{equation}{section}
\newtheorem{thm}{Theorem}[section]
\newtheorem{lem}{Lemma}[section]
\newtheorem{prop}{Proposition}[section]
\newtheorem{defn}{Definition}[section]
\newtheorem{cor}{Corollary}[section]
\newtheorem{rem}{Remark}[section]
\newtheorem{claim}{Claim}[]
\newtheorem{exa}{Example}[]
\renewcommand{\d}{\mathrm{d}}
\renewcommand{\H}{\mathbb{H}}
\newcommand{\R}{\mathbb{R}}
\newcommand{\Cn}{\mathbb{C}^n}
\newcommand{\Rn}{\mathbb{R}^n}
\newcommand{\Sn}{\mathbb{S}^n}
\renewcommand{\S}{\mathbb{S}}
\newcommand{\lam}{\lambda}
\newcommand{\pa}{\partial}
\newcommand{\var}{\varepsilon}
 \newcommand{\be}{\begin{equation}}
\newcommand{\ee}{\end{equation}}
\newcommand{\Lam}{\Lambda}
\newcommand{\si}{\sigma}
\newcommand{\al}{\alpha}

\newcommand{\Om}{\Omega}
\newcommand{\Hn}{\mathbb{H}^n}

\numberwithin{equation}{section}

\titleformat{\section}{\sc\centering\Large}{\thesection}{1em}{}

\makeatother

\begin{document}

\title{\bf On the CR Nirenberg problem:  density and multiplicity of solutions}
\author{\medskip  {\sc Zhongwei Tang}\footnote{Z. Tang is supported by National Science Foundation of China (12071036, 12126306).},  \ \
	{\sc Heming Wang},  \ {\sc Bingwei Zhang}     }

\date{}

\maketitle

\begin{abstract}
We prove some results on the density  and multiplicity of positive solutions to the prescribed Webster scalar curvature problem on the  $(2n+1)$-dimensional standard unit  CR sphere $(\S ^{2n+1},\theta_0)$.
Specifically, we construct arbitrarily many multi-bump solutions via the variational gluing method. In particular, we show  the Webster scalar curvature  functions of contact forms  conformal to $\theta_0$ are $C^{0}$-dense among bounded functions which are positive somewhere.
Existence results of  infinitely many positive solutions to the related equation $-\Delta_{\H} u=R(\xi) u^{(n+2) /n}$ on the Heisenberg group $\Hn $ with $R(\xi)$ being asymptotically periodic with respect to left translation are also obtained. Our proofs make use of  a refined analysis of bubbling behavior, gradient flow,  Pohozaev identity, as well as    blow up arguments.
\end{abstract}

\tableofcontents
\section{Introduction}\label{section1}
\subsection{The studied problem and its history}
The simplicity and subtlety of curvature and its global nature draw many attention. In particular, one celebrated problem, raised by Nirenberg in  the 1960’s, asks on the $n$-dimensional standard sphere $(\Sn,g_0)$ $(n\geq 2)$, if one
can find a conformally invariant metric $g$ such that the scalar curvature (Gauss curvature for $n=2$) of $g$ is equal to the given   function $K$. This is widely known as the Nirenberg problem and is also called the prescribed scalar curvature problem on $\Sn$. If we denote $g=e^{2v}g_0$ in the  case $n=2$
and $g=v^{4/(n-2)}g_0$ in the $n\geq 3$ dimensional case, this problem amounts to finding a positive solution $v$ of the equations
$$
-\Delta_{g_{0}}v+1=Ke^{2v} \quad\mbox{ on }\, \S ^2,
$$
and
\be\label{Nirenberg}
-\Delta_{g_{0}}v+c(n)R_0v=c(n)Kv^{\frac{n+2}{n-2}} \quad\mbox{ on }\, \Sn \quad  \mbox{ for } \,n\geq 3,
\ee
where  $\Delta_{g_{0}}$ is the Laplace-Beltrami operator on $(\Sn, g_{0})$, $c(n)=(n-2)/(4(n-1))$  and  $R_0=n(n-1)$ is the scalar curvature associated to $g_0$. The Nirenberg  problem has been studied extensively and it would be impossible to mention here all works in this area. Two significant aspects most  related to this paper are the fine analysis of blow up (approximate) solutions and the gluing methods in construction of solutions, see, e.g., \cite{Li1993,Li1995,Li1996,CGY1993,CY1991,WY2010,LWX2018,BC1988,BC1991,Li93d,SZ1996,MM2020} and references therein.

In the past half century, several studies have been performed for classical elliptic equations  which are similar to Nirenberg’s equations but with the  conformal sub-Laplacians on CR manifolds.  The geometry of CR manifolds, namely the abstract model of real hypersurfaces in complex manifolds, has attracted, since the late 1970’s, a lot of attention of prominent mathematicians as for instance, Chern-Moser \cite{CM1974}, Fefferman \cite{F1976}, Jacobowitz \cite{J1990}, Jerison-Lee \cite{JL1984,JL1987,JL1988,JL1989}, Tanaka \cite{T1975}, Webster \cite{W1978}, among many others. This geometry is very rich when the CR manifold admits a strictly pseudo-convex structure in which case we encounter a great analogy with the conformal geometry of Riemannian manifolds. Notably,  the study of the prescribing Webster scalar curvature problem on CR manifolds, which dates
back to Jerison-Lee \cite{JL1987,JL1989,JL1988}, has received a lot of attention, see, e.g.,  \cite{G2001,GY2001} and references therein.  For more recent and further studies, see \cite{DRY2022,CAY2010,CAY2013,R2013,R2012,GAG2015,SG2011} and related references.

As a natural analogue of the Nirenberg probelm for the CR geometry, one can consider the prescribed Webster (pseudo-hermitian) scalar curvature problem on the standard   CR sphere  which can be formulated as follows. Let $(\S^{2n+1},\theta_{0})$ be the unit   CR sphere in $\mathbb{C}^{n+1}$ with  $\theta_{0}$ being the standard contact form and $n\geq 1$. Given any  function $\bar{R}$ on $\S ^{2n+1}$, it is natural to ask: Does there exist a contact form $\theta$ conformally related to $\theta_0$ in the sense that $\theta=v^{2 / n} \theta_0$ for some  function $v>0$ such that $\bar{R}$ is the Webster scalar curvature of the Webster metric $g_{\theta}$ associated with the contact form $\theta$? Following the same way as in the Riemannian case, the Webster metric $g_{\theta}$ associated with $\theta$ obeys its scalar curvature which is given by $$
\operatorname{Scal}_{\theta}=u^{-\frac{n+2}{n}}\Big(-\frac{2(n+1)}{n} \Delta_{\theta_0} u+\operatorname{Scal}_{\theta_0} u\Big),
$$
where $\Delta_{\theta_0}$ is the sub-Laplacian with respect to the contact form $\theta_0$ and $\operatorname{Scal}_{\theta_0}=n(n+1)/2$ is the Webster scalar curvature of the Webster metric $g_{\theta_0}$ associated with the contact form $\theta_0$. Clearly, the problem of solving $\operatorname{Scal}_{\theta}=\bar{R}$ is equivalent to finding positive solutions $v$ to the following PDE
\be\label{maineq}
L_{\theta_{0}}v:=-\Delta_{\theta_0}v +\frac{n^2}{4} v=\bar{c}(n) \bar{R} v^{1+\frac{2}{n}},\quad v>0 \quad \text { on }\, \S^{2n+1}.
\ee
Here  $L_{\theta_{0}}$
is  called the  \emph{conformal sub-Laplacian} and transforms
according to the law  $L_{\theta}(\phi)=v^{-\frac{n+2}{n}} L_{\theta_0}(v \phi)$ for any  $\phi\in C^{\infty}(\S^{2n+1})$,
 the  sub-Laplacian operator $\Delta_{\theta_0} $   can be expressed explicitly in coordinates $\zeta=(\zeta_1, \ldots, \zeta_{n+1}) \in \S^{2 n+1}$ by
$$
\sum_{j=1}^{n+1} \frac{\pa^2}{\pa \zeta_j \pa \bar{\zeta}_j}+\sum_{j, k=1}^{n+1} \zeta_j \bar{\zeta}_k \frac{\pa^2}{\pa \zeta_j \pa \bar{\zeta}_k}+ \frac{n}{2}\sum_{k=1}^{n+1}\Big(\zeta_k \frac{\pa}{\pa \zeta_k}+\bar{\zeta}_k \frac{\pa}{\pa \bar{\zeta}_k}\Big),
$$
and $\bar{c}(n)=\frac{n}{2(n+1)}$, see, e.g., \cite{BFM2013,JL1987}. Geller  \cite{G1981} showed that,   for regular function $f: \S^{2 n+1} \to \mathbb{C}$, the function
\be\label{Green}
G(\zeta)=L_{\theta_{0}}^{-1}(f)(\zeta)=c_n\int_{\S^{2 n+1}} \operatorname{dist}(\zeta, \cdot)^{-2n} f(\cdot)
\ee
satisfies $L_{\theta_{0}} G=f$, where $c_n=\frac{2^{n-1} \Gamma(\frac{n}{2})^2}{\pi^{n+1}}$  and  $\operatorname{dist}(\zeta, \eta):=\sqrt{2|1-\zeta \cdot \bar{\eta}|}$  is the distance function on $\S^{2 n+1}$. We also refer to \eqref{Green}  as the \emph{Green's representation}  formula.

 Prescribing  such Webster scalar curvature  on the standard CR sphere can be interpreted as a generalization of the Nirenberg problem, called in this context: the \emph{CR Nirenberg problem}.  We will present more geometric and analytical backgrounds to understand our problem in Section \ref{sec:2.1}. Throughout the paper, we assume $n\geq 1$  without otherwise stated.

Equation \eqref{maineq} proves to be remarkably flexible and difficult to solve.  As we multiply $v$ to \eqref{maineq} and apply integration by parts, it is easy to see a simple necessary condition   is  $\max_{\S^{2n+1}} \bar{R} > 0$, but there are also some obstructions, which are said of \emph{topological type}. For example, a necessary condition is the  following Kazdan-Warner type condition  (see \cite[Theorem B]{C1991}): for any CR vector field $X$ on $\S^{2n+1}$, there holds
\be\label{KW1}
\int_{\S^{2 n+1}}X(\bar{R})  v^{2+\frac{2}{n}} \, \d vol_{\theta_0}=0
\ee
for any positive solution $v$ of  \eqref{maineq},  where  $\d vol_{\theta_0}$ denotes the volume of $\S^{2 n+1}$ with respect to $\theta_0$.
Note that the real or imaginary part of gradient (with respect to $\langle\cdot,\cdot\rangle_{\theta}$ under the Levi form $\theta$) of a bigraded spherical harmonic function $f$ of type $(1, 0)$ or $(0,1)$ is a CR vector field. This implies that \eqref{maineq} is not solvable if $\bar{R} =f+\text{constant}$. Another  difficulty in studying \eqref{maineq} is the lack of compactness due to the presence of the Sobolev critical exponent. A typical phenomenon encountered here is \emph{bubbling blow up}. Bubbles are solutions of \eqref{maineq} with $\bar{R}=1$, these arise as profiles of general diverging solutions and were classified
in \cite{JL1988} under the hypothesis that $v\in L^{2+\frac{2}{n}}$, which is equivalent to having finite volume. From the variational point of view, bubbles generate diverging Palais-Smale sequences for the Euler-Lagrange functional of \eqref{maineq}. Two main approaches have been used to understand the blow up phenomenon:  subcritical approximations (see, e.g., \cite{PR2003}), or the construction of pseudo-gradient flows (see, e.g., \cite{GH2017,GHM2020,SG2011}).	
Both of these two methods will be explored and put into application in  this paper.

The conformal sub-Laplacian operator $L_{\theta_{0}}$ can be seen more concretely when we deform the CR structure to the Heisenberg group, which is a flat CR manifold.   Let us   recall some  basic notions on the Heisenberg group first.

The Heisenberg group $\Hn $ is the Lie group whose underlying manifold is $\Cn   \times \R $  with coordinates $(z, t)$ and whose the group law $\circ$ defined by
$(z, t) \circ (\hat{z},  \hat{t}):=(z+\hat{z},  t+\hat{t}+2\operatorname{Im}(z\hat{z}))$.
 We will always use the notation $\xi=(z, t)=(x,y,t)$ with $z=x+i y$, $x=(x_1,\ldots,x_n)\in\Rn$ and $y=(y_1,\ldots,y_n)\in\Rn$ to denote an element in $\Hn$ and  $(\xi)^k:=\xi \circ \cdots \circ \xi$ to denote $k$-fold composition for simplicity.  Let $Q := 2n+2$ denote the homogeneous dimension of $\Hn$, see also \cite{FS1982}.  We consider
the norm on $\Hn$ defined  by $|\xi|:=(|z|^4 +t^2)^{1/4}$. The corresponding distance on $\Hn $ is defined accordingly by  $d(\xi, \xi_0):=|\xi_0^{-1} \circ \xi|$ for any $\xi,\xi_0\in \Hn$, where $\xi_0^{-1}$ is the inverse of $\xi_0$ with respect to $\circ$, i.e., $\xi_0^{-1}=-\xi_0$. In addition, we will denote by $B_r(\xi_0):=\{\xi\in \Hn:d(\xi_0,\xi)<r\}$ the ball with respect to the distance $d$, of center $\xi$ and radius $r$.

For any fixed $\xi_0 \in \Hn$ we will denote by $\tau_{\xi_0}: \Hn \to \Hn$ the left translation on $\Hn$ by $\xi_0$, defined by
$\tau_{\xi_0}(\xi)=\xi_0 \circ \xi$, while for any $\lam>0$ we will denote by $\delta_\lam: \Hn \to  \Hn$ the dilation defined by $\delta_\lam(\xi):=(\lam z, \lam^2 t)$.

Define the following left invariant vector fields in the  coordinate $(x,y,t)$:
\be\label{vectorfields}
X_j=\frac{\pa}{\pa x_j}+2 y_j \frac{\pa}{\pa t},\quad  Y_j=\frac{\pa}{\pa y_j}-2 x_j \frac{\pa}{\pa t},\quad T=\frac{\pa}{\pa t}. \ee
The Heisenberg gradient, or horizontal gradient, of a regular function
$u$ is then defined by $\nabla_{\H} u:=(X_1 u,\ldots,X_n u,Y_1 u,\ldots,Y_n u)$,
while its Heisenberg Hessian matrix is
\begin{align*}
	\nabla_\H^2 u: & =\left(\begin{array}{cccccc}
		X_1 X_1 u & \cdots & X_n X_1 u & Y_1 X_1 u & \cdots & Y_n X_1 u \\
		\vdots & \ddots & \vdots & \vdots & \ddots & \vdots \\
		X_1 X_n u & \cdots & X_n X_n u & Y_1 X_n u & \cdots & Y_n X_n u \\
		 X_1 Y_1 u & \cdots & X_n Y_1 u & Y_1 Y_1 u & \cdots & Y_n Y_1 u \\
		\vdots & \ddots & \vdots & \vdots & \ddots & \vdots \\
		X_1 Y_n u & \cdots & X_n Y_n u & Y_1 Y_n u & \cdots & Y_n Y_n u
	\end{array}\right).
\end{align*}
The Heisenberg Laplacian  is the trace of the above Heisenberg Hessian matrix, that is
$$L_0=-\Delta_{\H}:=- \sum_{j=1}^n(X_j^2+Y_j^2)=\sum_{j=1}^n\Big( \frac{\pa^2 u}{\pa x_j^2}+\frac{\pa^2 }{\pa y_j^2}+4 y_j \frac{\pa^2 }{\pa x_j \pa t}-4 x_j \frac{\pa^2 }{\pa y_j \pa t}+4(x_j^2+y_j^2) \frac{\pa^2 }{\pa t^2}\Big).$$

The Heisenberg group $\Hn$ is CR equivalent to the sphere $\S^{2n+1}\subset \mathbb{C}^{n+1}$ minus a point via the Cayley transform. The  Cayley transform from $\S^{2n+1} \backslash \{(0,\ldots,0,-1)\}$ to $\Hn$ is the inverse of
\be\label{inverse}
\mathcal{C}:\Hn \to \S^{2n+1} \backslash \{(0,\ldots,0,-1)\},\quad  (z,t) \mapsto\Big(\frac{2 z}{1+|z|^2+i t}, \frac{1-|z|^2-i t}{1+|z|^2+i t}\Big).
\ee
 Then using \eqref{Green} and  \eqref{inverse}, we have
$$
4(L_{\theta_{0}} \phi) \circ \mathcal{C}=(2|J_{\mathcal{C}}|)^{-(Q+2)/(2Q)} L_0((2|J_{\mathcal{C}}|)^{(Q-2)/(2Q)}(\phi \circ \mathcal{C}))\quad \text{ for }\,\phi \in C^{\infty}(\S^{2n+1}),
$$
where  $|J_{\mathcal{C}}|=\frac{2^{2 n+1}}{((1+|z|^2)^2+t^2)^{n+1}}$ is the deteminant of the Jacobian of $\mathcal{C}$.  Here and from now on, we also use the notation $\circ$  to denote the composition mapping of some functions.
Therefore, if we denote $u=(2|J_{\mathcal{C}}|)^{\frac{Q-2}{2 Q}}(v \circ \mathcal{C})$ and $R = \bar{R}\circ \mathcal{C}$, then problem \eqref{maineq} is equivalent to solving
\be\label{maineq1}
-\Delta_{\H} u=R(\xi) u^{(Q+2)/(Q-2)}, \quad u>0 \quad  \text{ in }\,\Hn,
\ee
up to  a harmless positive constant in front of $R(\xi)$.
Similar to \eqref{KW1},  Garofalo-Lanconelli \cite{GL1992} showed  that    a positive solution  $u$ to \eqref{maineq1} in the Sobolev space $E$ (with the notation in \eqref{inner}) satisfies the following identity:
\be\label{KW2}
\int_{\Hn}\langle(z, 2 t), \nabla R(z, t)\rangle u(z, t)^{2Q/(Q-2)} \,\d  z \,\d  t=0,
\ee
provided the integral is convergent and $R$ is bounded and suitably regular. This implies that there are no such solutions if $\langle(z, 2 t), \nabla R(z, t)\rangle$ does not change sign in $\Hn $ and $R$ is not constant.

Prescribing   Webster curvature  on $\S ^{2n+1}$ is a focus of reserach in the past decades and it continues to inspire new thoughts. Recent existence results mainly use
\begin{itemize}
	\item the 	Webster scalar curvature flow method (e.g., \cite{Ho20151,Ho20152,Ho2016,Ho2020}),	
	\item the reduced finite dimension variational method  (e.g., \cite{CPY2013,MU2002,FU2004}), or
\item 	the critical points at infinity method (e.g., \cite{GH2017,GHM2020,SG2011,RG2012,RG2012-2,R2013,R2012}).
\end{itemize}
The majority of these require the solution set to be uniformly bounded.
The main objective of this paper is  to include a larger class of functions $\bar{R}$ such that problem \eqref{maineq} is solvable.   Moreover,    the number of multi-bump solutions to \eqref{maineq} will be investigated, subject to some local hypotheses regarding the prescribed function $\bar{R}$.
Basically speaking, we demonstrate the Webster scalar curvature  functions of contact forms  conformal to $\theta_0$ are $C^{0}$-dense among bounded functions which are positive somewhere by constructing  multi-bump solutions to the  perturbed equations.  As a variation of this idea, the related problem \eqref{maineq1}  with $R(\xi)$ being periodic with respect to left translation are also studied and infinitely many multi-bump solutions (modulo left translations by its periods) are obtained under certain  flatness conditions.
\subsection{Main results}
We now list the main results of this paper and some remarks on them.
The first one deals with the existence of multi-bump solutions to the perturbed \emph{CR Nirenberg problem}.
\begin{thm}\label{thm:1}
  Let  $\bar{R}\in L^\infty(\S ^{2n+1})$ be a given function.   Suppose that there exists a point  $q_0\in \S ^{2n+1}$  such that
 $\bar{R}(q_0)>0$ and  $\bar{R}$ is continuous in a geodesic ball $B(q_0,\widetilde{\var })$ for some $\widetilde{\var }>0$.
  Then, for any $\var  \in (0, \widetilde{\var })$, any integers $k \geq 1$ and  $m\geq 2$, there exists a function    $\bar{R}_{\var , k, m}\in L^{\infty}(\S ^{2n+1})$ satisfying  $\bar{R}_{\var , k, m}-\bar{R}\in C^0(\S ^{2n+1})$,  $\| \bar{R}_{\var , k, m}-\bar{R} \|_{C^0(\S ^{2n+1})}<\var $, and $\bar{R}_{\var , k, m}\equiv \bar{R}$ in $\S ^{2n+1}\backslash B(q_0,\var)$.
Furthermore, for any integer $2\leq s \leq m$, the perturbed equation
  \be \label{perturbed}
  -\Delta_{\theta_0}v +\frac{n^2}{4} v=\bar{c}(n) \bar{R}_{\var , k, m} v^{1+\frac{2}{n}},\quad v>0 \quad \text { on }\, \S^{2n+1}
  \ee
  has at least $k$  positive solutions with $s$ bumps. Here we   denote by $B(q,\var)$   the geodesic ball in $\S ^{2n+1}$ with radius $\var$ and center $q$.
\end{thm}

By applying the Kazdan-Warner type condition \eqref{KW1}, we know that one cannot expect to perturb any $\bar{R}$ near any point $\zeta \in \S ^{2 n+1}$ in the sense of $C^{1}$ in order to obtain the existence of solutions.   For the precise  meaning of \emph{$s$ bumps}, see the  proof of Theorem \ref{thm:1} in Section \ref{sec:5}.	Roughly speaking,  a solution is said to have  $s$ \emph{bumps}  when the majority of its mass is concentrated in $s$ disjoint regions.  As both the number of \emph{bumps} and the number of solutions can be chosen arbitrarily, we can conclude the existence of infinitely many multi-bump solutions to equation \eqref{perturbed}.

The main feature of Theorem \ref{thm:1}  is that, even if a  given bounded function $\bar{R}$ which is positive somewhere  cannot be realized as the Webster scalar curvature of a contact form $\theta$ conformal to $\theta_{0}$,  nevertheless we can find a function $\bar{R}'$ arbitraly close to $\bar{R}$ in $C^{0}(\S ^{2 n+1})$ which is the Webster scalar curvature  as many conformal contact forms to $\theta_{0}$ as we want.  Here we give a  quite general existence result  since we can perturb  any given bounded  function which is positive somewhere   such that for the perturbed equations there exist arbitrarily many solutions.

As a consequence, we  have
\begin{cor}\label{cor:1}
	The  Webster scalar curvature  functions of contact forms  conformal to $\theta_0$ are dense in $C^{0}(\S^{2n+1})$ among bounded functions which are positive somewhere.
\end{cor}

Next we consider the related problem \eqref{maineq1}. Before  stating the  results, we introduce some notations.

Let $E$ be the completion of the space $C_c^\infty(\Hn)$ with respect to the  $\|\cdot\|$ norm introduced by the scalar  product
\be\label{inner}
\langle u,v\rangle: = \int_{\Hn} \nabla_{\H}u \nabla_{\H}v \, \d z\,\d t.
\ee    Whenever there is no risk of misunderstanding, we suppress $\d z\,\d t $ from
the integration expressions on domains in $\Hn$  and omit the  integral region   if it is $\Hn$.

When $R\equiv 1$, all solutions of \eqref{maineq1} satisfying the finite energy assumption $u\in L^{2Q/(Q-2)}(\Hn)$   have been classified by  Jerison-Lee \cite{JL1988} and are given by
\be\label{bubbles}
w_{a, \lam }:=\lam ^{(2-Q)/2}w_{0,1}\circ \delta_\lam  \circ \tau_{a^{-1}},
\ee  for any $a \in \Hn$ and $\lam  >0$,  where $ w_{0,1}(z, t)=c_0(t^2+(1+|z|^2)^2)^{(2-Q) / 4}$ with  $c_0>0$ being a suitable constant  depends only on $n$. Similar classification result  has been obtained in  \cite{GV2001}
under the assumption of cylindrical symmetry. Recently,  Catino, Li, Monticelli and Roncoroni \cite{CLMR2023} proved  a classification of all positive solutions  in $\H^1$  and a classification of positive solutions when $n\geq 2$ that satisfy
a suitable decay condition at infinity, which is weaker than finite energy assumption. Inspired by \cite{CLMR2023},  Afeltra \cite{Afeltra2024}  obtained a compactness result for the CR Yamabe problem in  dimension three.

Denote the  Sobolev critical exponent $Q^*:=\frac{2Q}{Q-2}$.  It is well-known (see \cite{JL1987}) that $E$ can be embedded into $L^{Q^*}(\Hn)$ and the sharp Sobolev inequality (or Folland-Stein inequality \cite{FS1974}) is
\be\label{Sobolevineq}
S_n \Big(\int |u|^{Q^*}\Big)^{1 / Q^*} \leq \Big(\int |\nabla_{\H} u|^2\Big)^{1 / 2},
\ee
where $S_{n}=\frac{ 2n \sqrt{\pi}}{(2^{2 n} n !)^{1 /(2(n+1))}}$ is the best constant. Then for every $(a,\lam)\in \Hn\times (0,\infty)$,  $w_{a, \lam }$ is the solution to
\eqref{maineq1} with $R\equiv 1$. Moreover, the functions in \eqref{bubbles} and its non-zero constant multiples attain the sharp Sobolev inequality \eqref{Sobolevineq} and such functions are usually called \emph{Jerison-Lee's bubbles}.

Let $R \in L^\infty(\Hn)$, we define the energy functional $I_R:E\to \R$ by
$$I_{R}(u) = \frac{1}{2}\int |\nabla_{\H }u|^2 - \frac{1}{Q^*} \int R |u|^{Q^*}.
$$
Obviously a positive critical point gives rise to a positive solution to \eqref{maineq1}.

Let $R(\xi)\in L^{\infty}(\Hn)$,  $O^{(1)}, \ldots, O^{(k)} \subset \Hn$ are some open sets with $\operatorname{dist}(O^{(i)}, O^{(j)}) \geq 1$  for any $ i \neq j$. If $R \in C^{0}(\cup_{i=1}^{k} O^{(i)})$,
we define $V(k, \var):=V(k, \var, O^{(1)}, \ldots, O^{(k)}, R)$  as the following open set in  $E$ for $\var>0$:
\be\label{bumps}\begin{aligned}
	V(k, \var):=\Big\{
	u\in E:&\,\exists \,\al=(\al_{1}, \ldots, \al_{k}) \in \mathbb{R}^k,\, \exists\, \xi=(\xi_{1}, \ldots, \xi_{k}) \in O^{(1)}\times \ldots \times O^{(k)}, \\&\,\exists\,\lam =(\lam _{1}, \ldots, \lam _{k}),\,\lam _{i}>\var^{-1},\,\forall \,i\leq k, \text{ such that}\\ &\, |\al_{i}-R(\xi_{i})^{(2-Q)/4}|<\var,\, \forall \,i\leq k,\text{ and}\, \Big\|u-\sum_{i=1}^{k} \al_{i} w_{\xi_i,\lam_i}\Big\|<\var
	\Big\}.\end{aligned}
\ee
 The open set  $V(k, \var)$ recodes the information of the concentration rate and the locations of concentration points, it also describes the neighborhood of \emph{potential critical points at infinity}.

Recently,  there have been some works devoted to the existence results via
studying the flatness condition effect, see, e.g., \cite{GH2017,GHM2020,SG2011,RG2012,RG2012-2}.  Here we will adopt the flatness hypothesis introduced in \cite{PR2003}, which is modified from \cite{Li1995}.

\textbf{Flatness condition:}	For any real munber $\beta>1$, we say that a sequence $\{R_i\}$ of functions  satisfies  condition $(*)_{\beta}$ for some sequence of constants  $\{ L_1(\beta, i) \}$, $\{ L_2(\beta, i) \}$ in some region $\Om_i\subset \Hn$ if  $\{R_i\} \in C^{[\beta]-1,1}(\Om _i)$ satisfies
$$	\| \nabla R_i \|_{C^0(\Om _i)} \leq L_1(\beta, i)
	$$
	and, if $\beta\geq 2$,
	$$
	|\nabla^s R_i(\xi)| \leq L_2(\beta, i)|\nabla R_i(\xi)|^{(\beta-s)/(\beta-1)}
	$$
	for all $2\leq s \leq [\beta]$, $\xi \in \Om _i$, $\nabla R_i(\xi) \neq 0$. Here and in the following, $\nabla^{s}$  denotes all possible partial derivatives of order $s$.

For $1 \leq j \leq 2 n$, we denote
$$
L_j=\left\{\begin{aligned}
&X_j,&&\text{ if }\,1\leq j\leq n,\\&Y_{j-n}, &&\text{ if }\, n+1\leq j\leq 2n,
\end{aligned}\right.
$$
where $X_j$, $Y_j$ are the left invariant vector fileds defined by \eqref{vectorfields}.
Let $\mathscr{B}_k=\{L_{a_1} \cdots L_{a_j}: 1 \leq a_i \leq 2 n, i=1, \ldots, j, j \leq k\}$ and $\mathscr{A}_k$ be the  linear span over $\mathbb{C}$ of $\mathscr{B}_k \cup\{\operatorname{Id}\}$.

Let $\Om\subset \Hn$ be an open set.  Using the notations  in  Folland-Stein \cite{FS1974}, we define the nonisotropic Lipschitz space $\Gamma_{\beta}(\Om)$ as follows. If $\beta \in (0,1)$, define
$$\Gamma_{\beta}(\Om)=\Big\{f\in L^{\infty}(\Om)\cap C^0(\Om):\sup_{\xi,\zeta\in \Om}\frac{|f(\xi)-f(\xi\circ\zeta)|}{|\zeta|^{\beta}}<\infty\Big\}.$$
If $\beta=1$, define $$\Gamma_{1}(\Om)=\Big\{f\in L^{\infty}(\Om)\cap C^0(\Om):\sup_{\xi,\zeta\in \Om}\frac{|f(\xi\circ\zeta)-2f(\xi)+f(\xi\circ\zeta^{-1})|}{|\zeta|}<\infty\Big\}.$$
If $\beta=k+\al$ with $k\in \mathbb{N}^+$  and $\al\in (0,1)$, define
$$\Gamma_{k+\al}(\Om)=\Big\{f\in L^{\infty}(\Om)\cap C^0(\Om):\mathscr{L}f\in \Gamma_{\al}(\Om)\,\text{ for $\mathscr{L}\in \mathscr{B}_k$} \Big\}.$$
  We can also define Lipschitz space $\Gamma_{\beta}$ on CR manifold in terms of the normal coordinates, see \cite{FS1974}.
  Note that we can  identify $\Hn$ with its Lie algebra which is Euclidean
space $\R^{2n+1}$ with the Euclidean norm $|\cdot|$ and the linear coordinates $x_j$ via the exponential map. Hence, we are able to  discuss the  usual smooth space $C^k$ for $0\leq k\leq \infty$. We refer to \cite{FS1974,F1975} for more details and regularity results.


The family of solutions we constuct is of the form (after using the CR equivalence for  \eqref{maineq}) $u=\sum_{i=1}^{k} \alpha_{i} w_{\xi_i,\lam_i}+v$, where the contribution of the error term $v$ can be  negligible. Moreover, the multi-bump solutions concentrate  near some critical points of $R(\xi)$ and the bumps can be chosen  arbitrarily many.  For this purpose, we assume that  $R(\xi)\in  \Gamma_{2+\al}(\Hn)$ satisfies the following conditions:
 \begin{itemize}
\item[$(R_1)$]  $R(\xi)$ is periodic in some $\hat{\xi}\in \Hn$ with respect to left translation, that is, $R(\hat{\xi}\circ \xi) = R(\xi)$, $\forall\,\xi \in \Hn$.
\item[$(R_2)$]  Let $\Sigma$ be the set of the critical points $\bar{\xi}$ of $R(\xi)$ satisfying: there exists some real number $\beta\in (Q-2,Q)$ such that near 0, 
$$
  \mathcal{R}(\xi)=\mathcal{R}(0) + \sum_{j=1}^{n}(a_i|x_j|^\beta +b_j|y_j |^\beta) +c|t|^{\frac{\beta}{2}} +P(\xi),
$$
 where  $\mathcal{R}(\xi):=R(\overline{\xi}\circ \xi)$,  $a_i, b_i, c$ are some non-zero constants depending on $\bar{\xi}$,  $\sum_{j=1}^{n}(a_j +b_j)+\kappa c \neq 0$ with $$\kappa=\frac{\int|x_1|^\beta w_{0,1}^{2Q/(Q-2)}}{\int |t|^{\frac{\beta}{2}}w_{0,1}^{2Q/(Q-2)}},$$  and  $P(\xi)$ is $C^{[\beta]-1,1}$ (up to $[\beta]-1$ derivatives are Lipschitz functions, $[\beta]$ denotes the integer part of $\beta$) near 0 and satisfies $$\sum_{s=0}^{[\beta]}|\nabla^sP(\xi)||\xi|^{-\beta+s} =o(1)\quad \text{as $\xi$ tends to  0.}$$
 \end{itemize}
\begin{rem}
Condition $(R_2)$ guarantees that $R$ satisfies  condition $(*)_{\beta}$ in a neighborhood of 0.  Notably, the range  $\beta\in (Q-2,Q)$ is a technical hypothesis to do blow up analysis based on the earlier work in  \cite{PR2003}, where a sequence of solutions can not blow up at more than one point. We also conjecture that if  $\beta=Q-2$,  the phenomenon of multiple blowups   would occur, as shown in \cite{Li1996}.
\end{rem}

We now establish the existence of multi-bump solutions to problem \eqref{maineq1}.
\begin{thm}\label{thm:2}
Assume that $R \in \Gamma_{2+\al}(\Hn)$ satisfies $(R_1),(R_2)$ and
 \begin{itemize}
\item[$(R_3)$] $R_{\max} := \max_{\xi \in \Hn}R(\xi)>0$ is achieved, and $R^{-1}(R_{\max}):=\{ \xi \in \Hn: R(\xi) = R_{\max}\}$ has at least one bounded connected component, denoted as $\mathscr{C}$.
 \end{itemize}
Then for any integer $m \geq 2$, \eqref{maineq1} has infinitely many $m$-bump solutions in $E$.
More precisely, for any $\var >0$, $\xi^* \in \mathscr{C}$ and integer $m \geq 2$, there exists a constant $l^*>0$ such that for any integers $l^{(1)}, \ldots , l^{(k)}$ satisfying $2 \leq k \leq m$ and the conditions $\min_{1 \leq i \leq k}|l^{(i)}|, \min_{i \neq j}|l^{(i)}-l^{(j)}|\geq l^*$, there exists at least one solution  $u$ of \eqref{maineq1} in $V(k, \var , B_\var (\xi^{(1)}), \ldots , B_\var (\xi^{(k)}))$ with $kc-\var  \leq I_R(u) \leq kc+\var $,
where \begin{gather*}
c = (R(\xi^*))^{(2-Q)/2}(S_n)^Q/Q, \quad  \xi_l^{(i)} = (\hat{\xi})^{l^{(i)}} \circ\xi^*,
\end{gather*} and $V(k, \var , B_\var (\xi^{(1)}), \ldots , B_\var (\xi^{(k)}))$ are some sets of $E$ defined  according to \eqref{bumps}.
\end{thm}

From the description of  $(R_3)$ we know that there exists a bounded  neighborhood $O$ of $\mathscr{C}$ such that $R_{\max}\geq \max_{\xi\in\pa O}R(\xi)+\delta$ with $\delta>0$ being a small  constant. This fact together with  $(R_2)$ implies that $R(\xi)$ has a sequence of local maximum points $\xi_j$ with $|\xi_j|\to  \infty$ as $j\to  \infty$. Furthermore,  $(R_3)$ is sharp in the sense that one can construct examples easily to show
that if  $(R_3)$ is not satiesfied, \eqref{maineq1} may have no nontrivial solutions, which shows that $(R_3)$ is
not merely a technical hypothesis, see Example \ref{exa} below.

\begin{exa}[Nonexistence]\label{exa} Suppose that $R(\xi) \in C^{1}(\Hn) \cap L^{\infty}(\Hn)$ and $\nabla_{\H} R$ are bounded in $\Hn$, $X_i R$ is nonnegative but not identically zero. Then the only nonnegative solution of \eqref{maineq1} in $E$ is the trivial solution $u\equiv 0$.
\end{exa}
\begin{proof}
	Let $u\geq 0$ be any  solution of  \eqref{maineq1} in  $E$. By using the Kazdan-Warner condition \eqref{KW1}  we obtain $\int X_i R u^{2Q/(Q-2)}=0$.
	The hypotheses on $R(\xi)$ imply that $u$ is identically zero in an open set, hence $u\equiv 0$ by the  unique continuation results (see, e.g., \cite{GL1992,GL1990}).
\end{proof}
From the definition in \eqref{bumps}	we know that $u\in V(k, \var, B_{\var}(\xi^{(1)}), \ldots, B_{\var}(\xi^{(k)}))$ implies $u$ has most of its mass concentrated in $B_{\var}(\xi^{(1)}), \ldots, B_{\var}(\xi^{(k)})$. In particular, if the tuples $(l^{(1)}, \ldots, l^{(k)})$  and $(\tilde{l}^{(1)}, \ldots, \tilde{l}^{(k)})$ are different, the  solutions $u$ and $\tilde{u}$ are different.

A more comprehensive understanding of the solutions derived in Theorem \ref{thm:2} can be achieved.
\begin{thm}\label{thm:3}
Assume that $R \in L^\infty(\Hn)$ satisfies $(R_1),(R_2)$ and
 \begin{itemize}
\item[$(R_3)'$] there exist a constant $A_1>1$ and a bounded open set $O \subset \Hn$ such that
\begin{gather*}
R \in C^1(\overline{O}),\\
1/A_1\leq R(\xi) \leq  A_1, \quad \forall\, \xi \in \overline{O},\\
\max_{\xi\in \overline{O}}R(\xi) = \sup_{\xi \in \Hn}R(\xi)>\max_{x\in \pa  O} R(\xi).
\end{gather*}
 \end{itemize}
Then for any $\var >0$,  \eqref{maineq1} has infinitely many $m$-bump solutions $u$ in $E$ satisfying
\be \label{8}
  c \leq I_R(u)\leq c+\var  \quad \text{ or } \quad 2c-\var  \leq I_R(u) \leq 2c+\var
\ee
and
$$
  \sup \{\|u\|_{L^\infty(\Hn)}: I_R' (u)=0, u>0, u \in E, u\,\text{satisfies}\,\eqref{8}\}=\infty,
$$
where $$c=(\max_{\xi\in \overline{O}}R(\xi))^{(2-Q)/2}(S_n)^Q/Q.$$
More precisely, for any $\var >0$, there exists $l^*>0$ such that for any integers $l^{(1)}, l^{(2)}$ satisfying $|l^{(1)}-l^{(2)}|\geq l^*$, there exists  at least one solution $u$ of  \eqref{maineq1}  in $V(1, \var , O,R)\cup V(2, \var , O_l^{(1)}, O_l^{(2)}, R)$, where
  $$O_l^{(i)}=(\hat{\xi})^{l^{(i)}}\circ O:=\{(\hat{\xi})^{l^{(i)}}\circ \xi:\xi\in  O\}\quad \text{ for $i=1,2$,}$$ and
 $V(1, \var , O,R)\cup V(2, \var , O_l^{(1)}, O_l^{(2)}, R)$ are some sets of $E$ defined  according to \eqref{bumps}.
\end{thm}
\begin{rem}
By utilizing \eqref{inverse},  it is evident that the solutions obtained in Theorems \ref{thm:2} and  \ref{thm:3} can be lifted to a solution of \eqref{maineq} on $\S^{2n+1}$ which is positive except at the point $\{(0, \ldots, 0,-1)\}$. In this sense, \eqref{maineq} is solvable under the assumptions of Theorems \ref{thm:2} and  \ref{thm:3}.
\end{rem}

\subsection{Plan of the paper and comment on the proof}

In search for metrics of constant scalar curvatures, Yamabe \cite{Y1960} initiates the subcritical method, which is now  one of the most natural approaches to study conformal equations with Sobolev critical exponent. We also refer to the reader \cite{Li1995,Li1996,Li93d,ES1986,LP1987,SchoenYau1994,Aubin1998,L1995,MM2020}. In this paper, we will   study the \emph{CR Nirenberg problem} by using the mentioned subcritical approach.  While recognizing the usefulness of compactness in finding solutions of equation \eqref{maineq}, one is left to ponder the dilemma: By selecting those functions $\bar{R}$  so that blow ups are impossible (i.e., compactness regained), we naturally miss functions that can afford a bounded and a blow up subcritical sequences. This intriguing thought breathes the idea that blow ups need not always be harmful in finding solutions. Under suitable conditions, we still can use a blow up subcritical sequence to produce a solution by removing the singularities.  Such considerations will be conducted in our final arguments.

We end the introduction with some remarks and history on the variational gluing technique developed by Ser\'e, Coti Zelati and Rabinowitz. The basic idea is as follows: Given finitely many solutions (at low energy), to translate their supports far apart and patch the pieces together create many multi-bump solutions.  The authors in  \cite{CR1,CR2,CES,Se}  have introduced the original and powerful ideas which permit the construction of such solutions via variational methods. In particular, they are able to find many homoclinic-type solutions to periodic Hamiltonian systems (see \cite{Se,CR1}) and to certain elliptic equations of nonlinear Schr\"odinger type on $\Rn$ with periodic coefficients (see \cite{CR2}). Li  has given a slight modification to the minimax procedure in \cite{CR1,CR2} and has applied it to certain problems where periodicity is not present, for example, the problem of prescribing scalar curvature on $\Sn$ (see \cite{Li1993,Li93d,Li1995,Li1996}).  Inspired by the above  works, we attempt to modify the above mentioned gluing method towards equations  in the CR setting or the conformal sub-Laplacian operators under a particular choice of  contact forms. This paper also overcomes the difficulty appearing in using  Lyapunov-Schmidt reduction method to locate the concentrating points of the solutions. We also believe that this gluing method can be applied to the construction of multi-bump solutions for various problems in conformal CR geometry, for instance,  the    Nirenberg type problem involving CR fractional sub-Laplacians, see, e.g.,  \cite{CW2017,LW2018}.

Let us introduce the structure of the paper and comment on the proof. Theorems \ref{thm:1}--\ref{thm:3} and  Corollary \ref{cor:1} are  derived in Section \ref{sec:5} from Proposition \ref{prop:5.1}, a more general result on \eqref{maineq1}. To  establish Proposition \ref{prop:5.1}, we first study a compactified problem Theorem \ref{thm:4.1} in Section \ref{sec:4}. Then we derive Proposition \ref{prop:5.1} by using Theorem \ref{thm:4.1} and some blow up analysis  in \cite{PR2003}. Theorem \ref{thm:4.1} is a  technical result in our paper, which is essential to make the variational gluing methods applicable. The proof of Theorem \ref{thm:4.1} will be divided into two parts:  we first  follow and refine the analysis of Bahri-Coron \cite{BC1988} to study the subcritical interaction of two well-spaced bubbles in Subsection \ref{sec:4.1} and then    apply the minimax procedure as in Coti Zelati-Rabinowitz \cite{CR1,CR2} to complete the  proof of Theorem \ref{thm:4.1} in Subsection \ref{sec:4.2}.
Our presentation is largely influenced by the papers \cite{Li1993,Li1995,Li1996,Li93d} which studied existence and compactness of solutions to the classical Nirenberg problem. Although certain parts of
the proof can be obtained by some modifications of the  arguments in \cite{Li1993,Li1995,Li1996,Li93d}, there are
plenty of technical difficulties which demand new ideas to handle  subelliptic  equations.

The present paper is organized as the following.  In Section \ref{sec:pre}, we present some analytic and geometric tools necessary to investigate the \emph{CR Nirenberg problem},  and several preliminary results will be established.
In Section \ref{sec:4}, existence and multiplicity result  for the subcritical case (Theorem \ref{thm:4.1})  will be stated, and its
proof will be sketched. The details of the proof  then will be carried out in Subections \ref{sec:4.1} and \ref{sec:4.2}. The main theorems are proved
in Section \ref{sec:5} with the aid of blow up analysis  developed by Prajapat-Ramaswamy \cite{PR2003} and the refine analysis of blow up profile  established in Appendix \ref{app:A}.

\subsection*{Notation}
We collect below a list of the main notation used throughout the paper.
\begin{itemize}\raggedright
	\item  We  always use the notation $\xi=(z, t)=(x,y,t)$ with $z=x+i y$, $x=(x_1,\ldots,x_n)\in\Rn$ and $y=(y_1,\ldots,y_n)\in\Rn$ to denote an element in $\Hn$. We denote  $\xi^{-1}$ as the inverse of $\xi$,  $(\xi)^k=\xi \circ \cdots \circ \xi$ means $k$-fold composition and $(\xi)^{-k}:=(\xi^{-1})^{k}$.   	
	\item 	We  denote the norm on $\Hn$ by $|\xi|=(|z|^4 +t^2)^{1/4}$ and the dilations by $\delta_\lam (\xi)= (\lam  z, \lam  ^2 t)$ for $\lam>0$.
	  The  distance function on $\Hn $ is denoted as  $d(\xi, \xi_0)=|\xi_0^{-1} \circ \xi|$ for any $\xi,\xi_0\in \Hn$, and  the left translation on $\Hn$ by $\xi_0$ is denoted as $\tau_{\xi_0}(\xi)=\xi_0 \circ \xi$.	
	\item	For any $\xi_0\in \Hn$ and $r>0$, denote the ball  $B_r(\xi_0)=\{\xi\in \Hn:d(\xi,\xi_0)<r\}$ and its boundary  $\pa B_{r}(\xi_0)=\{\zeta\in \Hn:d(\xi,\xi_0)=r\}$.  We will not keep writing the center $\xi_0$ if $\xi_0=0$.

\item  For any $q\in \S ^{2n+1}$,  we denote by $B(q,\var)$   the geodesic ball in $\S ^{2n+1}$ with radius $\var$ and center $q$.

	\item For $n\geq 1$, we denote $Q=2n+2$, $Q^*=\frac{2Q}{Q-2}$ and $H(z, t) =(\frac{4}{t^2 + (1+|z|^2)^2})^{(Q-2)/4}$.
		
	\item   The integral $\int$ always means $\int_{\Hn}$ unless specified.	
	
	\item  $C>0$ is a generic constant which can vary from line to line. Moreover, a notation $C(\al, \beta, \ldots)$ means that the positive constant $C$ depends on $\al, \beta, \ldots$.
\end{itemize}

\section{Preliminaries}\label{sec:pre}

In this section, we present some geometric and analytical backgrounds to understand our problem.  We also collect a Pohozaev identity, establish  some a priori estimates to subcritical solutions, and study a minimization problem on exterior domain.

\subsection{Review on the  CR geometry}\label{sec:2.1}
We start with recalling a basic material on CR manifolds, we refer to \cite{DT2006} for the details.

Let $M$ be an orientable CR manifold without boundary of CR dimension $n$. This is also equivalent to saying that $M$ is an orientable differentiable manifold of real dimension $(2 n+1)$ endowed with a pair $(H(M), J)$, where $H(M)$ is a subbundle of the tangent bundle $T(M)$ of real rank $2 n$ and $J$ is an integrable complex structure on $H(M)$. Since $M$ is orientable, there exists a 1-form $\theta$ called \emph{pseudo-Hermitian} structure on $M$. Then, we can associate each structure $\theta$ to a bilinear form $G_\theta$, called \emph{Levi form}, which is defined only on $H(M)$ by
$$
G_\theta(X, Y)=-(\d \theta)(J X, Y), \quad \forall\, X, Y \in H(M) .
$$
Since $G_\theta$ is symmetric and $J$-invariant, we then call $(M, \theta)$ \emph{strictly pseudo-convex} CR manifold if the Levi form $G_\theta$ associated with the structure $\theta$ is positive definite. The structure $\theta$ is then a contact form which immediately induces on $M$ the volume form $\theta \wedge(\d \theta)^n$.

Moreover, $\theta$ on a strictly pseudo-convex CR manifold $(M, \theta)$ also determines a \emph{normal} vector field $T$ on $M$, called  \emph{the Reeb vector field} of $\theta$. Via the Reeb vector field $T$, one can extend the Levi form $G_\theta$ on $H(M)$ to a semi-Riemannian metric $g_\theta$ on $T(M)$, called the Webster metric of $(M, \theta)$. Let
$$
\pi_H: T(M) \to  H(M)
$$
be the projection associated to the direct sum $T(M)=H(M) \oplus \mathbb{R} T$. Now, with the structure $\theta$, we can construct a unique affine connection $\nabla$, called the \emph{Tanaka-Webster connection} on $T(M)$. Using $\nabla$ and $\pi_H$, we can define the \emph{horizontal gradient} $\nabla_\theta$ by
$$
\nabla_\theta u=\pi_H \nabla u .
$$

Again, using the connection $\nabla$ and the projection $\pi_H$, one can define the sub-Laplacian $\Delta_\theta$ acting on a $C^2$-function $u$ via
$$
\Delta_\theta u=\operatorname{div}(\pi_H \nabla u) .
$$
Here $\nabla u$ is the ordinary gradient of $u$ with respect to $g_\theta$ which can be written as $g_\theta(\nabla u, X)=$ $X(u)$ for any $X$. Then integration by parts gives
$$
\int_M(\Delta_\theta u) f \,\theta \wedge(\d \theta)^n=-\int_M\langle\nabla_\theta u, \nabla_\theta f\rangle_\theta \,\theta \wedge(\d \theta)^n
$$
for any smooth function $f$. In the preceding formula, $\langle \cdot,\cdot\rangle_\theta$ denotes the inner product via the Levi form $G_\theta$ (or the Webster metric $g_\theta$ since both $\nabla_\theta u$ and $\nabla_\theta v$ are horizontal). 

Having $\nabla$ and $g_\theta$ in hand, one can talk about the curvature theory such as the curvature tensor fields, the pseudo-Hermitian Ricci and scalar curvature. Having all these, we denote by $\mathrm{Scal}_\theta$ the pseudo-Hermitian scalar curvature associated with the Webster metric $g_\theta$ and the connection $\nabla$, called the Webster scalar curvature, see \cite[Proposition 2.9]{DT2006}.

Being a pseudohermitian structure defined only up to a conformal factor on a CR manifold, the CR Yamabe problem  is a natural analogue of the Yamabe problem in Riemannian geometry. If $\widehat{\theta}=u^{2 / n} \theta$ for some smooth function $u>0$, the transformation law of the Webster curvature is
$$
\operatorname{Scal}_{\widehat{\theta}}=u^{-\frac{n+2}{n}}\Big(-\frac{2(n+1)}{n} \Delta_\theta u+\operatorname{Scal}_\theta u\Big).
$$
Clearly, the problem of solving $\operatorname{Scal}_{\widehat{\theta}}=h$ is equivalent to finding positive solutions $u$ to the following PDE:
\be\label{Yambe}
-\Delta_\theta u+\frac{n}{2(n+1)} \operatorname{Scal}_\theta u=\frac{n}{2(n+1)} h u^{1+2 / n} \quad \text { on }\, M .
\ee
When $h$ is constant, \eqref{Yambe} is known as the CR Yamabe problem.

Basic examples of CR manifolds include real hypersurfaces in $\mathbb{C}^{n+1}$, for example, any odd-dimensional unit sphere $\S^{2 n+1} \subset \mathbb{C}^{n+1}$ is a strictly pseudo-convex CR manifold.  Indeed, let $\theta_0$ be the standard contact form on the sphere $\S^{2 n+1}=\{\zeta=(\zeta^1, \ldots, \zeta^{n+1}) \in \mathbb{C}^{n+1} :| \zeta|^2=\sum_{j=1}^{n+1}|\zeta^j|^2=1\} \subset \mathbb{C}^{n+1}$, i.e.,
$$
\theta_0=\sqrt{-1}(\bar{\pa}-\pa)|\zeta|^2=\sqrt{-1} \sum_{j=1}^{n+1} (\zeta^j \, \d \bar{\zeta}^j-\bar{\zeta}^j \, \d \zeta^j) .
$$
Then $(\S^{2 n+1}, \theta_0)$ is a compact strictly pseudoconvex CR manifold of real dimension $(2 n+1)$.
The  Heisenberg group $\Hn$ as mentioned in the previous section  is  a more special example.   $\Hn$ plays a role among pseudoconvex pseudohermitian manifolds analogous to the role of $\Rn$ among Riemannian manifolds. In fact,  every pseudoconvex pseudohermitian manifold can locally be appoximated with $\Hn$, through coordinates analogous to the normal coordinates of Riemannian geometry known as pseudohermitian normal coordinates.

Since the Heisenberg group has zero Webster curvature and the pseudohermitian sub-Laplacian coincides with the Heisenberg Laplacian defined formerly, the \emph{CR Nirenberg}  problem, up to an inessential constant, is equivalent to finding a positive solution of \eqref{maineq1}.

We finally introduce the inversion map $\iota: \Hn \to  \Hn$ defined by
$$
\iota(\xi)=\iota(x, y, t):=(x,-y,-t)
$$
for every $\xi=(x, y, t) \in \Hn$, and the map $\varphi: \Hn \to  \Hn$ defined by Jerison and Lee in \cite{JL1987} which we shall refer to as the \emph{CR inversion} and which is defined by the following relations:
$$
\varphi(\xi):=\tilde{\xi},
$$
where $\tilde{\xi}=(\tilde{x}, \tilde{y}, \tilde{t})$ and
\be\label{CRinversion}
\tilde{x}:=\frac{x t+y|z|^2}{|\xi| ^4}, \quad \tilde{y}:=\frac{y t-x|z|^2}{|\xi| ^4}, \quad \tilde{t}:=\frac{-t}{|\xi| ^4} .
\ee
We explicitly remark that $|\varphi(\xi)| =\frac{1}{|\xi| }$.
Instead of using the CR inversion $\varphi$ defined in \eqref{CRinversion} as one of the generators of the group of CR maps on $\Hn$,  we will use the map $\check{\varphi}:=\varphi \circ \iota$ as in \cite{LD2012}, i.e., $\check{\varphi}(\xi)=(\check{x}, \check{y}, \check{t})$ for every $\xi \in \Hn$ with $(\check{x}, \check{y}, \check{t})$ being in turn defined by
\be\label{CRinversion2}
\check{x}=-\frac{x t+y|z|^2}{|\xi|^4}, \quad \check{y}=\frac{y t-x|z|^2}{|\xi|^4}, \quad \check{t}=\frac{t}{|\xi|^4} .
\ee
We make this choice because $\check{\varphi}(\check{\varphi}(\xi))=\xi$, while $\varphi(\varphi(\xi))=(-x,-y, t)$ for every $\xi=(x, y, t) \in\Hn\backslash\{0\}$.

\subsection{Pohozaev identity}
As in the Riemannian case, where the blow up analysis requires using the Pohozaev identity for $\Rn$, in the CR case we will need a Pohozaev formula for the Heisenberg group. Such roles of the Pohozaev type identity in analyzing the blow ups were first observed in Schoen\cite{Schoen43,Schoen44,Schoen45}.
 Pohozaev-type formulas have already been studied  on pseudohermitian manifold (see \cite{Afeltra2024}), Heisenberg group (see \cite{GL1992,PR2003})   and
  more general Carnot groups (see \cite{GV2000}).

We associate any  point $(z, t)=(x,y,t) \in \Hn$  with a $(2n+1)\times (2n+1)$ symmetric matrix $A =(a_{ij})$ defined by
$$
\left(\begin{array}{ccc}
I_n & 0_n & 2 y \\
0_n & I_n & -2 x \\
2 y & -2 x & 4|z|^2
\end{array}\right),
$$
where $I_n$ and $0_n$ denote respectively the identity matrix and the zero matrix in $\Rn$.
The matrix $A$ is related to  $\Delta_{\H}$ by the formula $\Delta_{\H} = \operatorname{div}(A \nabla)$, where $\nabla$ denotes the gradient in $\R^{2n+1}$.

Let $\Om \subset \Hn$ be an open set and $\mathcal{S}^2(\overline{\Om})$ denote the space of all continuous functions $u: \overline{\Om} \to \R$ such that $X_j u, Y_j u, X_j^2 u, Y_j^2 u$ are continuous functions in $\Om$ which can be extended to $\overline{\Om}$. Furthermore, let
\be\label{vectorfield2}
\mathcal{X}=\sum_{j=1}^n x_j \frac{\pa}{\pa x_j}+y_j \frac{\pa}{\pa y_j}+2 t \frac{\pa}{\pa t}
\ee
be  the generator for the one parameter family of dilations in $\Hn$ centered at the origin. Using this vector field, we can derive a Pohozaev type integral identity which is stated below.
\begin{lem}
Let $B_{\si}$ be a ball in $\Hn$
		centered at the origin with radius $\si>0$, $p\geq 1$ and $R\in \mathcal{S}^2(\overline{B}_{\si})$. Suppose that $u$ is a $C^2$ solution of $$-\Delta_{\H}u = R(\xi)|u|^{p-1}u\quad \text{ in }\, B_\si.$$
		Then  we have
		\begin{align}
		\int_{\pa {B_\si }}B(\si , \xi, u, \nabla_{\H}u)=& \Big( \frac{Q}{p+1} - \frac{Q-2}{2}\Big) \int_{B_\si }R|u|^{p+1}\notag\\&+ \frac{1}{p +1}\int_{B_\si } \mathcal{X}(R) |u|^{p+1}-\frac{1}{p +1}\int_{\pa  B_\si } R |u|^{p+1} \mathcal{X} \cdot \nu,\label{Pohozaevid}
		\end{align}
		where $\nu$  is the outward unit normal vector with respect to  $\pa B_{\si}$,	$\mathcal{X} \cdot \nu=\mathcal{X} \cdot \frac{\nabla d}{|\nabla d|}=\frac{\mathcal{X} d}{|\nabla d|}=\frac{d}{|\nabla d|}$ with $d$ being the distance function,   and  $$
		B(\si , \xi, u, \nabla_{\H}u) =
		\frac{Q-2}{2}(A \nabla u \cdot \nu) u-\frac{1}{2}|\nabla_{\H} u|^2 \mathcal{X} \cdot \nu +(A \nabla u \cdot \nu) \mathcal{X} (u).
		$$
	\end{lem}
\begin{proof}
The proof  can be found  in \cite[Theorem 2.1]{GL1992} (or \cite[Theorem 4.1]{PR2003}), so we omit it.
\end{proof}
The boundary term $B(\si , \xi, u, \nabla_{\H}u)$ has the following properties:
\begin{lem}\label{lem:positivepoho}
	\begin{itemize}
		\item[(i)]For $u(\xi) =|\xi|^{2-Q}$ and any $\si>0$, it holds
		$B(\si, \xi, u, \nabla_{\H} u)=0$ for  all  $\xi \in \pa B_\si$.
		\item[(ii)]		For $u(\xi) = |\xi|^{2-Q} + A + h(\xi)$, where $A>0$ is a  constant and $h(\xi)$ is differentiable near the origin with $h(0)=0$. Then
we have
		$$
		\lim _{\si  \to  0} \int_{\pa  B_\si } B(\si , \xi, u, \nabla_{\H} u)=-\frac{\sqrt{\pi}\Gamma(\frac{n+1}{2})}{2\Gamma(\frac{n}{2}+1)} A (Q-2)^2|\S ^{2 n-1}|<0,
		$$
		where $\Gamma$ is the Gamma function and $|\S ^{2 n-1}|$ is the surface measure of the unit sphere in $\R ^{2n}$.	
	\end{itemize}
\end{lem}
\begin{proof}
The proof can be found in  \cite[Proposition 4.3]{PR2003}, we omit it here.
\end{proof}

\subsection{Some a Priori estimates}\label{sec:2}
We intend to derive some a priori estimates for   solutions to  subcritical equations.  Our proofs are in the spirit of those in \cite{Li93d,Li1993} with some standard rescaling arguments.   We begin with introducing some notations.

Let $\Om\subset \Hn$ be an open set. We define the nonisotropic Sobolev space $S_k^p(\Om)$ (see   \cite{FS1974}) as follows:
For $1 \leq p \leq \infty$ and $k=0,1,2, \ldots$, we denote
$$
S_k^p(\Om)=\{f \in L^p(\Om): D f \in L^p(\Om) \,\text { for all } D \in \mathscr{A}_k\}.
$$
Here $Df$ is meant as a distribution derivative. $S_k^p(\Om)$ is a Banach space under the norm
$$
\|f\|_{S_k^p(\Om)}=\|f\|_{L^{p}(\Om)}+\sum_{D \in \mathscr{B}_k}\|D f\|_{L^{p}(\Om)}.
$$
We say $f\in S_k^p(loc)$  if   $\phi f\in S_k^p(\Hn)$ for every $\phi\in C_{c}^{\infty}(\Hn)$.
\begin{prop}\label{prop:2.1}
Suppose that $R\in L^{\infty}(\Hn \backslash B_1)$ and $\|R\|_{L^{\infty}(\Hn \backslash B_1)} \leq A_0$ for some constant $A_0>0$. Then there exist two positive constants $\mu_1 = \mu_1(n, A_0) $ and $C(n, A_0)$ such that for any  weak  solutions $u$ of
$$
-\Delta_{\H} u = R(\xi)|u|^{4/(Q-2)}u, \quad |\xi|\geq 1,
$$
satisfying  $u \in L^{Q^* }(\Hn \backslash B_1)$ and
\be \label{equ2.41}
  \int_{|\xi| \geq 1} |\nabla_{\H} u|^2 \le \mu_1,
\ee
we have
$$
\sup_{|\xi| \geq 2}|\xi|^{Q-2} |u(\xi)| \leq C(n, A_0) .
$$
\end{prop}
\begin{proof}
We perform a CR inversion  \eqref{CRinversion2} on $u(\xi)$. Let
$$
\hat{\xi} =\check{\varphi}(\xi)=(\check{x}, \check{y}, \check{t}) , \quad |\xi| \geq 1,\quad \text{ and }\quad
v(\hat{\xi}) = \frac{1}{|\hat{\xi}|^{Q-2}}u(\hat{\xi}).
$$
Using \cite[Corollary 2.8]{LD2012} we know that $v$ satisfies
$$
-\Delta_{\H} v(\hat{\xi}) = R(\hat{\xi})|v(\hat{\xi})|^{4/(Q-2)}v(\hat{\xi}), \quad 0< |\hat{\xi}|<1.
$$
Furthermore, it follows from \eqref{equ2.41} that
$$
\int_{|\hat{\xi}| \leq 1}|\nabla_{\H } v|^2 + \int_{|\hat{\xi}| \leq 1} |v|^{Q^* } \leq C_0(n) \mu_1.
$$
Thus, we duduce from \cite[Lemma 2.5]{CR2016} that  $v \in L^q_{loc}(B_{1})$ for any $q \geq n$, and then by the regularity results in \cite[Theorem 10.13]{FS1974} we have $v \in C^{\al}_{loc}(B_{1})$ for some $\al\in (0,1)$. To complete the proof of Proposition \ref{prop:2.1},  we  need to give  a priori bound of $\|v\|_{L^\infty(B_{0.5})}$. We claim that there exists a constant
$C(n,A_0)>0$ such that
\be\label{prop2.1-1}
\|v\|_{L^\infty(B_{0.5})}\leq C(n,A_0).
\ee
This will be done by contradiction argument.

Suppose the contrary of \eqref{prop2.1-1}, then there exist two sequences of $\{R_j(\xi)\}$, $\{u_j(\xi)\}$ satisfying
\begin{gather*}
  \| R_j \|_{L^{\infty}(\Hn \backslash B_1)} \geq A_0,\\
  -\Delta_{\H} u_j = R_j(\xi) |u_j|^{4/(Q-2)}u_j,\quad |\xi|\geq 1, \\
  \int_{|\hat{\xi}|\leq 1} |\nabla_{\H } v_j|^2 + \int_{|\hat{\xi}|\leq 1}|v_j|^{Q^* } \geq C_0(n) \mu_1,
  \end{gather*}
  but $$\|v_j\|_{L^\infty(B_{0.5})} \geq j,$$
where $v_j$ is obtained by CR inversion on $u_j$ as before.

Since $v_j\in  C^{\al}_{loc}(B_{1})$, thus we can choose $\hat{\xi}_j$ such that
$$
(0.9 - |\hat{\xi}_j|)^{(Q-2)/2}|v_j(\hat{\xi}_j)| = \max_{|\hat{\xi}| \leq 0.9} (0.9 - |\hat{\xi}|)^{(Q-2)/2}|v_j(\hat{\xi})|.
$$
Let $\si _j = \frac{1}{2}(0.9 - |\hat{\xi}_j|)>0$. Some standard calculations   in \cite{Li93d,Li1993} show that
\begin{gather*}
  |\hat{\xi}_j| \leq 0.9, \\
  (\si _j)^{(Q-2)/2} \max_{d (\hat{\xi} , \hat{\xi}_j)\le \si _j}|v_j(\hat{\xi})| \to  \infty\quad \text{ as }\, j\to \infty,\\
  |v_j(\hat{\xi}_j)| \geq 2^{(2-Q)/2} \max_{d (\hat{\xi} , \hat{\xi}_j)\le \si _j}|v_j(\hat{\xi})|.
\end{gather*}

Without loss of generality, we assume that $v_j(\hat{\xi}_j) >0$.
Let
$$
w_j(\widetilde{\xi}) = \frac{1}{v_j(\hat{\xi}_j)} v_j( \hat{\xi}_j \circ\delta_{v_j(\hat{\xi}_j)^{-2/(Q-2)}}(\widetilde{\xi}) ), \quad |\widetilde{\xi}| < v_j(\hat{\xi}_j)^{2/(Q-2)} \si _j \to  \infty.
$$
Clearly, $w_j$ satisfies
\begin{gather*}
  \int_{|\widetilde{\xi}_j| \leq v_j(\hat{\xi}_j)^{2/(Q-2)} \si _j} |\nabla_{\H } w_j|^2 + |w_j|^{Q^* } \leq C_0(n) \mu_1,\\
  -\Delta_{\H} w_j(\widetilde{\xi}) = R_j( \hat{\xi}_j \circ\delta_{v_j(\hat{\xi}_j)^{-2/(Q-2)}}(\widetilde{\xi}) )|w_j(\widetilde{\xi})|^{4/(Q-2)}  w_j(\widetilde{\xi}), \quad \forall \, |\widetilde{\xi}| < v_j(\hat{\xi}_j)^{2/(Q-2)} \si _j,    \\
  w_j(0) = 1,  \\
  w_j(\widetilde{\xi}) \leq 2^{(Q-2)/2}, \quad \forall \, |\widetilde{\xi}| < v_j(\hat{\xi}_j)^{2/(Q-2)} \si _j.
\end{gather*}
By \cite[Theorem 6.1]{F1975}, $w_j$ is bounded in $S_2^q(loc)$, $q>1$. Thus, modulo a subsequence, we have $w_j \rightharpoonup w$ in $S_2^q(loc)$ for some  function $w\in S_2^q(loc)$.
Moreover, $w$ satisfies
\begin{gather}w(0)=1,\notag\\
\label{equ2.17}
-\Delta_{\H} w = \bar{R}|w|^{4/(Q-2)} w \quad \text{ in }\, \Hn ,\\\int |\nabla_{\H } w|^2 + |w|^{Q^*} \leq C_0(n) \mu_1,\notag
\end{gather}
where $\bar{R}$ is the weak $*$ limit of $\{ R_j( \hat{\xi}_j\circ\delta_{v_j(\hat{\xi}_j)^{-2/(Q-2)}}(\widetilde{\xi}) ) \}$ in $L_{loc}^\infty(\Hn )$ satisfying  $\|\bar{R}\|_{L^\infty(\Hn)} \leq A_0$.
Multiplying \eqref{equ2.17} with $w$ and integrating by parts, we obtain
$$
\int |\nabla_{\H }w|^2 = \int \bar{R}|w|^{Q^* }
\leq A_0\Big( \int |\nabla_{\H } w|^2 \Big)^{Q/(Q-2)}(S_n)^{-Q^*},
$$
where $S_n$ is defined by \eqref{Sobolevineq}. Therefore,
$$
1 \leq A_0\Big( \int |\nabla_{\H } w|^2 \Big)^{2/(Q-2)}(S_n)^{-Q^*}\leq A_0(C_0(n) \mu_1)^{2/(n-2)}(S_n)^{-Q^*}.
$$
This is a contradiction if we choose $\mu_1 = \mu_1(n, A_0)$ small enough such that $$A_0(C_0(n) \mu_1)^{2/(n-2)}(S_n)^{-Q^*}<1.$$
We have proved the validity of \eqref{prop2.1-1} and thus complete the proof.
\end{proof}
We can deduce from  Proposition \ref{prop:2.1} the following result.

\begin{prop}\label{prop:2.2}
Let $\mu_1$ and $C(n, A_0)$ be the constants in Proposition \ref{prop:2.1}. Then for any $2< l_1 < l_2 < \infty$, there exists a  constant $S_1 = S_1(n, A_0, \mu_1, l_1, l_2)>l_2$ such that for any $R \in L^{\infty}(B_{S_{1}}\backslash B_1)$ with $\|R \|_{L^{\infty}(B_{S_1}\backslash B_1)} \leq A_0$ and any weak solutions $u$ of
$$
-\Delta_{\H}u = R(\xi)|u|^{4/(Q-2)}u, \quad 1<|\xi|< S_1,
$$
satisfying
$$
\int_{1<|\xi|<S_1} |\nabla_{\H } u|^2 + \int_{1<|\xi|<S_1} |u|^{Q^* } \leq\mu_1,
$$
we have
$$
\sup_{l_1 \leq |\xi| \leq l_2}|\xi|^{Q-2}|u(\xi)| \leq 2C(n, A_0).
$$
\end{prop}
\begin{proof}
Suppose the contrary, then  for $S_j = l_2 + j$, $j = 3, 4, 5 \ldots $, there exist two sequences of  $\{R_j\}$, $\{u_j\}$ satisfying
\begin{gather*} \|R_j\|_{L^{\infty}(B_{S_j}\backslash B_1)} \leq A_0,\\
  -\Delta_{\H}u_j = R_j(\xi)|u_j|^{4/(Q-2)}u_j,\quad 1<|\xi|< S_j, \\
  \int_{1<|\xi|<S_j} |\nabla_{\H } u_j|^2 + \int_{1<|\xi|<S_j} |u_j|^{Q^* } \leq\mu_1,
\end{gather*}
but
$$
\sup_{l_1\leq |\xi| \leq l_2}|\xi|^{Q-2}|u_j(\xi)|>2C(n, A_0).
$$

Arguing as in the proof of Proposition \ref{prop:2.1}, we know that for any $\mu \in (0,1)$, $\| u_j \|_{L^{\infty}(B_{S_j/2}\backslash B_{1+\mu})}$ is bounded by a constant independent of $j$. Let $u$ be the $S_2^q(loc)$ weak limit of $u_j$ and $\bar{R}(\xi)$ be the weak $*$ limit of $R_j(\xi)$ in $L^\infty(\Hn\backslash B_1)$, it holds
\begin{gather}
\| \bar{R} \|_{L^\infty(\Hn\backslash B_1)} \leq A_0, \notag\\-\Delta_{\H}u = \bar{R}(\xi)|u|^{4/Q-2}u, \quad |\xi|>1,\notag\\
\sup_{l_1\leq |\xi| \leq l_2}|\xi|^{Q-2}|u(\xi)|>2C(n, A_0).\label{prop2.2-1}
\end{gather}
We   immediately obtain a contradiction by \eqref{prop2.2-1} and Proposition \ref{prop:2.1}.
\end{proof}
Next we give a horizontal gradient estimate.
\begin{prop}\label{prop:2.3}
Suppose that $R\in L^{\infty}(B_{l_2}\backslash B_{l_1})$,  $l_2>100l_1>100$. Then for any  weak solutions $u$ of
$$
-\Delta_{\H} u = R(\xi)|u|^{4/Q-2}u, \quad l_1 \leq |\xi| \leq l_2,
$$
satisfying
$$
\sup_{l_1\leq |\xi| \leq l_2}|\xi|^{Q-2} |u(\xi)| \leq A
$$
for some constant $A>0$, we have
$$
|\nabla_{\H }u(\xi)| \leq \frac{C(n, A, \| R \|_{L^\infty(B_{l_2}\backslash B_{l_1})})}{|\xi|^{Q-1}},\quad 4l_1 \leq |\xi| \leq l_2/4.
$$
\end{prop}

\begin{proof}
For any $r\in (4l_1, l_2/4)$, it holds
$$
  -\Delta_{\H} u = R(\xi)|u|^{4/(Q-2)}u, \quad r/2 \leq |\xi| \leq 2r,
$$
and $$
\sup_{r/2\leq |\xi| \leq 2r}|u(\xi)| \leq \sup_{r/2\leq |\xi| \leq 2r}\frac{A}{|\xi|^{Q-2}}\leq \Big( \frac{2}{r} \Big)^{Q-2}A.$$
Let $v(\xi) = r^{Q-2}u(\delta_{r}(\xi))$, then $v$ satisfies
\begin{gather*}
-\Delta_{\H} v(\xi) = \frac{1}{r^2} R(\delta_{r}(\xi))|v(\xi)|^{4/(Q-2)}v,\quad  1/2\leq |\xi| \leq 2,\\\sup_{1/2\leq |\xi| \leq 2}|v(\xi)| \leq 2^{Q-2}A,\\\sup_{1/2\leq |\xi| \leq 2}|-\Delta_{\H} v(\xi)| \leq \| R \|_{L^\infty(B_{l_2}\backslash B_{l_1})}2^{Q+2} A^{(Q+2)/(Q-2)}.
\end{gather*}
Now we deduce from the regularity theories in \cite[Theorem 10.13]{FS1974} that
$$
|\nabla_{\H }v(\xi)| \leq C(n,A, \| R \|_{L^\infty(B_{l_2}\backslash B_{l_1})}), \quad |\xi| = 1.
$$
As a consequence,
$$
|\nabla_{\H }u(\xi)| \leq \frac{C(n,A, \| R \|_{L^\infty(B_{l_2}\backslash B_{l_1})})}{|\xi|^{Q-1}}, \quad |\xi| = r.
$$
This completes the proof.
\end{proof}
\begin{prop}\label{prop:2.4}
Let $\mu_1$, $S_1$ and $C(n,A_0)$ be the constants in Proposition \ref{prop:2.2}. Then for any $2 < l_1 < l_2 < \infty$,  there exist two positive constants $ \mu_2 = \mu_2(n, A_0) \leq \mu_1$ and $\overline{\tau} = \overline{\tau}(n ,A_0, l_1, l_2)$, such that for any $0 \leq \tau \leq \overline{\tau}$, $R \in L^\infty(B_{S_1} \backslash B_1)$ with $\| R \|_{L^{\infty}(B_{S_1} \backslash B_1)} \leq A_0$, and any weak  solutions $u$ of
$$
-\Delta_{\H} u = R(\xi)|u|^{4/(Q-2)-\tau}u, \quad 1< |\xi| <2S_1,
$$
satisfying
$$
\int_{1 < |\xi| < 2S_1}|\nabla_{\H }u|^2 + \int_{1 < |\xi| < 2S_1}|u|^{Q^* } \leq \mu_2,
$$
we have
$$
\sup_{l_1 \leq |\xi| \leq l_2}|\xi|^{Q-2} |u(\xi)| \leq 3C(n, A_0)
$$
and
$$
\sup_{l_1 \leq |\xi| \leq l_2}|\xi|^{Q-1} |\nabla_{\H} u(\xi)| \leq 2C(n, A_0,A),
$$where $C(n, A_0,A)$ is the constant in Proposition \ref{prop:2.3} with $A$ replaced by $3C(n, A_0)$.
\end{prop}
\begin{proof}The proof is similar to Proposition \ref{prop:2.2}, we omit it here.
\end{proof}

\subsection{A minimization problem}\label{sec:3}
For any $\xi_1, \xi_2 \in \Hn $ satisfy $d (\xi_1, \xi_2) \geq 10$, denote $\Om: = \Hn\backslash \{ B_1(\xi_1) \cup B_1(\xi_2)\}$. We define $E_\Om $ by taking the closure of $C_c^\infty(\overline{\Om })$  under the norm
$$
  \|u\|_{E_{\Om }}= \Big( \int_{\Om } |\nabla_{\H}u|^2 \Big)^{1/2} + \Big( \int_{\Om } |u|^{Q^*} \Big)^{1/Q^*}.
$$
Clearly, $E_\Om $ is a Banach space. Using similar arguments  in \cite[Poposition 3.2]{Li1993}, we know that
 $u \in E_{\Om }$ if and only if there exists $\bar{u} \in E$ such that $u = \bar{u}|_{\Om }$. Moreover, for any $u \in E_\Om $, we have a Sobolev type
 inequality on $E_\Om$:
\be \label{Sobolevineqdomain}
\Big(\int_{\Om }|u|^{Q^*}\Big)^{1/Q^*} \leq C(n)\Big( \int_{\Om }|\nabla_{\H }u|^2\Big)^{1/2},
\ee
where  the positive constant $C(n)$ depends only on $n$. In particular, it does not depend on $\xi_1, \xi_2$ provided $d (\xi_1, \xi_2) \geq 10$.

Let $R\in L^{\infty}(\Om)$ satisfy $\|R\|_{L^{\infty}(\Om )}\leq A_0$ for some constant $A_0 > 0$. We define a functional on $E_{\Om }$ by
$$
I_{R,\Om}(u) = \frac{1}{2}\int_{\Om } |\nabla_{\H }u|^2 - \frac{1}{Q^*-\tau} \int_{\Om }RH^{\tau}|u|^{Q^*-\tau},
$$
where $\tau \in [0,2/(Q-2)]$.  For any $u \in E_{\Om }$, using H\"older inequality and \eqref{Sobolevineqdomain}, we have
\be\label{16}
\Big| I_{R,\Om}(u) - \frac{1}{2}\int_{\Om } |\nabla_{\H }u|^2 \Big|
\leq A_0C_0(n) \Big( \int_{\Om } |\nabla_{\H }u|^2 \Big)^{(Q^*-\tau)/2}
\ee
with $C_0(n)$ being a positive constant depends only on $n$.

\begin{prop}\label{prop:3.1}
Let $E_{\Om}$ be defined as above. There exist two constants $r_0 =r_0(n, A_0)\in (0,1)$ and  $C_1 = C_1(n)>1$ such that for any $\xi_1, \xi_2 \in \Hn $ with $d(\xi_1, \xi_2) \geq 10$, and $\varphi \in H^{1/2}(\pa \Om)$ with $r= \| \varphi \|_{H^{1/2}(\pa \Om)} \leq r_0$, the following minimum problem is achieved:
\be\label{mini}
\min_{u\in E_{\Om}}\Big\{I_{R,\Om} (u) : u|_{\pa \Om} = \varphi, \int_{\Om } |\nabla_{\H }u|^2 \leq C_1r_0^2\Big\}.
\ee
The minimizer is unique (denoted $u_\varphi$) and satisfies $\int_{\Om } |\nabla_{\H }u_\varphi|^2 \leq C_1r^2/2$. Furthermore, the map $\varphi \mapsto u_\varphi$ is continuous from $H^{1/2}(\pa \Om)$ to $E_{\Om }$.
\end{prop}
\begin{rem}
Note that   $\pa \Om$ 	 is a smooth hypersurface of $\Hn$ with a finite number of non-degenerate characteristic points, one can give a meaning to $u|_{\pa\Om}$ and define the Sobolev space  $H^{1/2}(\pa \Om)$ by invoking the theory of traces,   see \cite{BCX2005} for more details.
\end{rem}
\begin{proof}[Proof of Proposition \ref{prop:3.1}]
According to the theory of traces in Bahouri-Chemin-Xu \cite[Theorem 1.8]{BCX2005} that there exist a constant $C_1 =C_1(n)>0$ and $\Phi \in E_{\Om }$ such that
\be\label{C1}
\int_{\Om } |\nabla_{\H }\Phi|^2 \leq \frac{C_1}{8}r^2\quad \text{ and }\quad \Phi|_{\pa \Om } = \varphi.
\ee
We fix the value of $C_1$ from now on and the value of $r_0$ will be determined in the following.

First it follows from \eqref{16}-\eqref{C1} that if $r_0(n, A_0)>0$ is chosen small enough, then
\be\label{upper}
I_{R,\Om}(\Phi) \leq \frac{1}{2}\int_{\Om }|\nabla_{\H }\Phi|^2 + A_0C_0(n) \Big( \int_{\Om }|\nabla_{\H }\Phi|^2 \Big)^{(Q^* -\tau)/2}\leq \frac{4}{5}\int_{\Om }|\nabla_{\H }\Phi|^2 \leq \frac{C_1}{10}r^2.
\ee
Since for any $u$ in $ \{u \in E_{\Om }: C_1r^2/2 \leq \int_{\Om } |\nabla_{\H }u|^2 \leq 2C_1r_0^2\}$, we derive from \eqref{16} that
$$I_{R,\Om} (u)\geq \Big(\frac{1}{2}-A_0 C_0(n)(2 C_1 r_0^2)^{2/(Q-2)-\tau / 2}\Big)\int_{\Om } |\nabla_{\H }u|^2.$$
Thus, if we choose $r_0>0$ to further satisfy  $A_0C_0(n)(2C_1 r_0^2)^{2/(Q-2)-\tau / 2} \leq 1/4$, then using \eqref{C1} and  \eqref{upper} we have $$I_{R,\Om} (u)\geq \frac{1}{4}\int_{\Om } |\nabla_{\H }u|^2\geq \frac{1}{4} \Big(\frac{1}{2}C_1r^2\Big) >I_{R,\Om}(\Phi).$$  Therefore, the minimizer is not achieved in the set $\{u \in E_{\Om }: C_1r^2/2 \leq \int_{\Om } |\nabla_{\H }u|^2 \leq 2C_1r_0^2\}$.

Next we prove the existence of the minimzer.
Write $u = v+ \Phi$, $v|_{\pa  \Om } = 0$, $J_{R,\Om}(v) =: I_{R,\Om}(u) = I_{R,\Om}(v + \Phi)$.
We only need to minimize $J_{R,\Om}(v)$ for $\int_{\Om } |\nabla_{\H }v|^2 \leq 2C_1 r_0^2$ due to the above argument. Obviously,  $J_{R,\Om}$ is strictly convex in the ball $\{ v \in E_{\Om }: v|_{\pa \Om } = 0, \int_{\Om } |\nabla_{\H }v|^2 \leq 2C_1r_0^2 \}$ if $r_0$ is small enough. Thus it is standard to conclude the existence of a unique local minimizer $v_\varphi$.

Finally, set $u= v_\varphi + \Phi$, then $u$ is a local minimizer and $u$ satisfies
$\int_{\Om } |\nabla_{\H }u_\varphi|^2 \leq C_1r^2/2$. The uniqueness and continuity of the map $\varphi \mapsto u_\varphi$ follows from the strict local convexity of $J_{R,\Om}$.
\end{proof}

\section{Construction of a family of approximate solutions}\label{sec:4}
Due to the presence of the  Sobolev critical exponent, the
Euler-Lagrange functional corresponding to \eqref{maineq1} does not satisfy the Palais-Smale condition. As previously mentioned in the introduction, we turn our attention to the following equation:
\be\label{subcritical}
-\Delta_{\H} u=R(\xi)H^{\tau} u^{(Q+2)/(Q-2)-\tau}, \quad u>0 \quad  \text{ in }\,\Hn,
\ee
which is the subcritical version of  \eqref{maineq} after using Green’s representation \eqref{Green} and the  Cayley transform \eqref{inverse}, where $\tau>0$ is a small constant and $H(z, t) =(\frac{4}{t^2 + (1+|z|^2)^2})^{(Q-2)/4}$.  In this section, we
will construct multi-bump solutions to the above subctitical type equations.

We first introduce some notations which are used throughout the paper.

Let $\{ R_l(\xi) \}$ be a sequence of functions  satisfying the following conditions.
\begin{enumerate}
  \item[(i)] There exists some  constant $A_1>0$ such that for any $l = 1, 2, 3,\ldots  ,$
  \be \label{18}
    |R_l(\xi)| \leq A_1, \quad \forall\, x \in \Hn .
  \ee
  \item[(ii)]
  For some integers $m \geq 2$, there exist $\xi_l^{(i)} \in \Hn $, $1 \leq i \leq m$, $S_l \leq \frac{1}{2}\min_{i \neq j}d (\xi_l^{(i)}, \xi_l^{(j)})$, such that $R_l$ is continuous near $\xi_l^{(i)}$ and
  \begin{align}\label{19}
    &\lim_{l \to  \infty} S_l = \infty, & \\\label{20}
    &R_l(\xi_l^{(i)}) = \max_{\xi \in B_{S_l}(\xi_l^{(i)})}R_l(\xi), \quad &1\leq i \leq m,\\\label{21}
    &\lim_{l \to  \infty} R_l(\xi_l^{(i)}) = a^{(i)}, &1\leq i \leq m,\\\label{22}
    & R_\infty^{(i)}(\xi) := (\text{weak $*$})\lim_{l \to  \infty} R_l (\xi_l^{(i)} \circ \xi), &1\leq i \leq m.
  \end{align}
  \item[(iii)] There exist some  constants $A_2, A_3>1$, $\delta_0, \delta_1 >0$, and some bounded open sets $O_l^{(1)}, \ldots , O_l^{(m)} \subset \Hn $, such that, if we define  for $1 \leq i \leq m$,
      \begin{gather*}
        \widetilde{O}_l^{(i)} = \{ \xi \in \Hn : \operatorname{dist}\,(\xi, O_l^{(i)}) < \delta_0 \}, \\
        O_l = \bigcup_{i=1}^m O_l^{(i)}, \quad \widetilde{O}_l = \bigcup_{i=1}^m \widetilde{O}_l^{(i)},
      \end{gather*}
      we have
      \begin{gather}\label{23-1}
        \xi_l^{(i)} \in O_l^{(i)}, \quad \operatorname{diam}(O_l^{(i)})<S_l/10, \\\label{23-2}
        R_l \in C^1( \widetilde{O}_l,[ 1/A_2, A_2 ] ), \\\label{24}
        R_l(\xi_l^{(i)}) \geq \max_{\xi \in \pa  O_l^{(i)}} R_l(\xi)+ c\delta_1,\\
        \max_{\xi \in \widetilde{O}_l}|\nabla_{\H}R_l(\xi)| \leq A_3, \label{25}
        \end{gather}
        where $c=c(\delta_0)>0$ is a constant such that $\operatorname{dist}(\xi_{l}^{(i)},\pa O_{l}^{(i)})\geq \delta_1/A_3$ for any $1\leq i\leq m$ and $\operatorname{diam} O := \sup\{ d(\xi, \zeta): \xi, \zeta \in O \}$ for any  set $O$ in $\Hn$.
    \end{enumerate}

For $\var>0$ small, we define $V_l(m, \var )= V(\var , O_l^{(1)}, \ldots , O_l^{(m)}, R_l)$.  In order to simplify our analysis, we only focus on the case $m=2$, as the more general result is similar in nature.


      If $u$ is a function in $V_l(2,\var)$, one can find an optimal representation, following the ideas introduced in \cite{BC1988,BC1991}. Namely, we have
      \begin{prop}\label{prop:4.1}
      There exists $\var _0  \in (0,1)$ depending only on $A_1, A_2, A_3, n, \delta_0$, but independent of $l$, such that for any $\var  \in (0, \var _0]$, $u \in V_l(2, \var )$, the following minimization problem
      \be \label{minimization}
        \min_{(\al  , \xi, \lam ) \in D_{4\var }}\Big\| u - \sum_{i=1}^{2} \al  _iw_{\xi_i,\lam _i} \Big\|
      \ee
      has a unique solution $(\al  , \xi, \lam )$ up to a permutation.  Moreover, the minimizer is achieved in $D_{2\var }$ for large $l$, where
        \begin{align*}
        D_{\var }=\big\{(\al  , \xi, \lam ) : &1/(2 A_2^{(Q-2) / 4})  \leq \al  _1, \al  _2 \leq 2 A_2^{(Q-2) / 4}, \\
        &\xi=  (\xi_1, \xi_2)\in O_l^{(1)} \times O_l^{(2)},  \lam =(\lam _1, \lam _2), \lam _1, \lam _2 \geq \var^{-1}\big\} .
        \end{align*}
In particular, we can write $u$ as $u = \sum_{i=1}^{2}\al_iw_{ \xi_i,\lam _i}+v$, where $v\in E$ and for each $i=1,2$, it holds\begin{gather*}
\langle w_{\xi_i,\lam _i}, v\rangle=\Big\langle\frac{\pa w_{\xi_i,\lam _i}}{\pa \lam_i}, v\Big\rangle=\langle X_j w_{\xi_i,\lam _i}, v\rangle=\langle Y_j w_{\xi_i,\lam _i}, v\rangle=\langle T w_{\xi_i,\lam _i}, v\rangle =0
\end{gather*}
 for all $j=1,\ldots,n$, where  $\langle\cdot, \cdot\rangle$ denotes the inner product defined by \eqref{inner} and $X_j,Y_j,T$ are the left invariant vector fields in \eqref{vectorfields}.
In addition, the variables $\{\al_i\}$ satisfy
\be\label{prop4.1-1}
|\al_i-R_l(\xi_i)^{(2 -Q) / 4}|=o_{\var}(1) \quad \text { for }\, i=1,2,
\ee
where $o_{\var}(1)\to 0$ as $\var\to 0$.
      \end{prop}
\begin{proof}
The proof  is similar  to  the corresponding statements in \cite{BC1988,BC1991},
we omit it here.
\end{proof}
\begin{rem}
	A. Bahri introduced the theory of critical points at infinity which is a set of ideas
	and techniques to handle noncompactness issues in nonlinear partial differential equations, we refer
	to  \cite{B1988book} for more explanations. This method is very powerful and has been applied to obtain
	so called Bahri–Coron-type existence criterium in various noncompactness problems, including the prescribed Webster Scalar Curvature problem
	on CR manifolds, see, e.g., \cite{CAY2010,GH2017,CAY2013,GHM2020,R2013,R2012,GAG2015,SG2011,RG2012-2,RG2012,G2017}. We will adopt these ideas in this section.
\end{rem}
In the sequel, we will often spilt $u$, a function in $V_l(2, \var)$, $\var\in (0,\var_0]$, under the form
      \be \label{181}
        u = \al  _1^lw_{ \xi_1^l,\lam _1^l} +\al  _2^lw_{\xi_2^l, \lam _2^l} + v^l
      \ee
after making the minimization \eqref{minimization}.     Proposition \ref{prop:4.1}  guarantees the existence and uniqueness of $\al_{i}=\al_{i}(u)=\al_{i}^{l}$, $\xi_{i}=\xi_{i}(u)=\xi_{i}^{l}$ and $\lam _{i}=\lam _{i}(u)=\lam _{i}^{l}$ for $i=1,2$ (we omit the index $l$ for simplicity).

For any $R\in L^{\infty}(\Hn)$ and $u\in E$, we define energy functional related to \eqref{subcritical}
$$I_{R,\tau}(u):=\frac{1}{2}\int |\nabla_{\H }u|^2-\frac{1}{Q^*-\tau}\int R H^{\tau}|u|^{Q^*-\tau}$$
with $\tau\geq 0$ small. Clearly, $I_{R}=I_{R,0}$.

Now we follow and refine the analysis of Bahri and Coron \cite{BC1988,BC1991} to study the subcritical interaction of two well-spaced bubbles.   To continue our proof,     let $\overline{\tau}_l>0$ be a sequence satisfying
\be \label{26}
\lim_{l \to \infty}\overline{\tau}_l = 0, \quad \lim_{l \to \infty}(|\xi_l^{(1)}|+ |\xi_l^{(2)}|)^{\overline{\tau}_l} =1.
\ee
We first give a lower bound energy estimate for some well-spaced \emph{bubbles}.
      \begin{lem}\label{lem:3.1}
Let $\var_0$ be the constant in Proposition \ref{prop:4.1}. Suppose that $\var_1\in (0,\var_0)$ small enough and $l$ large enough, $0\leq \tau \leq \overline{\tau}_l$.
Then there exists a constant $A_4 = A_4(n,\delta_1, A_2)>1$ such that for  any $u \in V_l(2, \var _1)$ with $\xi_1(u) \in \widetilde{O}_l^{(1)}$, $\xi_2(u) \in \widetilde{O}_l^{(2)}$, and $\operatorname{dist}(\xi_1(u), \pa  \widetilde{O}_l^{(1)})< \delta_1/(2A_3)$ or $\operatorname{dist}(\xi_2(u), \pa  \widetilde{O}_l^{(2)})< \delta_1/(2A_3)$, we have
   $I_{R_l, \tau}(u) \geq c^{(1)} + c^{(2)} + 1/A_4$,
      where $c^{(i)} =(a^{(i)})^{(2-Q)/2}(S_n)^Q/Q$ for $i=1,2$.
       \end{lem}
      \begin{proof}
       We assume that  $\operatorname{dist}(\xi_1(u), \pa  O_l^{(1)})< \delta_1/(2A_3)$. It follows from \eqref{Sobolevineq}, \eqref{prop4.1-1}, and some direct computations that, for $\var_1$ small and $l$ large,
          \begin{align*}
          I_{R_{l,\tau}}(u) =& \sum_{i=1}^{2}I_{R_{l,\tau}}(\al  _i w_{\xi_i,\lam _i}) +o_{\var _1}(1)\\
          =&\sum_{i=1}^{2}I_{R_{l,\tau}}(R_l(\xi_i)^{(2-Q)/4}   w_{\xi_i,\lam _i}) +o_{\var _1}(1)\\
          =& \sum_{i=1}^{2}\Big\{ \frac{1}{2}R_l(\xi_i)^{(2-Q)/2}\int|\nabla_{\H }w_{0,1}|^2 \\
          &-\frac{1}{Q^*}R_l(\xi_i)^{-Q/2}\int R_l w_{\xi_i,\lam _i}^{Q^* -\tau} \Big\} +o_{\var _1}(1) +o(1)\\
          \geq &\sum_{i=1}^{2}\Big\{ \frac{1}{2}R_l(\xi_i)^{(2-Q)/2}\int|\nabla_{\H }w_{0,1}|^2\\
          &-\frac{1}{Q^*}R_l(\xi_i)^{(2-Q)/2}\int w_{0,1}^{Q^* } \Big\} +o_{\var _1}(1) +o(1)\\
          = & \sum_{i=1}^2 \frac{1}{Q}R_l(\xi_i)^{(2-Q)/2}(S_n)^Q +o_{\var _1}(1) +o(1).
          \end{align*}
          Combining above estimate with the assumption $\operatorname{dist}(\xi_1(u), \pa  O_l^{(1)})< \delta_1/(2A_3)$, we obtain
         \begin{align*}
        I_{R_{l,\tau}}(u)  \geq & \frac{1}{Q}(R_l(\xi_l^{(1)})-\delta_1/2)^{(2-Q)/2}(S_n)^Q\\& + \frac{1}{Q}R_l(\xi_l^{(2)})^{(2-Q)/2}(S_n)^Q + o_{\var _1}(1) +o(1)\\
        \geq &\sum_{i=1}^{2} c^{(i)} + 1/A_4,
            \end{align*}
     where the choice of $A_4$ is evident thanks to \eqref{21}, \eqref{22}, \eqref{24} and \eqref{25}. The proof is now complete.
        \end{proof}

        From now on, the values of $A_4$ and $\var _1$ are fixed.  The main result in this section can be stated as follows:
        \begin{thm}\label{thm:4.1}
          Suppose that $\{R_l\}$ is a sequence of functions  satisfying (i)--(iii). If there exist some bounded open sets $O^{(1)}, \ldots , O^{(m)} \subset \Hn $ and some  constants $\delta_2, \delta_3 >0$, such that, for all $1 \leq i \leq m$,
      \begin{gather*}
     (\xi_l^{(i)})^{-1}\circ \widetilde{O}_l^{(i)} \subset O^{(i)} \quad \text{ for all } \,l,\label{2.13}\\
     \big\{ u: I_{R_\infty^{(i)}}'(u) = 0, u>0, u \in E, c^{(i)} \leq I_{R_\infty^{(i)}}(u) \leq c^{(i)} + \delta_2 \big\}\cap V(1, \delta_3, O^{(i)}, R_\infty^{(i)})= \emptyset.\label{29}
      \end{gather*}
          Then for any $\var >0$, there exists integer $\overline{l}_{\var , m}>0$, such that for all $l \geq  \overline{l}_{\var , m}$, $\tau \in (0 , \overline{\tau}_l)$, there exists $u_l = u_{l, \tau} \in V_l(m, \var )$ which solves
          \be \label{30}
            -\Delta_{\H}u_l = R_l (\xi) H^\tau u_l^{(Q+2)/(Q-2)-\tau} ,
           \quad  u_l>0 \quad \text{ in }\, \Hn.
          \ee
          Furthermore, $u_l$ satisfies
          \be \label{31}
            \sum_{i=1}^{m}c^{(i)}- \var  \leq I_{R_{l, \tau}}(u_l) \leq \sum_{i=1}^{m}c^{(i)} + \var .
          \ee
        \end{thm}
\begin{rem}
 \eqref{31}  follows from the  definition of $V_l(k,\var)$ provided that $u_l$ satisfies \eqref{30}.
	\end{rem}

        We prove Theorem \ref{thm:4.1} by contradiction argument. For simplicity, we only consider the case  $m=2$ since the changes for $m>2$ are evident.

         From now on, we suppose the contrary of Theorem \ref{thm:4.1}, namely,  for some $\var^*>0$, there exist a sequence of $l \to  \infty$ and $0<\tau_l <\overline{\tau}_l$, such that equation \eqref{30} for $\tau = \tau_l$ has no solution in $V_l(2, \var ^*)$ satisfying \eqref{31} with  $\var  = \var ^*$. Some complicated procedure will be followed to derive a contradiction. It will be outlined
         now and the details will be given in the next two subsections. The proof consists of two parts:
\begin{itemize}
	\item  \emph{Part 1.} Under the contrary of Theorem \ref{thm:4.1}, we obtain a uniform lower bound of the gradient
	vectors in some annular regions. It is a standard consequence of the Palais-Smale condition
	in variational argument
	\item \emph{Part 2.} Construct an approximating minimax curve via variational method. The result in Part 1 will be used to construct a deformation. In our setting, we will follow the nonnegative	gradient flow to make a deformation, which is an important process to derive a contradiction.	
\end{itemize}
Part 1 will be carried out in Subsection \ref{sec:4.1} and Part 2 in  Subsection \ref{sec:4.2}.

\subsection{First part of the proof of Theorem \ref{thm:4.1}}\label{sec:4.1}
  For $\var _2>0$, we denote $\widetilde{V}_l(2, \var _2)$ the set of functions $u$ in $E$ satisfies:  there exist $\al=(\al  _1, \al  _2) \in \R^2 $, $\xi = (\xi_1, \xi_2)\in O_l^{(1)} \times O_l^{(2)}$ and $\lam  = (\lam _1, \lam _2)\in \R^2$, such that
        \begin{gather*}\label{188}
          \lam _1, \lam _2> \var _2^{-1}, \\
          |\lam _i^{\tau_l}-1|< \var _2, \quad i = 1, 2, \\
          |\al  _i-R_l(\xi_i)^{(2-Q)/4}|< \var _2, \quad i=1, 2, \\
          \Big\| u-\sum_{i=1}^{2} \al  _i w_{\xi_i,\lam _i}^{1+O(\tau_l)}\Big\|<\var _2.
         \end{gather*}
         Throughout the paper, we denote $p_l=\frac{Q+2}{Q-2}-\tau_l$.
        \begin{lem}\label{lem4.2}
        For $\var _2= \var _2(n,\var _1, \var ^*)>0$ small enough, we have, for large $l$,
        \be \label{189}
          \widetilde{V}_l(2, \var _2) \subset V_l(2, o_{\var _2}(1)) \subset V_l(2, \var _1)\cap V_l(2, \var ^*),
        \ee
        where $o_{\var _2}(1)$ denotes some quantity which is independent of $l$ and tends to zero as $\var _2$ tends to zero.
         \end{lem}
        \begin{proof}
        It is straightforward to verify \eqref{189} by	using the definition of $\widetilde{V}_l(2, \var _2)$. Therefore, we we omit it here.
            \end{proof}
 Now we state the main result in this section, which reveals the uniform lower bounds of the gradient vectors in certain regions of $E$.
        \begin{prop}\label{prop:4.2}
         Under the hypotheses of Theorem \ref{thm:4.1} and the contrary of Theorem \ref{thm:4.1}, there exist two constants $\var _2 \in (0, \min \{ \var _0, \var _1, \var ^*, \delta_3 \})$ and $\var _3 \in (0, \min \{ \var _0, \var _1, \var _2, \var ^*, \delta_3 \})$, which are independent of $l$, such that \eqref{189} holds for such $\var _2$, and there exist $\delta_4=\delta_4(\var _2, \var _3)>0$ and  $l_{\var _2, \var _3}' >1$ such that for any $l \geq l_{\var _2, \var _3}'$, $u \in \widetilde{V}_l(2, \var _2)\backslash \widetilde{V}_l(2, \var _2/2)$ satisfying $|I_{R_l, \tau_l}(u) - (c^{(1)} + c^{(2)})|< \var _3$, we have
          $\| I_{R_l, \tau_l}'(u)\| \geq \delta_4$,
where $I_{R_l, \tau_l}'$ denotes Fr\'echet derivative.
        \end{prop}
    \begin{rem}
    	Proposition \ref{prop:4.2} will be used to construct an approximating minimaxing curve in Part 2.    Evidently we have, under the contrary of Theorem \ref{thm:4.1}, that for each $l$,
    $$
    	\inf\{ \| I_{R_l, \tau_l}'(u)\| : u \in \widetilde{V}_l(2, \var _2)\backslash \widetilde{V}_l(2, \var _2/2), I_{R_l, \tau_l}'(u)-(c^{(1)} + c^{(2)})< \var _3 \}>0.
    $$
    \end{rem}

           We prove Proposition \ref{prop:4.2} by contradiction argument. Suppose the statement in the Proposition \ref{prop:4.2} is not true, then no matter how small $\var _2, \var _3>0$ are, there exists a subsequence (still denoted as $\{u_l\}$) such that
          \begin{gather}\label{33}
            \{u_l\}\in \widetilde{V}_l(2, \var _2)\backslash \widetilde{V}_l(2, \var _2/2), \\\label{34}
            |I_{R_l, \tau_l}'(u_l)-(c^{(1)} + c^{(2)})| <\var _3, \\\label{35}
            \lim_{l\to  \infty } \| I_{R_l, \tau_l}'(u_l)\| =0.
          \end{gather}
        However, under the above assumptions, we can prove that there exists another subsequence, still denotes by $\{u_l\}$, such that $u_l\in \widetilde{V}_l(2, \var _2/2)$, which leads to a contradition. The existence of such
        sequence needs some lengthy indirect analysis to the interaction of two \emph{bubbles}. We break the proof of Proposition \ref{prop:4.2} into several claims.

     First we write
          \be \label{36}
            u_l= \al  _1^lw_{ \xi_1^l,\lam _1^l} +\al  _2^lw_{\xi_2^l, \lam _2^l} + v_l
          \ee  after making the minimization \eqref{minimization}.  By Proposition \ref{prop:4.1} and some standard arguments in \cite{B1988book,BC1988,BC1991}, if $\var _2>0$ small enough, we have
          \begin{gather}\label{37}
            (\lam _1^l)^{-1}, (\lam _2^l)^{-1} = o_{\var _2}(1), \\\label{38}
            |\al  _i^l -R_l(\xi_i^l)^{(2-Q)/4}| = o_{\var _2}(1), \\\label{39}
            \| v_l \| = o_{\var _2}(1), \\\label{40}
            \operatorname{dist}\,(\xi_1^l, O_l^{(1)}),~ \operatorname{dist}\,(\xi_2^l, O_l^{(2)}) = o_{\var _2}(1).
          \end{gather}

          Next we will derive some elementary estimates of the interaction of the \emph{bubbles} in \eqref{36} and find another representation of $u_l$ in \eqref{36}, from which we can deduce its location and concentrate rate. Let us introduce a linear isometry operator first.

          For any $\xi \in \Hn $, we define a linear isometry $T_{\xi}: E \to  E$ by
          $$
          (T_\xi u)(\cdot) = u(\xi \circ \cdot).
          $$It is easy to see $\|T_\xi u\|=\|u\|$.

          Now we give some estimates concerning with \emph{bubble}'s profile in \eqref{36}.
        \begin{claim}\label{claim:1}
 For $\var _2$ small enough, we have       $\lim_{l \to  \infty}\lam _1^l =\lim_{l \to  \infty}\lam _2^l =\infty$.
        \end{claim}
        \begin{proof}
        	Assume to the contrary that $\lam _1^l = \lam _1 + o(1)$ up to a subsequence. Here and in the following, we use $o(1)$  to denote any sequence tending to 0 as $l\to \infty$. Now the proof consists of three steps.

\textbf{Step 1} (Construct a positive solution). First, one observes from \eqref{36} that
   $$
          T_{\xi_1^l}u_l = \al  _1^lw_{0,\lam _1^l} + \al  _2^lw_{(\xi_1^l)^{-1} \circ \xi_2^l,\lam _2^l} +  T_{\xi_1^l}v_l.
 $$
  Then by Proposition \ref{prop:4.1},      by passing to a subsequence, we have
\be\label{3.10}
 \lim_{l \to  \infty}\al  _1^l = \al  _1 \in \Big[\frac{1}{2} (A_2)^{(2-Q)/4} - o_{\var _2}(1), 2(A_2)^{(2-Q)/4} + o_{\var _2}(1)\Big],
\ee
and $T_{\xi_1^l}v_l \rightharpoonup w_0$  weakly in $E$ for some $w_0\in E$.
           It follows from standard functional analysis arguments and \eqref{39} that
        \be \label{41}
          \|w_0\| \leq \liminf_{l \to  \infty}\| T_{\xi_1^l}v_l \| = o_{\var _2} (1).
        \ee
 Using the assumption (ii) (stated in the beginning of Section \ref{sec:4}), we get
        \be \label{42}
          \lim_{l \to  \infty} d (\xi_2^l, \xi_1^l) \geq  \lim_{l \to  \infty} S_l = \infty.
        \ee
   Therefore,
        \be \label{43}
          T_{\xi_1^l}u_l \rightharpoonup w:= \al  _1 w_{0,\lam _1} + w_0\quad \text{ weakly in $E$.}
        \ee
Obviously, $w \neq 0$ if $\var _2$ is small enough.

        Next we prove that $w$ is a weak solution of the following equation
        \be \label{44}
          -\Delta_{\H} w = T_\zeta R_\infty^{(1)}(\xi)|w|^{(Q+2)/(Q-2)}w \quad \text{ in } \,\Hn ,
        \ee
        where $\zeta \in O^{(1)}$ with  $\operatorname{dist}(\zeta, \pa  O^{(1)})> \delta_0/2$.

        For any $\phi \in C_c^\infty(\Hn )$, it follows from \eqref{35} that
      $$
          I_{R_l, \tau_l}'(u_l)(T_{(\xi_1^l)^{-1}}\phi) = o(1)\|
          T_{(\xi_1^l)^{-1}}\phi \| = o(1)\|\phi \| =o(1).
    $$
      Summing up \eqref{26}, \eqref{43}, \eqref{22} and \eqref{25}, we find that
        \begin{align*}
          o(1) &= \int \nabla_{\H} u_l \nabla_{\H} T_{(\xi_1^l)^{-1}}\phi-\int R_l H^{\tau_l}|u_l|^{p_l-1}u_lT_{(\xi_1^l)^{-1}}\phi\\
          &= \int \nabla_{\H}T_{\xi_1^l}u_l \nabla_{\H} \phi - \int T_{\xi_1^l} R_l(T_{\xi_1^l}H)^{\tau_l} |T_{\xi_1^l}u_l|^{p_l-1}T_{\xi_1^l}u_l\phi\\
          &=\int \nabla_{\H}w\nabla_{\H}\phi - \int T_\zeta R_\infty^{(1)}(\xi)|w|^{4/(Q-2)}w\phi + o(1),
          \end{align*}
        where $\zeta =\lim_{l \to  \infty}(\xi_l^{(1)})^{-1} \circ \xi_1^l$ along a subsequence. This means $w$ is a weak solution of \eqref{44}.

        The positivity of $w$ can be verified from the following argument.

        Let $w = w^+  - w^- $, where $w^+  = \max ( w, 0) $,  $w^- = \max ( -w, 0)$. It follows from \eqref{43} and \eqref{41} that
        \be\label{3.15}
        \int (w^-)^{Q^*}= o_{\var _2} (1).\ee  Multiplying \eqref{44} with $w^-$ and integrating by part, we have
      $$
          \int |\nabla_{\H }w^-|^2 \leq \int T_\zeta R_\infty^{(1)}(w^-)^{Q^*}
           \leq o_{\var _2}(1)\Big(\int(w^-)^{Q^*}\Big)^{2/Q^*} \leq o_{\var _2}(1)\int |\nabla_{\H }w^-|^2,
          $$
          where we used \eqref{3.15} in the second inequality and \eqref{Sobolevineq} in the last step.
         If $\var _2$ small enough, we immediately obtain $w^- \equiv 0$, namely, $w \geq 0$. It follows from \eqref{44} and strong maximum principle (see \cite{B1969}) that $w>0$.

   \textbf{Step 2} (Energy bound estimates).   Now we begin to estimate the value of $I_{T_\zeta R_\infty^{(1)}}(w)$ in order
      to produce a contradiction. The estimate we are going to establish is
        \be \label{45}
          c^{(1)} \leq I_{T_\zeta R_\infty^{(1)}}(w) \leq c^{(1)} + o_{\var _2}(1),
        \ee
        where $o_{\var _2}(1)\to  0$  as $\var _2 \to 0$.

        Firstly, we know from \eqref{44} that
          $\int |\nabla_{\H }w|^2 = \int T_\zeta R_\infty^{(1)}w^{Q^*}$.
Thus,
    $$
          I_{T_\zeta R_\infty^{(1)}}(w) = \frac{1}{2}\int |\nabla_{\H }w|^2 - \frac{1}{Q^*} \int T_\zeta R_\infty^{(1)}w^{Q^*} = \frac{1}{Q} \int |\nabla_{\H }w|^2.
   $$
        Then we conclude from \eqref{Sobolevineq}, \eqref{44}, and the fact $T_\zeta R_{\infty}^{(1)} \leq a^{(1)}$ that
  $$
          S_n \leq \frac{(\int |\nabla_{\H }w|^2)^{1/2}}{(\int w^{Q^*})^{1/Q^*}}
          \leq \frac{(\int |\nabla_{\H }w|^2)^{1/2}}{(\int T_\zeta R_\infty^{(1)}w^{Q^*})^{1/Q^*}}(a^{(1)})^{1/Q^*}= \Big(\int |\nabla_{\H }w|^2\Big)^{1/Q}(a^{(1)})^{1/Q^*}.
     $$
 Therefore, we complete the proof of the first inequality in \eqref{45}.

        On the other hand, we deduce from \eqref{18} that $|R_\infty^{(1)}(\xi)| \leq A_1$, $\forall\, \xi \in \Hn$.  Owing to  \eqref{Sobolevineq}, \eqref{36}, \eqref{37}, \eqref{39} and \eqref{42}, we have
          \begin{align*}
          I_{R_l, \tau_l}(u_l) &= I_{R_l, \tau_l}(\al  _1^lw_{ \xi_1^l,\lam _1^l}) + I_{R_{l, \tau_l}}(\al  _2^lw_{\xi_2^l, \lam _2^l}) + o_{\var _2}(1)\\
          &= I_{T_{\xi_1^l} R_l, \tau_l}(\al  _1^lw_{0,\lam _1^l}) + I_{R_l, \tau_l}(\al  _2^lw_{\xi_2^l, \lam _2^l}) + o_{\var _2}(1) + o(1)\\
          &= I_{T_{\xi_1^l} R_l, \tau_l}(\al  _1w_{ 0,\lam _1}) + I_{R_l, \tau_l}(\al  _2^lw_{\xi_2^l, \lam _2^l}) + o_{\var _2}(1) + o(1)\\
          &= I_{T_{\zeta} R_\infty^{(1)}}(\al  _1w_{0,\lam _1}) + I_{R_l, \tau_l}(\al  _2^lw_{\xi_2^l, \lam _2^l}) + o_{\var _2}(1) + o(1)\\
          &= I_{T_{\zeta} R_\infty^{(1)}}(w) + I_{R_l, \tau_l}(\al  _2^lw_{\xi_2^l, \lam _2^l}) + o_{\var _2}(1) + o(1),
          \end{align*}
        where $o(1)$ denotes some quantity which, for fixed $\var _2, \var _3$ goes to zero as $l \to  \infty$. Consequently,
        \be \label{47}
          I_{T_{\zeta} R_\infty^{(1)}}(w) = I_{R_l, \tau_l}(u_l) - I_{R_l, \tau_l}(\al  _2^lw_{\xi_2^l, \lam _2^l}) + o_{\var _2}(1) + o(1).
        \ee
 Combining \eqref{35} and \eqref{36}, we find
$$
          o(1) = I_{R_l, \tau_l}'(u_l)(\al  _2^lw_{\xi_2^l, \lam _2^l}) = I_{R_l, \tau_l}'(\al  _2^lw_{\xi_2^l, \lam _2^l})(\al  _2^lw_{\xi_2^l, \lam _2^l})  + o_{\var _2}(1) + o(1).
$$
        Namely,
        \begin{gather}
         \label{48}
        \int |\nabla_{\H}(\al  _2^lw_{\xi_2^l, \lam _2^l})| ^2= \int R_l H^{\tau_l} (\al  _2^lw_{\xi_2^l, \lam _2^l})^{Q^*-\tau_l} + o_{\var _2}(1) + o(1),\\\label{49}
        I_{R_l, \tau_l}(\al  _2^lw_{\xi_2^l, \lam _2^l}) = \frac{1}{Q}\int |\nabla_{\H}(\al  _2^lw_{\xi_2^l, \lam _2^l})| ^2 + o_{\var _2}(1) + o(1).
        \end{gather}
      From \eqref{Sobolevineq} and \eqref{3.10}, we obtain
        \be \label{50}
       \int |\nabla_{\H}(\al  _2^lw_{\xi_2^l, \lam _2^l})| ^2 \geq \frac{1}{2}\Big(A_2^{(2-Q)/4} - o_{\var _2}(1)\Big)(S_n)^Q > \frac{1}{4}(A_2)^{(2-Q)/4}(S_n)^Q>0.
        \ee
    Then, by \eqref{Sobolevineq}, \eqref{19}-\eqref{21}, \eqref{39}, \eqref{42}, and H\"older inequality, we have
        \begin{align*}
          S_n  &\leq \frac{(\int|\nabla_{\H }(\al  _2^lw_{\xi_2^l, \lam _2^l})|^2)^{1/2}}{(\int (\al  _2^lw_{\xi_2^l, \lam _2^l})^{Q^* })^{1/Q^{*}  }}\\
          &=\frac{(\int|\nabla_{\H }(\al  _2^lw_{\xi_2^l, \lam _2^l})|^2)^{1/2}}{(\int_{B_{S_l}(\xi_l^{(2)})} (\al  _2^lw_{\xi_2^l, \lam _2^l})^{Q^* })^{1/Q^{*}  }+ o(1)}\\
          &\leq\frac{(\int|\nabla_{\H }(\al  _2^lw_{\xi_2^l, \lam _2^l})|^2)^{1/2}}{(\int_{B_{S_l}(\xi_l^{(2)})} (\al  _2^lw_{\xi_2^l, \lam _2^l})^{Q^* -\tau_l})^{1/Q^{*}  }+ o(1)}\\
          &\leq\frac{(\int|\nabla_{\H }(\al  _2^lw_{\xi_2^l, \lam _2^l})|^2)^{1/2}R_l(\xi_l^{(2)})^{1/Q^{*}  }}{(\int_{B_{S_l}(\xi_l^{(2)})} R_l H^{\tau_l}(\al  _2^lw_{\xi_2^l, \lam _2^l})^{Q^* -\tau_l})^{1/Q^{*}  }+ o(1)}\\
          &=\frac{(\int|\nabla_{\H }(\al  _2^lw_{\xi_2^l, \lam _2^l})|^2)^{1/2}(a^{(2)})^{1/Q^{*}  } +o(1)}{(\int R_l H^{\tau_l}
          (\al  _2^lw_{\xi_2^l, \lam _2^l})^{Q^* -\tau_l})^{1/Q^{*}  }+ o(1)}.
          \end{align*}
Thus, using \eqref{48}, we establish that
 $$
          S_n \leq \Big(\int|\nabla_{\H }(\al  _2^lw_{\xi_2^l, \lam _2^l})|^2\Big)^{1/Q} (a^{(2)})^{1/Q^{*}  } +o(1).
$$
    This together with \eqref{49} gives
        \be \label{3.21}
          I_{R_l, \tau_l}(\al  _2^lw_{\xi_2^l, \lam _2^l}) \geq \frac{1}{Q}(a^{(2)})^{(2-Q)/2}(S_n)^Q + o_{\var _2}(1) + o(1) = c^{(2)} + o_{\var _2}(1) + o(1).
        \ee
        Putting \eqref{47}, \eqref{34} and above estimate together, we obtain the right hand side of \eqref{38}.

  \textbf{Step 3} (Completion of the proof).  Finally,     for $\var _2$ small enough, a contradiction arises from \eqref{43},  \eqref{44},  \eqref{45}, \eqref{29} and the positivity of $w$. This proves that $ \lim_{l \to  \infty}\lam _1^l =\infty$. Similarly $ \lim_{l \to  \infty}\lam _2^l =\infty$. Claim \ref{claim:1} has been established.
        \end{proof}

        For any $\lam >0$ any $\xi \in \Hn $, we define $\mathscr{T}_{l, \lam , \xi}: E \to  E$ by $$\mathscr{T}_{l, \lam , \xi}u(\cdot) = \lam ^{2/(1-p_l)}u(\xi \circ \delta_{\lam ^{-1}}(\cdot)).$$ It is clear that $$\mathscr{T}_{l, \lam , \xi}^{-1}u(\cdot)= \lam ^{2/(p_l-1)}u( \delta_{\lam }(\xi^{-1} \circ \cdot))$$
        and
        $$\int \nabla_{\H} u  \nabla_{\H}\mathscr{T}_{l, \lam , \xi}^{-1}\phi=\lam^{2(p_l +1)/(p_l-1) -Q} \int \nabla_{\H}  \mathscr{T}_{l, \lam , \xi}u  \nabla_{\H} \phi\quad \text{ for any }\, \phi\in C_{c}^{\infty}(\Hn).$$

        \begin{lem}\label{lem:4.3}
        There exists some constant $C= C(n,A_2)$, such that, for $\var _2$ small enough and $l$ large enough, we have $(\lam _1')^{\tau_l},  (\lam _2')^{\tau_l} \leq C$.
         \end{lem}
        \begin{proof} Applying \eqref{35}, we deduce that $I_{R_l, \tau_l}'(u_l)(w_{ \xi_1^l,\lam _1^l}) =o(1)$. Now an explicit calculation from
        	\eqref{39}, \eqref{42}, Claim \ref{claim:1}, \emph{bubbles}’ interaction estimates in \cite[Part 1]{B1988book}, and Proposition \ref{prop:4.1} yields that
          \be \label{228}
            (\al  _1^l)^{p_l} \int R_lH^{\tau_l}w_{ \xi_1^l,\lam _1^l}^{p_l +1} = \al  _1^l \int |\nabla_{\H }w_{ \xi_1^l,\lam _1^l}|^2 + o(1) +o_{\var _2}(1).
          \ee
        Then the proof of the first term completed from \eqref{228}, \eqref{23-2}, \eqref{38}, \eqref{26}, and Claim \ref{claim:1}. Similarly, we have $(\lam _2^l)^{\tau_l} \leq C$.
        \end{proof}

        Without loss of generality, we assume that
          \be \label{53}
            \lam _1^l \leq \lam _2^l.
          \ee

A direct computation using \eqref{36} shows that
          \be \label{54}
            \mathscr{T}_{l,  \xi_1^l,\lam _1^l}u_l = \widetilde{\al  }_1^l w_{0,1} +\widetilde{\al  }_2^l w_{ \delta_{\lam _1^l}((\xi_1^l)^{-1}\circ \xi_2^l),\lam _2^l/\lam _1^l} + \mathscr{T}_{l,  \xi_1^l,\lam _1^l}v_l,
          \ee
          where
       $$
            \widetilde{\al  }_1^l =\al  _1^l(\lam _1^l)^{(Q-2)/2 - 2/(p_l -1)},\quad
            \widetilde{\al  }_2^l =\al  _2^l(\lam _2^l)^{(Q-2)/2 - 2/(p_l -1)}.
    $$
 Then we can verify the existence of $u_1 \in E$ and $\zeta_1 \in O^{(1)}$ such that
          \begin{gather}\label{55}
            \mathscr{T}_{l,  \xi_1^l,\lam _1^l}u_l\rightharpoonup u_1 \quad \text{ weakly in $E$,} \\\label{56}
            \lim_{l \to  \infty}(\xi_l^{(1)})^{-1} \circ \xi_1^l =\zeta_1,
          \end{gather}
up to a subsequence.

          Accordingly, by making use of \eqref{22}, \eqref{25}, \eqref{40} and \eqref{56}, we have
          \be \label{57}
            \lim_{l \to  \infty} R_l(\xi_1^l)^{(2-Q)/4} = R_\infty^{(1)}(\zeta_1)^{(2-Q)/4}.
          \ee
          For any $\phi \in C_c^\infty(\Hn )$, it follows from \eqref{35} that
            \begin{align*}
            o(1) =& I_{R_l, \tau_l}'(u_l)(\mathscr{T}_{l,  \xi_1^l,\lam _1^l}^{-1}\phi)\\
            =&(\lam _1^l)^{2(p_l +1)/(p_l-1) -Q} \Big\{ \int \nabla_{\H}\mathscr{T}_{l,  \xi_1^l,\lam _1^l}u_l \nabla_{\H} \phi - \int T_{\xi_l^{(1)}}R_l((\xi_l^{(1)})^{-1}\circ \xi_1^l \circ \delta_{1/\lam _1^l} (\cdot))\\
            &\times H^{\tau_l}(\xi_1^l \circ \delta_{1/\lam _1^l}(\xi_1^l)) |\mathscr{T}_{l,  \xi_1^l,\lam _1^l}u_l|^{p_l-1}(\mathscr{T}_{l,  \xi_1^l,\lam _1^l}u_l) \phi \Big\}.
            \end{align*}
Taking the limit $l\to  \infty$, and then using \eqref{55}, \eqref{26}, \eqref{56}, \eqref{22}, \eqref{40}, and Lemma \ref{lem:4.3} we obtain
  $$
            \int \nabla_{\H}u_1 \nabla_{\H} \phi - \int R_\infty^{(1)}(\zeta_1)|u_1|^{4/(Q-2)}u_1 \phi =0.
     $$
    Namely, $u_1$ satisfies
          \be \label{23-17}
            -\Delta_{\H}u_1 = R_\infty^{(1)}(\zeta_1) |u_1|^{4/(Q-2)}u_1.
          \ee
Moreover, we see from \eqref{54} that  $u_1 \not\equiv 0$ if $\var _2$ is small enough.
          We then argue as before to obtain $u_1>0$.

       By the classification  theorem of positive solutions of \eqref{23-17} in $E$ (see \cite{JL1988}), there exist $\xi^* \in \Hn $ and  $\lam ^* >0$  such that
          \be \label{59}
            u_1 = R_\infty^{(1)}(\zeta_1)^{4/(Q-2)}w_{ \xi^*,\lam ^*}.
          \ee
          \begin{claim}\label{claim:2}
          For $l$ large enough, we have $|\xi^*| = o_{\var _2}(1)$, $|\lam ^*-1| = o_{\var _2}(1)$, $(\lam _1^l)^{\tau_l} = 1+o_{\var _2}(1)$.
          \end{claim}
          \begin{proof}
       First of all,   using Lemma \ref{lem:4.3}, we find $(\lam _1^l)^{\tau_l} = A_{\var _2, \var _3} +o(1)$ along a subsequence, where $A_{\var _2, \var _3}>0$ is a  constant independent of $l$ for fixed $\var_2$ and $\var_3$.  Thanks to \eqref{38} and \eqref{57},  we have
          \be \label{60}
            \al  _1^l = R_\infty^{(1)}(\zeta_1)^{(2-Q)/4}+ o_{\var _2}(1) + o(1).
          \ee
Note that
          $$
          \widetilde{\al  }_1^l = \al  _1^l(\lam _1^l)^{(Q-2)/2 - 2/(p_l -1)} = \al  _1^l(\lam _1^l)^{-(Q-2)^2\tau_l/8 + O(\tau_l^2)}.
          $$
                   Therefore,
          \be \label{61}
            \widetilde{\al  }_1^l = R_\infty^{(1)}(\zeta_1)^{(2-Q)/4}(A_{\var _2, \var _3})^{-(Q-2)^2/8} +o_{\var _2}(1) + o(1).
          \ee
          From \eqref{37}, \eqref{42} and \eqref{53}-\eqref{55}, we see that
          \be \label{62}
            \widetilde{\al  }_1^l w_{0,1} + \mathscr{T}_{l,  \xi_1^l,\lam _1^l}v_l \rightharpoonup u_1 \quad \text{ weakly in $E$.}
          \ee
             It follows from \eqref{39}, \eqref{61}, \eqref{62}, and Lemma \ref{lem:4.3} that
       $$
            \| R_\infty^{(1)}(\zeta_1)^{(2-Q)/4}(A_{\var _2, \var _3})^{-(Q-2)^2/8}w_{0,1} - R_\infty^{(1)}(\zeta_1)^{(2-Q)/4}w_{ \xi^*,\lam ^*} \| =o_{\var _2}(1) + o(1).
   $$
Finally, taking the limit $l\to \infty$, we get
 $|\xi^*| =o_{\var _2}(1)$,  $\lam ^* = 1+o_{\var _2}(1)$, $ A_{\var _2, \var _3} = 1+o_{\var _2}(1)$. Claim \ref{claim:2} has been established.
          \end{proof}

          We define $\phi_l \in E$ by
          \be \label{63}
            \mathscr{T}_{l,  \xi_1^l,\lam _1^l}u_l= u_1 +\mathscr{T}_{l,  \xi_1^l,\lam _1^l}\phi_l.
          \ee
      It follows from \eqref{55} that
          \be \label{63-1}
            \mathscr{T}_{l,  \xi_1^l,\lam _1^l}\phi_l \rightharpoonup 0 \quad \text{ weakly in $E$.}
          \ee
          \begin{claim}\label{claim:3}
         For $\var_2$ small enough,   we have $\| I_{R_l, \tau_l}'(\phi_l) \|=o(1)$.
          \end{claim}
          \begin{proof}
            For any $\phi \in C_c^\infty(\Hn )$, it follows from \eqref{35}, \eqref{23-17}, \eqref{63} and Lemma \ref{lem:4.3} that
          \begin{align}
              o(1)\|\phi \|=&I_{R_l, \tau_l}'(u_l)(\mathscr{T}_{l,  \xi_1^l,\lam _1^l}^{-1}\phi)\notag\\
              =&(\lam _1^l)^{2(p_l +1)/(p_l-1) -Q} \Big\{ \int \nabla_{\H}\mathscr{T}_{l,  \xi_1^l,\lam _1^l}u_l \nabla_{\H} \phi - \int R_l( \xi_1^l \circ \delta_{1/\lam _1^l} (\cdot))\notag\\
              &\times H^{\tau_l}(\xi_1^l \circ \delta_{1/\lam _1^l}(\xi_1^l)) |\mathscr{T}_{l,  \xi_1^l,\lam _1^l}u_l|^{p_l-1}(\mathscr{T}_{l,  \xi_1^l,\lam _1^l}u_l) \phi \Big\}\notag\\
              =&(\lam _1^l)^{2(p_l +1)/(p_l-1) -Q} \Big\{ \int \nabla_{\H}u_1 \nabla_{\H} \phi + \int \nabla_{\H}\mathscr{T}_{l,  \xi_1^l,\lam _1^l}\phi_l \nabla_{\H} \phi\notag\\
              &- \int R_l( \xi_1^l \circ \delta_{1/\lam _1^l} (\cdot))
              H^{\tau_l}(\xi_1^l \circ \delta_{1/\lam _1^l}(\xi_1^l)) |\mathscr{T}_{l,  \xi_1^l,\lam _1^l}u_l|^{p_l-1}(\mathscr{T}_{l,  \xi_1^l,\lam _1^l}u_l) \phi \Big\}\notag\\
              =&(\lam _1^l)^{2(p_l +1)/(p_l-1) -Q} \Big\{ \int R_{\infty}^{(1)}(\zeta_1) u_1^{(Q+2)/(Q-2)} \phi + \int \nabla_{\H}\mathscr{T}_{l,  \xi_1^l,\lam _1^l}\phi_l \nabla_{\H} \phi\notag\\
              &- \int R_l( \xi_1^l \circ \delta_{1/\lam _1^l} (\cdot))
              H^{\tau_l}(\xi_1^l \circ \delta_{1/\lam _1^l}(\xi_1^l)) |\mathscr{T}_{l,  \xi_1^l,\lam _1^l}\phi_l|^{p_l-1}(\mathscr{T}_{l,  \xi_1^l,\lam _1^l}\phi_l) \phi \notag\\
              &+\int R_l( \xi_1^l \circ \delta_{1/\lam _1^l} (\cdot))
              H^{\tau_l}(\xi_1^l \circ \delta_{1/\lam _1^l}(\xi_1^l)) |\mathscr{T}_{l,  \xi_1^l,\lam _1^l}\phi_l|^{p_l-1}(\mathscr{T}_{l,  \xi_1^l,\lam _1^l}\phi_l) \phi\notag \\
              &- \int R_l( \xi_1^l \circ \delta_{1/\lam _1^l} (\cdot))
              H^{\tau_l}(\xi_1^l \circ \delta_{1/\lam _1^l}(\xi_1^l)) |\mathscr{T}_{l,  \xi_1^l,\lam _1^l}u_l|^{p_l-1}(\mathscr{T}_{l,  \xi_1^l,\lam _1^l}u_l) \phi \Big\}\notag\\
              =&I_{R_l, \tau_l}'(\phi_l)(\mathscr{T}_{l,  \xi_1^l,\lam _1^l}^{-1}\phi)+(\lam _1^l)^{2(p_l +1)/(p_l-1) -Q} \Big\{ \int R_{\infty}^{(1)}(\zeta_1) u_1^{(Q+2)/(Q-2)} \phi\notag\\
              &+\int R_l( \xi_1^l \circ \delta_{1/\lam _1^l} (\cdot))
              H^{\tau_l}(\xi_1^l \circ \delta_{1/\lam _1^l}(\xi_1^l))|\mathscr{T}_{l,  \xi_1^l,\lam _1^l}\phi_l|^{p_l-1}(\mathscr{T}_{l,  \xi_1^l,\lam _1^l}\phi_l) \phi\notag \\
              &- \int R_l( \xi_1^l \circ \delta_{1/\lam _1^l} (\cdot))
              H^{\tau_l}(\xi_1^l \circ \delta_{1/\lam _1^l}(\xi_1^l)) |\mathscr{T}_{l,  \xi_1^l,\lam _1^l}u_l|^{p_l-1}(\mathscr{T}_{l,  \xi_1^l,\lam _1^l}u_l) \phi \Big\}.\label{64}
            \end{align}
Then a direct calculation exploiting \eqref{26}, \eqref{57}, \eqref{59}, Claim \ref{claim:2}, H\"older inequalities and  Sobolev embedding theorems that
            \be \label{65}
              \Big| \int R_l( \xi_1^l \circ \delta_{1/\lam _1^l} (\cdot))
              H^{\tau_l}(\xi_1^l \circ \delta_{1/\lam _1^l}(\xi_1^l)) (u_1)^{p_l}\phi-\int R_{\infty}^{(1)}(\zeta_1) u_1^{(Q+2)/(Q-2)} \phi \Big| =o(1)\| \phi \|.
            \ee
  Finally, by \eqref{64}, \eqref{65}, Lemma \ref{lem:4.3} and some elementary inequalities, we deduce that
     $$
              |I_{R_l, \tau_l}'(\phi_l)(\mathscr{T}_{l,  \xi_1^l,\lam _1^l}^{-1}\phi)|
              = o(1)\| \phi \| + O(1)\int (|\mathscr{T}_{l,  \xi_1^l,\lam _1^l}\phi_l|^{p_l-1}u_1 + \mathscr{T}_{l,  \xi_1^l,\lam _1^l}\phi_l|u_1|^{p_l-1})|\phi|
           =o(1)\| \phi \|,
     $$
    where      the last inequaity follows from \eqref{63}, \eqref{59}, Claim \ref{claim:2}, H\"older inequalities and the Sobolev embedding theorems. Claim \ref{claim:3} has been established now.
          \end{proof}
          \begin{claim}\label{claim:4}
            $I_{R_l, \tau_l}(\phi_l) \leq c^{(2)} + \var _3 + o(1)$.
          \end{claim}
          \begin{proof}
         By a change of variable and using Claim \ref{claim:2}, \eqref{63}, \eqref{63-1} and  \eqref{59},  some calculations lead
         to
            \begin{align}
              I_{R_l, \tau_l}(u_l)=&I_{R_l, \tau_l}(\phi_l)+ (\lam _1^l)^{2(p_l +1)/(p_l-1) -Q} \Big\{ \frac{1}{2}\int |\nabla_{\H }u_1|^2 \notag\\
              &- \frac{1}{Q^*} \int R_l( \xi_1^l \circ \delta_{1/\lam _1^l} (\cdot))
              H^{\tau_l}(\xi_1^l \circ \delta_{1/\lam _1^l}(\cdot))|u_1|^{p_l+1}  \Big\}+ o(1).\label{66}
              \end{align}
                   We derive from \eqref{56}, \eqref{22} and \eqref{Sobolevineq} that
                \begin{align}
                &\frac{1}{2}\int |\nabla_{\H }u_1|^2
               - \frac{1}{Q^*} \int R_l( \xi_1^l \circ \delta_{1/\lam _1^l} (\cdot))
               H^{\tau_l}(\xi_1^l \circ \delta_{1/\lam _1^l}(\cdot)) |u_1|^{p_l+1}\notag\\
               =&I_{R_\infty^{(1)}(\zeta_1)}(u_1)+o(1)\notag\\
               \geq& \frac{1}{Q}R_\infty^{(1)}(\zeta_1)^{(2-Q)/2}(S_n)^Q + o(1)\notag \\\geq& c^{(1)} + o(1).\label{67}
                \end{align}
            Claim \ref{claim:4} follows from \eqref{66},  \eqref{67}, \eqref{34}, and the fact  $(\lam _1^l)^{2(p_l +1)/(p_l-1) -Q} \geq 1$.
          \end{proof}

          From \eqref{63}, \eqref{36} and \eqref{59} we have
          \be \label{3.39}
            \phi_l = u_l -\mathscr{T}_{l,  \xi_1^l,\lam _1^l}^{-1}u_1 = \al  _2^lw_{\xi_2^l, \lam _2^l} + w_l,
          \ee
          where
    $$
            w_l = \al  _1^lw_{ \xi_1^l,\lam _1^l} - (\lam _1^l)^{2/(p_l-1)-(Q-2)/2}w_{\xi _1^l \circ \delta_{1/\lam _1^l}(\xi^*),\lam ^* \lam_1^l,} + v_l.
    $$
          Using Claim \ref{claim:2} and \eqref{60}, we have, for large $l$, that
          \be \label{70}
            \| w_l \|= o_{\var _2}(1).
          \ee

          Now we repeat the previous arguments on $\phi_l$ instead of $u_l$. For simplicity, we only carry out some crucial steps and omit similar proofs.

          Using \eqref{3.39} we have
          \be \label{72}
            \mathscr{T}_{l, \xi_2^l, \lam _2^l}\phi_l= \overline{\al  }_2^l w_{0,1} + \mathscr{T}_{l, \xi_2^l, \lam _2^l}w_l,
          \ee
          where
          \be \label{73}
            \overline{\al  }_2^l = \al  _2^l(\lam _2^l)^{(Q-2)/2-2/(p_l-1)}.
          \ee
Then we can verify the existence of  $u_2 \in E$ and $\zeta_2 \in O^{(2)}$ such that
          \begin{gather}\label{74}
            \mathscr{T}_{l, \xi_2^l, \lam _2^l}\phi_l \rightharpoonup u_2 \quad \text{ weakly in $E$,} \\\label{75}
            \lim_{l \to  \infty}(\xi_l^{(2)})^{-1} \circ \xi_2^l =\zeta_2,
          \end{gather}
up to a subsequence.

Accordingly, by making use of \eqref{22}, \eqref{25} and \eqref{75}, we have
          \be \label{76}
            \lim_{l \to  \infty}R_l(\xi_2^l)^{(2-Q)/4} = R_l(\zeta_2)^{(2-Q)/4}.
          \ee

          For any  $\phi \in C_c^\infty(\Hn )$, it follows from Claim \ref{claim:3} and Lemma \ref{lem:4.3} that
           \begin{align*}
            o(1) =& I_{R_l, \tau_l}'(\phi_l)(\mathscr{T}_{l, \xi_2^l, \lam _2^l}^{-1}\phi)\notag\\
            =&(\lam _1^l)^{2(p_l +1)/(p_l-1) -Q} \Big\{ \int \nabla_{\H}\mathscr{T}_{l, \xi_2^l, \lam _2^l}u_l \nabla_{\H} \phi - \int T_{\xi_l^{(2)}}R_l((\xi_l^{(2)})^{-1}\circ \xi_2^l \circ \delta_{1/\lam _2^l} (\cdot))\notag\\
            &\times H^{\tau_l}(\xi_2^l \circ \delta_{1/\lam _2^l}(\xi_2^l)) |\mathscr{T}_{l, \xi_2^l, \lam _2^l}u_l|^{p_l-1}(\mathscr{T}_{l, \xi_2^l, \lam _2^l}u_l) \phi \Big\}.
            \end{align*}
    Taking the limit $l \to  \infty$ and arguing as before, we have
$$
            \int \nabla_{\H}u_2 \nabla_{\H} \phi - \int R_\infty^{(1)}(\zeta_2)|u_2|^{4/(Q-2)}u_2 \phi =0.
$$
Namely, $u_2$ satisfies
          \be \label{77}
            -\Delta_{\H}u_2 = R_\infty^{(2)}(\zeta_2) |u_2|^{4/(Q-2)}u_2.
          \ee

          Arguing as before, for $\var _2$ small enough we can prove that $u_2>0$ and for some $\xi^{**} \in \Hn $ and $\lam ^{**} >0$,
          \be \label{78}
            u_2 = R_\infty^{(2)}(\zeta_2)^{(2-Q)/4}w_{\xi^{**},\lam ^{**}}.
          \ee
          \begin{claim}\label{claim:5}
            For $l$ large enough, we have $|\xi^{**}| = o_{\var _2}(1)$, $|\lam ^{**}-1|= o_{\var _2}(1)$, $(\lam _2^l)^{\tau_l} = 1+o_{\var _2}(1)$.
          \end{claim}
          \begin{proof}
            The proof is similar to  the proof of Claim \ref{claim:2},  we  omit it here.
          \end{proof}

          We define $\eta_l \in E$ by
          \be \label{257}
            \mathscr{T}_{l, \xi_2^l, \lam _2^l}\phi_l = u_2 +\mathscr{T}_{l, \xi_2^l, \lam _2^l}\eta_l.
          \ee
        Clearly,
     \be\label{3.49}
            \mathscr{T}_{l, \xi_2^l, \lam _2^l}\eta_l\rightharpoonup 0 \quad \text{ weakly  in $E$.}
 \ee
          \begin{claim}\label{claim:6}
     For $\var _2$ small enough,  we have       $\| I_{R_l, \tau_l}'(\eta_l) \|=o(1)$.
          \end{claim}
          \begin{proof}
            The proof is similar to the proof of Claim \ref{claim:3}, we  omit it here.
          \end{proof}
          \begin{claim}\label{claim:7}
            $I_{R_l, \tau_l}(\eta_l) \leq \var _3 + o(1)$.
          \end{claim}
          \begin{proof}
             The proof is similar to the proof of Claim \ref{claim:4},  we  omit it here.
          \end{proof}

          \begin{claim}\label{claim:8}
           For $\var _2$ small enough,  we have $\eta_l \to  0$ strongly in $E$.
          \end{claim}
          \begin{proof}
 The proof makes use of  contradiction argument  and  Claims \ref{claim:6} and \ref{claim:7},  we  omit the details here.
          \end{proof}
          Rewriting \eqref{63} and \eqref{257}, we have
          \be \label{93}
            u_l = \mathscr{T}_{l,  \xi_1^l,\lam _1^l}^{-1}u_1 + \mathscr{T}_{l,  \xi_1^l,\lam _1^l}^{-1}u_2 + \eta_l.
          \ee

          \begin{claim}\label{claim:9}
            For $\var _2$ small enough, we have
             $ (\lam _i^l)^{\tau_l} = 1+o_{\var _3}(1) +o(1)$ for $i=1,2$.
          \end{claim}
          \begin{proof}
        We first deduce from \eqref{66}, \eqref{67}, and Lemma \ref{lem:4.3} that
            \be \label{94}
              I_{R_l, \tau_l}(u_l) \geq I_{R_l, \tau_l}(\phi_l)+ (\lam _1^l)^{2(p_l +1)/(p_l-1) -Q} c^{(1)} + o(1).
            \ee
   In view of Claim \ref{claim:5}, \eqref{78}-\eqref{3.49}, some calculations similar to the proof of Claim \ref{claim:4} lead to
   \begin{align*}
   I_{K_l, \tau_l}(\phi_l)= & I_{K_l, \tau_l}(\eta_l)+(\lam_2^l)^{2(p_l+1) /(p_l-1)-Q}\Big\{\frac{1}{2}\int |\nabla_{\H }u_2|^2-\frac{1}{Q^*} \int R_l( \xi_2^l \circ \delta_{1/\lam _2^l} (\cdot)) \\
   & H^{\tau_l}(\xi_2^l \circ \delta_{1/\lam _2^l}(\cdot)) |u_2|^{p_l+1}\Big\}+o(1) .
   \end{align*}
  Similar to the calculation in \eqref{67}, we derive from \eqref{22}, \eqref{75} and \eqref{Sobolevineq} that
   $$
   \frac{1}{2}\int |\nabla_{\H }u_2|^2 - \frac{1}{Q^*} \int R_l( \xi_2^l \circ \delta_{1/\lam _2^l} (\cdot))
   H^{\tau_l}(\xi_2^l \circ \delta_{1/\lam _2^l}(\cdot)) |u_2|^{p_l+1} \geq c^{(2)}+o(1) .
   $$
         Combining the above estimates with  Lemma \ref{lem:4.3}  we have
            \be \label{95}
              I_{R_l, \tau_l}(\phi_l) \geq I_{R_l, \tau_l}(\eta_l)+ (\lam _2^l)^{2(p_l +1)/(p_l-1) -Q} c^{(2)} + o(1).
            \ee
             Then we use Claim \ref{claim:8} to deduce that
            \be \label{96}
              I_{R_l, \tau_l}(\eta_l) = o(1).
            \ee
            Finally, we put together \eqref{34}, \eqref{94}-\eqref{96} to obtain
        $$
              \sum_{i=1}^{2}\{ (\lam _i^l)^{2(p_l +1)/(p_l-1) -Q}-1 \}c^{(i)} \leq \var _3 +o(1).
     $$
       This completes the proof of Claim \ref{claim:9}.
          \end{proof}
          \begin{claim}\label{claim:10}
            Let $\delta_5 =\delta_1/(2A_3)>0$. Then if $\var _2$ is chosen to be small enough, we have, for large $l$, that $
              \operatorname{dist}\,(\xi_i^l, \pa  O_l^{(i)}) \geq \delta_5$ for $i=1,2$.
          \end{claim}
          \begin{proof}
            The proof is similar to the proof of Lemma \ref{lem:3.1}, we  omit it here.
          \end{proof}
  Now we are in the position to prove Proposition \ref{prop:4.2}.
  \begin{proof}[Proof of Proposition \ref{prop:4.2}]
  Applying \eqref{22}, \eqref{25}, \eqref{56}, \eqref{59}, and Claim \ref{claim:9}, we deduce that
  	\begin{align*}
  	\mathscr{T}_{l, \xi_2^l, \lam _2^l}^{-1}u_1 &=(\lam _1^l)^{2/(p_l-1)}u_1( \delta_{\lam _1^l}((\xi_1^l)^{-1} \circ \cdot))\\
  	&=(\lam _1^l)^{2/(p_l-1)-(Q-2)/2}R_\infty^{(1)}(\zeta_1)^{(2-Q)/4}w_{\xi _1^l \circ \delta_{1/\lam _1^l}(\xi^*),\lam ^* \lam_1^l}\\
  	&=R_\infty^{(1)}(\zeta_1)^{(2-Q)/4}w_{\xi _1^l \circ \delta_{1/\lam _1^l}(\xi^*),\lam ^* \lam_1^l,}+ o_{\var _3}(1)\\
  	&=R_l\big(\xi_1^l \circ \delta_{1/\lam _1^l}(\xi^*)\big)^{(2-Q)/4}w_{\xi _1^l \circ \delta_{1/\lam _1^l}(\xi^*),\lam ^* \lam_1^l}+ o_{\var _3}(1) +o(1).
  	\end{align*}
  	Similarly, we have
$$
  	\mathscr{T}_{l, \xi_2^l, \lam _2^l}^{-1}u_2 =R_l\big(\xi_2^l \circ \delta_{1/\lam _2^l}(\xi^{**})\big)^{(2-Q)/4}w_{ \xi _2^l \circ \delta_{1/\lam _2^l}(\xi^{**}),\lam ^{**}\lam_2^l,}+ o_{\var _3}(1)+o(1).
 $$
  	Therefore, we can rewrite \eqref{93} as (see Claim \ref{claim:8} and the above)
  	\begin{align}
  	u_l =&R_l\big(\xi_1^l \circ \delta_{1/\lam _1^l}(\xi^*)\big)^{(2-Q)/4}w_{\xi_1^l \circ \delta_{1/\lam _1^l}(\xi^*),\lam ^* \lam_1^l,}\notag\\
  	&+R_l\big(\xi_2^l \circ \delta_{1/\lam _2^l}(\xi^{**})\big)^{(2-Q)/4}w_{ \xi_2^l \circ \delta_{1/\lam _2^l}(\xi^{**}), \lam ^{**}\lam_2^l}+ o_{\var _3}(1)+o(1).\label{101}
  	\end{align}
  	 	We now fix the value of $\var _2$  small enough to make all the previous arguments hold and then make $\var _3$ small (depending on $\var_2$)  such that (using Claim \ref{claim:9}):
  	\be \label{102}
  	|(\lam ^*\lam _i^l)^{\tau_l}-1|= o_{\var _3}(1) +o(1) < \var _2/2 \quad \text{ for }\, i=1,2.
  	\ee
  	From \eqref{101}, \eqref{102}, Claims \ref{claim:1} and  \ref{claim:9}, we see that for $\var _3$ small, we have,  for large $l$, $u_l \in \widetilde{V}_l(2, \var _2/2)$.
  	This contradicts to  \eqref{33}. We conclude the proof of Proposition \ref{prop:4.2}.
  \end{proof}

  \subsection{Complete the proof of Theorem \ref{thm:4.1}}\label{sec:4.2}
   In this subsection we will complete the proof of Theorem \ref{thm:4.1}.      Precisely, under the contrary of Theorem \ref{thm:4.1} and combining with the Proposition \ref{prop:4.2} established in the previous  subsection, a contradiction will be produced by adopting and modifying the minimax procedure as in \cite{CR1,CR2,Li1993,CES,Se}.
To reduce overlaps, we will omit the proofs of several intermediate results which closely follow standard arguments, giving appropriate references.
We start the proof by defining a certain family  of sets and minimax values and giving some notations.

  Let $\Om$ be a smooth bounded domain in $\Hn$.       Define the space $S_0^1(\Om)$ by taking the closure of $C_c^{\infty}(\Om )$ under the norm 
      $$
            \| u\|_{S_0^1(\Om )}= \Big(\int_{\Om }|\nabla_{\H}u|^2\Big)^{1/2}+ \Big(\int_{\Om }|u|^2\Big)^{1/2}.
$$
By means of \eqref{Sobolevineq}, this norm is equivalent to the norm generated by the inner product
$\langle u,v\rangle_{S_0^1(\Om )}=\int_{\Om}\nabla_{\H } u \nabla_{\H } v$.
 Using the invariance under translations and dilations, it is easy to see that $S_n$ is also the best
Sobolev constant for the embedding $S_0^1(\Om)\hookrightarrow L^{Q^*}(\Om)$ and is not achieved, see \cite{JL1988}.

          In the following part of this section,  we write $\tau_l=\tau$, $p_l= p$.

Now, for each $i = 1, 2$, we define
            \begin{align*}
            \gamma_{l, \tau}^{(i)} =&\{ g^{(i)} \in C([0,1], S_0^1(B_{S_l}(\xi_l^{(i)}))): g^{(i)}(0) = 0, I_{R_l, \tau}(g^{(i)}(1))<0\},\\
            c_{l, \tau}^{(i)} =& \inf_{g^{(i)} \in \gamma_{l, \tau}^{(i)}}\max_{0 \leq \theta_i \leq 1}I_{R_l, \tau}(g^{(i)}(\theta_i)).
            \end{align*}
      We have abused the notation a little by writing $I_{R_l, \tau}: S_0^1(B_{S_l}(\xi_l^{(i)})) \to  \R $ for $i=1,2$.

          \begin{prop}\label{prop:4.2.1}
          Let $\{ R_l \}$ be a sequence of functions  satisfying \eqref{18}, \eqref{20} and \eqref{21}. Then it
          holds $c_{l, \tau}^{(1)}=c^{(1)} + o(1)$ for $i=1,2$, where $o(1)\to 0$ as $l\to \infty$.          \end{prop}
          \begin{proof}The proof can be completed by using the definition of $c_{l, \tau}^{(i)}$ with some standard functional analysis arguments, we omit the details here.
          \end{proof}

          We define
          \begin{gather}\notag
            \Gamma_l = \{ G =g^{(1)} +g^{(2)}: g^{(1)}, g^{(2)}\text{ satisfy }\, \eqref{110}-\eqref{113}\}, \\\label{110}
            g^{(1)}, g^{(2)} \in C([0,1]^2, E), \\\label{111}
            g^{(1)}(0, \theta_2)= g^{(2)}(\theta_1,0)=0, \quad 0 \leq \theta_1, \theta_2 \leq 1, \\\label{112}
            I_{R_l, \tau}(g^{(1)}(1, \theta_2))<0,~ I_{R_l, \tau}(g^{(2)}(\theta_1,1))<0, \quad 0 \leq \theta_1, \theta_2 \leq 1,\\\label{113}
            \operatorname{supp}g^{(i)}(\theta) \subset B_{S_l}(\xi_l^{(i)}), \quad \theta=(\theta_1,\theta_2) \in [0, 1]^2,~i=1,2, \\
            b_{l, \tau} = \inf_{G \in \Gamma_l}\max_{\theta \in [0,1]^2} I_{R_l, \tau}(G(\theta)).\notag
          \end{gather}
  \begin{rem}
  	Observe that if $G=g^{(1)}+g^{(2)}$ with $g^{(1)}\in \gamma_{l, \tau}^{(1)}$, $g^{(2)}\in \gamma_{l, \tau}^{(2)}$, $\operatorname{supp}g^{(1)}\cap \operatorname{supp}g^{(2)}=\emptyset$,
  	then $I_{R_{l}, \tau}(G)=I_{R_{l}, \tau}(g^{(1)})+I_{R_{l}, \tau}
  	(g^{(2)})$.
  \end{rem}
          \begin{prop}\label{prop:4.2.2}
          Let $\{R_l \}$ be a sequence of functions satisfying \eqref{18}, \eqref{20} and \eqref{21}, then it holds $b_{l, \tau} =c_{l, \tau}^{(1)} +c_{l, \tau}^{(2)} +o(1)$.
          \end{prop}
          \begin{proof}
We first prove that $b_{l, \tau} \geq c_{l, \tau}^{(1)}+c_{l, \tau}^{(2)}$. Indeed it can be achieved from  the definition  of  $c_{l, \tau}^{(i)}$ with additional compactness argument on  $[0,1]^2$, we omit it here and refer to \cite[Proposition 3.4]{CR1} for details.

        On the other hand, for   $0 \leq \theta_1, \theta_2 \leq 1$, let
           $
              g_l^{(i)}(\theta_i) = \theta_i C_1R_l(\xi_l^{(i)})^{(2-Q)/4}\eta(\xi_l^{(i)} \circ \cdot)w_{\xi_l^{(i)},\lam _l}$ for $i=1,2$,
      where $\lam_l\to \infty$ is a sequence satisfying $(\lam_l)^{\tau}=1+o(1)$, and
       $C_1 = C_1(n,A_1, A_2)>1$ is a constant, such  that  for $l$  large,  $
             I_{R_l, \tau}g_l^{(i)}(1)<0$ for each $i=1,2$.
 We fix the value of $C_1$ from now on.

            For $\theta = (\theta_1, \theta_2) \in [0, 1]^2$, let $
              G_l(\theta) = g_l^{(1)}(\theta_1) + g_l^{(2)}(\theta_2)$.
Clearly, $G_l \in \Gamma_l$ and
            \begin{align*}
              \max_{\theta \in[0, 1]^2}I_{R_l, \tau}(G_l(\theta)) =&\max_{\theta \in[0, 1]^2}I_{R_l, \tau}(g_l^{(1)}(\theta_1)) +\max_{\theta \in[0, 1]^2}I_{R_l, \tau}(g_l^{(1)}(\theta_1))\\
              \leq&\sum_{i=1}^2\max_{0 \leq s < \infty}I_{R_l, \tau}(s\eta(\xi_l^{(i)} \circ \cdot)w_{\xi_l^{(i)},\lam _l,})+o(1)\\
              =&\sum_{i=1}^2\frac{1}{Q}(a^{(i)})^{(2-Q)/2}(S_n)^Q  +o(1)\\
              =&c_{l, \tau}^{(1)} +c_{l, \tau}^{(2)} +o(1),
            \end{align*}
            where the last equality is due to Proposition \ref{prop:4.2.1}. Therefore, $b_{l, \tau} \leq c_{l, \tau}^{(1)} +c_{l, \tau}^{(2)} +o(1)$.
          \end{proof}

          In the subsequent analysis, we show that under the contrary of Theorem \ref{thm:4.1}, it is possible to construct $H_{l} \in \Gamma_l$ for large $l$, such that
         $$
            \max_{\theta \in[0, 1]^2}I_{R_l, \tau}(H_l(\theta))< b_{l, \tau},
         $$
          which contradicts to the definition of $b_{l, \tau}$. A lengthy construction is required to establish this fact and a brief sketch of it will be given now.

         \textbf{Step 1:} Choose some suitably small number $\var _4>0$, we can construct $G_l \in \Gamma_l$ such that
        $$
            \max_{\theta \in[0, 1]^2}I_{R_l, \tau}(G_l(\theta)) \leq b_{l, \tau}+ \var _4.
       $$
         Furthermore, $G_l$ satisfies some further properties.

          \textbf{Step 2:} We follow the negative gradient flow of $I_{R_l, \tau}$ to deform $G_l$ to $U_l$ with
         $$
            \max_{\theta \in[0, 1]^2}I_{R_l, \tau}(U_l(\theta)) \leq b_{l, \tau}- \var _4.
      $$
          However, $U_l$ is not necessarily in $\Gamma_l$ any more since the deformation may not preserve properties \eqref{113}.

       \textbf{Step 3:} Applying Propositions \ref{prop:4.2}, \ref{prop:2.4} and \ref{prop:3.1}, we modify $U_l$ to obtain $H_l \in \Gamma_l$ with
        $$
            \max_{\theta \in[0, 1]^2}I_{R_l, \tau}(H_l(\theta))\leq b_{l, \tau}-\var _4/2.
      $$

\textbf{Step 4:} Complete the proof by using the minimax structure of $H_l$.

          All four steps are completed for large $l$. Now we start to establish these  steps.

        \textbf{Step 1: Construction of $G_l$.}

        Let $G_l$ be the one we have just defined. We establish some properties of $G_l$ which are needed.
          \begin{lem}\label{lem:4.2.1}
            For any $ \var \in (0,1)$, if  $I_{R_{l}, \tau}(g_{l}^{(i)}(\theta_{i})) \geq c_{l, \tau}^{(i)}-\var$ for  $i = 1,2$,  then there exist two
            constants $\Lam_1 = \Lam_1(\var , A_1, A_3)>1$ and $C_{0}=C_{0}(n)>0$, such that for any $l \geq \Lam_1$, $0 \leq \theta_1, \theta_2 \leq 1$, we have $|C_{1} \theta_{i}-1| \leq C_{0} \sqrt{\var}$ for $i=1,2$, where $C_1$ is the constant  in the proof of Proposition \ref{prop:4.2.2}.
           \end{lem}
          \begin{proof}
          We only take into account the case $i = 1$ since the other case can be covered in the same  way.  Let $s = C_1 \theta_1$, a direct calculation shows that
              \begin{align*}
              I_{R_l, \tau}(g_l^{(1)}(\theta_1)) =& \frac{1}{2}s^2 R_l(\xi_l^{(1)})^{(2-Q)/2} \int |\nabla_{\H}(\eta(\xi_l^{(1)} \circ \cdot)w_{\xi_l^{(1)},\lam _l })|^2\\
              &- \frac{1}{p+1}s^{p+1}  R_l(\xi_l^{(1)})^{-(Q-2)(p+1)/4} \int R_l H^{\tau}|\eta(\xi_l^{(1)} \circ \cdot)w_{\xi_l^{(1)},\lam _l }|^{p+1}\\
              =&\Big( \frac{1}{2} +o(1) \Big)s^2R_l(\xi_l^{(1)})^{(2-Q)/2}\int|\nabla_{\H}w_{0,1}|^2\\
              &-\Big( \frac{1}{Q^*} +o(1) \Big)s^{p+1}R_l(\xi_l^{(1)})^{(2-Q)/2}\int w_{0,1}^{Q^* }\\
              =& \Big[ \Big( \frac{Q}{2} +o(1) \Big)s^2 -\Big( \frac{Q-2}{2} +o(1) \Big)s^{p+1} \Big]c_{l, \tau}^{(1)},
              \end{align*}
      where Proposition \ref{prop:4.2.2} is used in the last step. Hence, using above identity and the hypothesis $I_{R_{l}, \tau}(g_{l}^{(1)}(\theta_{1})) \geq c_{l, \tau}^{(1)}-\var$, we complete the proof.
          \end{proof}
          \begin{lem}\label{lem:4.2.2}
          For any $ \var \in (0,1)$, there exists $\Lam_2= \Lam_2(\var , A_1, A_3)>\Lam_1$ such that for $l \geq \Lam_2$, $0 \leq \theta_1, \theta_2 \leq 1$, we have $I_{R_l, \tau}(g_l^{(i)}(\theta_i)) \leq c_{l, \tau}^{(i)} +\var/10$ for $i=1,2$.
           \end{lem}
          \begin{proof}The proof is similar to that of Proposition \ref{prop:4.2.2}, we omit the details here.
          \end{proof}
          \begin{lem}\label{lem4.2.3}
          For any $\var\in (0,1)$, there exists $\Lam_3 = \Lam_3(n,\var , A_1, A_3)> \Lam_2$ such that for $l \geq \Lam_3$, we have
$I_{R_l, \tau}(G_l(\theta))|_{\theta \in \pa [0, 1]^2} \leq \max \{ c^{(1)} + \var , c^{(2)} + \var  \}$.
           \end{lem}
          \begin{proof}
             Lemma \ref{lem4.2.3} follows immediately from Lemma \ref{lem:4.2.2}.
          \end{proof}
          \begin{lem}\label{lem:4.2.4}
            There exists some universal constant $C_0 = C_0(n)>1$ such that for any $\var\in (0,1/2)$, $l \geq \Lam_3(\var , A_1, A_3)$ and $\theta \in [0, 1]^2$, $
              I_{R_l, \tau}(G_l(\theta))\geq c_{l, \tau}^{(1)} +c_{l, \tau}^{(2)} -\var$   implies that   $|C_1 \theta_i -1|\leq C_0 \sqrt{\var }$  for  $i=1,2$.
              \end{lem}
          \begin{proof}Lemma \ref{lem:4.2.4} follows  from   Lemmas  \ref{lem:4.2.1} and \ref{lem:4.2.2}, we omit the details here.
          \end{proof}

\textbf{Step 2: The deformation of $G_l$.}

          Let
      $$
            M_l = \sup \{ \| I_{R_l, \tau}' (u) \|: u \in V_l(2, \var _1) \}, \quad
            \beta_l = \operatorname{dist}(\pa  \widetilde{V}_l(2, \var _2), \pa  \widetilde{V}_l(2, \var _2/2)).
        $$
    One can see from the definition of  $M_l$ that there exists a constant $C_2 = C_2(n,A_1, \var _2)>1$ such that $M_l \leq C_2$.
    It is also clear from the definition of $\widetilde{V}_l(2, \var _2)$ that $\beta_l \geq \var _2/4$.

      By Lemma \ref{lem:4.2.4}, we choose $\var _4$ to satisfy, for $l$ large, that
          \begin{gather}\label{121}
            \var _4 < \min\Big\{\var _3,\frac{1}{2A_4},\frac{\var _2 \delta_4^2}{8C_2}\Big\}, \\
            \begin{aligned}
            I_{R_l, \tau}&(G_l(\theta)) \geq c_{l, \tau}^{(1)} + c_{l, \tau}^{(2)} -\var _4\text{ implies that}\\
            & \qquad \qquad G_l(\theta) \in \widetilde{V}_l(2, \var _2/2), \xi_1(G_l(\theta)) \in O_l^{(1)}, \xi_2(G_l(\theta)) \in O_l^{(2)},
            \end{aligned}\label{123}
          \end{gather}
where  $\delta_4=\delta_4(\var_2,\var_3)$ is the constant   in   Proposition \ref{prop:4.2}. $G_l(\theta)$ has been defined by now.

          We know from Lemma \ref{lem:4.2.2} that for $l$ large enough,
       $  \max_{\theta \in [0, 1]^2} I_{R_l, \tau}(G_l(\theta)) \leq c_{l, \tau}^{(1)} + c_{l, \tau}^{(2)} +\var _4$.

     For any $u_0\in \widetilde{V}_l(2, \var _2/2)$,     we consider the negative gradient of $I_{R_l, \tau}$:
          \be \label{124}
          \begin{aligned}
            \frac{\d }{\d s}\phi(s, u_0)&=-I_{R_l, \tau}' (\phi(s, u_0)), \quad s \geq 0,\\
            \phi(0, u_0) &= u_0.
          \end{aligned}
          \ee
Under the contrary of Theorem \ref{thm:4.1}, we know that $I_{R_{l}, \tau}$ satisfies the Palais-Smale condition. Furthermore,  the flow defined above never stops before exiting $V_{l}(2, \var^{*})$.

Now         we define $U_l \in C([0, 1]^2, E)$ by the following.
\begin{itemize}
	\item If $I_{R_l, \tau}(G_l(\theta)) \leq c_{l, \tau}^{(1)} + c_{l, \tau}^{(2)} -\var _4$, we define $s_l^*(\theta) = 0$.
	\item	 If $I_{R_l, \tau}(G_l(\theta)) > c_{l, \tau}^{(1)} + c_{l, \tau}^{(2)} -\var _4$, then, according to \eqref{123}, $G_l(\theta) \in \widetilde{V}_l(2, \var _2/2)$, $\xi_1(G_l(\theta)) \in O_l^{(1)}$, $\xi_2(G_l(\theta)) \in O_l^{(2)}$. We define $s_l^*(\theta) =\min \{s>0 : I_{R_l, \tau}(\phi(s, G_l(\theta))) = c_{l, \tau}^{(2)} -\var _4\}$.
\end{itemize}
We set $$U_l(\theta) = \phi(s_l^*(\theta), G_l(\theta)).$$
          The above definition is justified in the following.
          \begin{lem}\label{lem:4.2.5}
          For any $u_0 \in \widetilde{V}_l(2, \var _2/2)$, with $\xi_1(u_0) \in O_l^{(1)}$, $\xi_2(u_0) \in O_l^{(2)}$, and $c_{l, \tau}^{(1)} + c_{l, \tau}^{(2)} -\var _4 < I_{R_l, \tau}(u_0) <c_{l, \tau}^{(1)} + c_{l, \tau}^{(2)} +\var _4$, the flow line $\phi(s, u_0)$ $(s\geq 0)$ cannot leave $\widetilde{V}_l(2, \var _2)$ before reaching $I_{R_l, \tau}^{-1}(c_{l, \tau}^{(1)} + c_{l, \tau}^{(2)} -\var _4)$.
           \end{lem}
          \begin{proof}The proof can be done exactly in the same way as in \cite[Lemma 5]{BC1991}, so we omit it.
          \end{proof}
      \begin{rem}
      We   see from Lemma \ref{lem:4.2.5} that $s_{l}^{*}(\theta)$ is well defined. Since $I_{R_{l}, \tau}$ has no critical point in $
      	\widetilde{V}_{l}(2, \var_{2}) \cap\{ u \in E: | I_{R_{l}, \tau}(u)-c^{(1)}-c^{(2)}|\leq \var_{4}\} \subset V_{l}(2, \var^{*}) \cap\{u  \in E:|I_{R_{l}, \tau}(u)-c^{(1)}-c^{(2)}|\leq \var^{*}\}$
      	under the contradiction hypothesis, $s_{l}^{*}(\theta)$ is continuous in $\theta$ (see also  \cite[Proposition 5.11]{Li1993} and \cite[Lemma 5]{BC1991}), hence $U_l \in C([0,1]^{2}, E)$.
      \end{rem}

\textbf{Step 3: The construction of $H_l$.}

It follows from the construction of $U_l$ that  $\max_{\theta\in [0,1]^2}I_{R_l,\tau}(U_l(\theta))\leq c_{l, \tau}^{(1)}+c_{l, \tau}^{(2)}-\var_{4}$.
Since the gradient flow does not keep property \eqref{113}, $U_l(\theta)$ is not necessarily in $\Gamma_{l}$ any more. It follows from Lemma \ref{lem:4.2.5} that if $I_{R_{l}, \tau}(G_{l}(\theta))>c_{l, \tau}^{(1)}+c_{l,\tau}^{(2)}-\var_{4}$, then the  gradient flow $\phi(s,G_l(\theta))$ $(s\geq 0)$ cannot leave  $\widetilde{V}_{l}(2, \var_{2})$ before reaching $I_{R_{l}, \tau}^{-1}(c_{l, \tau}^{(1)}+c_{l, \tau}^{(2)}-\var_{4})$.      It follows that if $I_{R_l, \tau}(G_l(\theta))> c_{l, \tau}^{(1)} + c_{l, \tau}^{(2)}- \var _4$, then $U_l(\theta) \in \widetilde{V}_l (2,\var_2)
          \subset V_l(2, o_\var (1))$ with  $\xi_1(U_l(\theta)) \in O_l^{(1)}$, $\xi_2(U_l(\theta)) \in O_l^{(2)}$, which implies that
          \begin{gather}\label{125}
            \int_{\Om _l} |\nabla_{\H}U_l(\theta)|^2 + |U_l(\theta)|^{Q^* } = o_{\var _2}(1), \\\label{126}
            \| U_l(\theta) \|_{H^{1/2}(\pa \Om _l)}= o_{\var _2}(1),
          \end{gather}
          where
  \begin{gather*}
   \Om _l = \Hn \backslash \{ B_S(\xi_l^{(1)})\cup B_S(\xi_l^{(2)}) \},\\
  S = 4(\operatorname{diam} O^{(1)} + \operatorname{diam} O^{(2)}),\\ \operatorname{diam} O^{(i)}=\sup\{d(\xi,\xi_0):\xi,\xi_0\in \Hn\}\quad \text{ for }\, i=1,2,
  \end{gather*}
and $O^{(1)}$, $O^{(2)}$ are defined by \eqref{2.13}.

          Without loss of generality, we can assume that $\var _2>0$ has been so small that we can apply Proposition \ref{prop:3.1}. We modify $U_l(\theta)$ in $\Om _l$ after making the following minimization.

          Let
       $$
            \varphi_l(\theta) = U_l(\theta)|_{\pa \Om_l }.
   $$
Because of \eqref{125} and \eqref{126}, we can apply Proposition \ref{prop:3.1} to obtain the minimizer  $u_{\varphi_l}(\theta)$ to  the problem \eqref{minimization} with $\varphi=\varphi_l(\theta)$, $\Om=\Om_l$.
   We define for $\theta \in [0, 1]^2$ that
      $$
            W_l(\theta)(\xi)=\begin{cases}
                          U_l(\theta)(\xi), & \xi\in \Hn \backslash \Om _l, \\
                          u_{\varphi_l}(\theta)(\xi), & \xi\in \Om _l.
                        \end{cases}
$$
       It follows from Proposition \ref{prop:3.1} that $W_l \in C([0, 1]^2, E)$ and satisfies
          \begin{gather}\label{127}
            \max_{\theta \in [0, 1]^2}I_{R_l, \tau}(W_l(\theta)) \leq \max_{\theta \in [0, 1]^2}I_{R_l, \tau}(U_l(\theta)) \leq c_{l, \tau}^{(1)} + c_{l, \tau}^{(2)}- \var _4,\\
            \int_{\Om _l}|\nabla_{\H}W_l(\theta)|^2 + |W_l(\theta)|^{Q^* }= o_{\var _2}(1), \label{127-2}\\\label{128}
            -\Delta_{\H}W_l(\theta)= R_l(\xi)H^\tau|W_l(\theta)|^{p-1}W_l(\theta) \quad \text{ in }\,\Om _l.
          \end{gather}
Moreover, $ W_l(\theta)\geq 0$ in $\Om _l$ can be proved by using \eqref{127-2} and \eqref{Sobolevineqdomain}, see also the proof in the
Claim \ref{claim:1}. $ W_l(\theta)> 0$ in $\Om _l^c$
can be seen from the definition of $V_l(2, o_{\var_2}(1))$ and Proposition \ref{prop:4.1}.

Write
\begin{align*}
 \Omega_l^1:=&(B_{l_1}(\xi_l^{(1)}) \backslash B_S(\xi_l^{(1)})) \cup(B_{l_1}(\xi_l^{(2)}) \backslash B_r(z_l^{(2)})), \\
 \Om_l^2:=&(B_{l_2}(\xi_l^{(1)}) \backslash B_{l_1}(\xi_l^{(1)})) \cup(B_{l_2}(\xi_l^{(2)}) \backslash B_{l_1}(\xi_l^{(2)})), \\
 \Om_l^3:=&(\Hn \backslash B_{l_2}(\xi_l^{(1)})) \cap(\Hn \backslash B_{l_2}(\xi_l^{(2)})).
\end{align*}
Obviously, $\Om_l=\Om_l^1 \cup \Om_l^2 \cup \Om_l^3$ for large $l$.
For $l_2>100l_1>1000S$ (the values of $l_1$, $l_2$ will be determined
          in the end), we introduce the cut-off functions $\eta_l \in C_c^{\infty}(\Hn)$ satisfying
          \begin{gather*}
            \eta_l(\xi)=\begin{cases}
                          1, & \xi\in B_{l_1}(\xi_l^{(1)})\cap B_{l_1}(\xi_l^{(2)}), \\
                          0, & \xi \in (\Hn \backslash B_{l_2}(\xi_l^{(1)}))\cup (\Hn \backslash B_{l_2}(\xi_l^{(2)})),\\
                          \geq 0, & \text{ otherwise},
                        \end{cases}\\
                        |\nabla_{\H}\eta_l(\xi)|\leq \frac{10}{l_2 -l_1}, \quad \xi \in \Hn,
          \end{gather*}
and set  $H_l(\theta) = \eta_lW_l(\theta)$.

  \textbf{Step 4:} Now we complete the proof by using the minimax structure of $H_l$.
 Roughly speaking, we  will prove that $H_{l}(\theta)\in \Gamma_{l}$  but its  energy bound contradicts  to  $b_{l,\tau}$.

           Multiplying   $(1-\eta_l)W_l(\theta)$ on both sides of \eqref{128} and integrating by parts, we have
  $$
          \int_{\Om_l} \nabla_{\H}W_l(\theta) \nabla_{\H}((1-\eta_l)W_l(\theta)) = \int_{\Om_l} R_l H^\tau (1-\eta_l)|W_l(\theta)|^{p+1}.
$$
A direct computation shows that
\begin{align*}
& \int_{\Om_l^3} |\nabla_{\H}W_l(\theta)|^2-\int_{\Om_l^3} R_l H^\tau|W_l(\theta)|^{p+1} \\
= & -\int_{\Om_2^l}\nabla_{\H}W_l(\theta) \nabla_{\H}((1-\eta_l)W_l(\theta))+\int_{\Om_2^l} R_l H^\tau(1-\eta_l)|W_l(\theta)|^{p+1} \\
\geq & -\int_{\Om_2^l}|\nabla_{\H}W_l(\theta)|^2-\frac{10}{l_2-l_1}\int_{\Om_2^l}|\nabla_{\H}W_l(\theta)||W_l(\theta)| - 2 A_1 \int_{\Om_2^l}|W_l(\theta)|^{p+1} .
\end{align*}
    By Proposition \ref{prop:2.4} we know that
          \begin{gather}\label{129}
            |W_l(\theta)(\xi)| \leq \frac{C_3(n,A_1)}{d(\xi, \xi_l^{(i)})^{Q-2}}, \quad l_1 \leq d(\xi, \xi_l^{(i)}) \leq l_2, \\\label{131}
            |\nabla_{\H} W_l(\theta)(\xi)| \leq \frac{C_3(n,A_1)}{d(\xi, \xi_l^{(i)})^{Q-1}}, \quad l_1 \leq d(\xi, \xi_l^{(i)}) \leq l_2 ,
          \end{gather}
    when $l$ is chosen large enough.
     By combining    \eqref{129} and \eqref{131}, we have
 \be\label{132}
    \int_{\Om_l^3} |\nabla_{\H}W_l(\theta)|^2-\int_{\Om_l^3} R_l H^\tau|W_l(\theta)|^{p+1}
    \geq  -C_0(n)C_3(n,A_1)\Big(\frac{1}{l_1^2}-\frac{1}{l_2^2}\Big) .
\ee
  Using the above  estimates \eqref{129}--\eqref{132}, we obtain
            \begin{align*}
            I_{R_l, \tau}(H_l(\theta))
            =&\frac{1}{2}\int|\nabla_{\H}  \eta_l|^2 |W_l(\theta))|^2 +\int\eta_l W_l(\theta)\nabla_{\H}\eta_l\nabla_{\H} W_l(\theta) \\
            &+\frac{1}{2}\int \eta_l^2|\nabla_{\H} W_l(\theta)|^2 -\frac{1}{p+1} \int R_lH^\tau|\eta_l W_l(\theta)|^{p+1}
            \\=&I_{R_l, \tau}(W_l(\theta))+\frac{1}{2}\int_{\Om_l^2}|\nabla_{\H}  \eta_l|^2 |W_l(\theta))|^2 +\int_{\Om_l^2}\eta_l W_l(\theta)\nabla_{\H}\eta_l\nabla_{\H} W_l(\theta) \\
            &+\frac{1}{2}\int_{\Om_l^2\cup \Om_l^3} (\eta_l^2-1)|\nabla_{\H} W_l(\theta)|^2 +\frac{1}{p+1} \int_{\Om_l^2\cup \Om_l^3} R_lH^\tau(1-\eta_l^{p+1})| W_l(\theta)|^{p+1}
            \\
            \leq& I_{R_l, \tau}(W_l(\theta))+C_0(n)C_3(n,A_1)\frac{\ln l_2-\ln l_1}{(l_2-l_1)^2}+ C_0(n)C_3(n,A_1)\frac{1}{l_1l_2}
           \\&+C_0(n)C_3(n,A_1)\Big(\frac{1}{l_1^2}-\frac{1}{l_2^2}\Big) \\& +\frac{1}{2} \int_{\Om_l^2} (|\eta_l|^2-1)|\nabla_{\H} W_l(\theta)|^2
            + \frac{1}{p+1} \int_{\Om_l^2} R_lH^\tau (1-|\eta_l|^{p+1})| W_l(\theta)|^{p+1}\\
            \leq& I_{R_l, \tau}(W_l(\theta))+C_0(n)C_3(n,A_1)\frac{\ln l_2-\ln l_1}{(l_2-l_1)^2} +C_0(n)C_3(n,A_1)\frac{1}{l_1l_2}\\&+C_0(n)C_3(n,A_1)\Big(\frac{1}{l_1^2}-\frac{1}{l_2^2}\Big)+C_0(n)C_3(n,A_1)\Big(\frac{1}{l_1^4}-\frac{1}{l_2^4}\Big).
            \end{align*}
     Now using \eqref{127}  and choosing $l_2 >200l_1$, $l_1>10S$ to be  large enough, we have
      $$
         I_{R_l, \tau}(H_l(\theta)) \leq  c_{l, \tau}^{(1)} + c_{l, \tau}^{(2)} -\var _4/2.
      $$
          Then for $l$ large enough (depending on $l_{1}, l_{2}, \var's, C's$), it holds
          $H_l \in \Gamma_l$. Therefore, for $l$ large enough, we obtain
         $$
            \max_{\theta \in [0, 1]^2}I_{R_l, \tau}(H_l(\theta)) \leq c_{l, \tau}^{(1)} + c_{l, \tau}^{(2)} -\var _4/2 <b_{l, \tau},
       $$
which contradicts to the definition of  $b_{l, \tau}$. We now complete the proof of Theorem \ref{thm:4.1}.
\section{Blow up analysis and proof of main theorems}\label{sec:5}
In this section we  present the main result Proposition \ref{prop:5.1}, from which we deduce Theorems \ref{thm:1}--\ref{thm:3} and Corollary \ref{cor:1}. The crucial ingredients of our proofs are the understanding of the blow up profiles, see the work in Prajapat-Ramaswamy \cite{PR2003}.
\subsection{Subcritical approximation}
We state the main result as following:
\begin{prop}\label{prop:5.1}
  Assume that $\{R_l\}$ is a sequence of functions satisfying conditions (i)-(iii) and $(R_2)$. Assume also that there exist some bounded open sets $O^{(1)}, \ldots , O^{m} \subset \Hn$ and some  constants $\delta_2, \delta_3>0$, such that for all $1 \leq i \leq m$,
  \begin{gather*}
    (\xi_l^{(i)})^{-1} \circ \widetilde{O}_l^{(i)} \subset O^{(i)} \quad \text{ for all }\,l,  \\
    \{ u: I_{R_\infty^{(i)}}'(u) = 0, u>0, u \in E, c^{(i)} \leq I_{R_\infty^{(i)}}(u) \leq c^{(i)} + \delta_2 \}\cap V(1, \delta_3, O^{(i)}, R_\infty^{(i)})= \emptyset.
  \end{gather*}
  Then for any $\var  >0$, there exists  integer $\overline{l}_{\var , m}>0$, such that, for all $l \geq \overline{l}_{\var , m}$, there exists $u_l \in V_l(m, \var )$ which solves
  \be \label{139}
-\Delta_{\H}u_l = R_l (\xi) u_l^{(Q+2)/(Q-2)}  ,\quad u_l>0 \quad \text{ in }\, \Hn.
  \ee
  Furthermore, $u_l$ satisfies
  \be \label{140}
    \sum_{i=1}^{m}c^{(i)}- \var  \leq I_{R_{l}}(u_l) \leq \sum_{i=1}^{m}c^{(i)} + \var .
  \ee
\end{prop}
The proof of Proposition \ref{prop:5.1} is by contradiction argument,
depending on  blow up analysis for a family of subcritical equations \eqref{30} approximating \eqref{139}. More precisely, if the sequence
of  subcritical solutions $u_{l, \tau}$ $(0<\tau<\overline{\tau}_{l})$ obtained in Theorem \ref{thm:4.1} is uniformly bounded
as $\tau\to  0$, some subelliptic estimates in \cite[Claim 5.3]{PR2003}  imply that there exists a subsequence converging to a
positive solution $u_l$ of \eqref{139} satisfying  \eqref{140}. However, a prior $\{u_{l,\tau}\}$ might blow up,  we have to rule out this possibility.

Note that $u_{l, \tau}\in V_{l}(m, o_{\var_2}(1))$, which consists of functions with $m$ $(m \geq 2)$  \emph{bumps},     some blow up analysis results in \cite[Sections 6-7]{PR2003} imply that, as $\tau\to  0$, there is no blow up occurring under the hypotheses of Proposition \ref{prop:5.1}.    Thus we  only need to show the boundedness of $\{u_{l, \tau}\}$ as $\tau \to  0$, this will be done by  contradiction argument with the aid of blow up analysis established in Prajapat-Ramaswamy \cite{PR2003}.  We give a brief introduction for readers' convenience.

Let $\{\tau_i\}_{i=1}^{\infty}$ be a sequence of nonnegative constants satisfying $\lim _{i \to  \infty} \tau_i=0$, and set $p_i=\frac{Q+2}{Q-2}-\tau_i$.
Suppose that $0 \leq u_i \in \Gamma_2(\Om)$ satisfies
\be\label{blowupeq}
-\Delta_{\H}u_i = R_i (\xi) u_i^{p_i} \quad  \text{ in }\, \Om,
\ee
where $\Om $ is a domain in $\Hn$ and  $R_i \in \Gamma_{2+\al}(\Om )$, $0<\al<1$ satisfy, for some positive constants $A_1$ and $A_2$,
\be\label{bdd}
1 / A_1 \leq R_i \quad \text { and } \quad\|R_i\|_{C^1(\Om )} \leq A_2.
\ee

We recall the notion of various types of blow up points, which were introduced by Schoen \cite{Schoen43,Schoen44,Schoen45}.	 This incisive concept  helps to regain compactness, and forms a natural demarcation to more complicated types of blow up phenomenon.
\begin{defn}\label{def:blowuppoint}
	Suppose that $\{u_i\}$ satisfies \eqref{blowupeq} and  $\{R_i\}$ satisfies \eqref{bdd}.
	\begin{itemize}
		\item[(1)]$\overline{\xi}\in \Om$ is called a blow up point of $\{u_i\}$ if there exists a
		sequence $\xi_i\in \Om $ such that $\xi_i$ is a local maximum point of $u_i$
		satisfying $u_i(\xi_i)\to \infty$ and $\xi_i\to \overline{\xi}$ as $i\to \infty$. For simplicity, we will often say
		that $\xi_i\to \overline{\xi}$ is a  blow up point of $\{u_i\}$.
		\item[(2)]$\overline{\xi}\in \Om$ is an isolated blow up point of $\{u_i\}$ if  $\xi_i\to \overline{\xi}$ is a blow up point such that
		$$	u_i(\xi)\leq \overline C d( \xi,\xi_i)^{-2/(p_i-1)} \quad \text{ for any }\, \xi\in B_{\overline{r}}(\xi_i)\setminus\{\xi_i\},$$
		where	$0<\overline{r}<\operatorname{dist}(\overline{\xi},\Om)$ and  $\overline C>0$ are some constants.
				\item[(3)]For any  $\theta \in \pa B_1$, we define the function $f_{u_i, \theta}(s):[0, R] \to \R $ (for a fixed $R>0$) as
				$$
				f_{u_i, \theta}(s) = s^{2/(p_i-1)}u_i(\xi_i\circ \delta_s \theta),
				$$
				where  $\delta_s \theta$ is the dilation in $\Hn$. 	We say that an isolated blow up point $\overline{\xi}\in \Om$ of $\{u_i\}$ is simple if there exists $\rho>0$ (independent of $i$ and $\theta \in \pa B_1$) such that $f_{u_i, \theta}$ has precisely one critical point in $(0,\rho)$ for every $\theta\in  \pa B_1$ for large $i$.
	\end{itemize}
\end{defn}

	Roughly speaking, item (2) (or (3), respectively) in the above definition describes the situation when clustering of bubbles (or bubble towers, respectively) is excluded among various blow up scenarios. 	We also remark that item (3) is a modified  definition of isolated simple blow up point when comparing with Riemannian manifold. Indeed,  according to Schoen \cite{Schoen43,Schoen44,Schoen45}, a simple blow up point on a sphere $\Sn$ $(n\geq3)$ is a point where the solution of \eqref{Nirenberg},  with the exponent $\frac{n+2}{n-2}$ substituted by $p\in (1,\frac{n+2}{n-2}]$,  approximates the \emph{standard solution} up to a conformal transformation, in a neighborhood. This definition was further reformulated by Li in \cite{Li1996} using spherical averages. However, this definition does not seem to work for the Heisenberg group since 	the \emph{standard solution} in the case of CR sphere $\S^{2n+1}$ is not radial.

Before starting the proof of Proposition \ref{prop:5.1}, we  have additional remarks on the solutions $\{u_{l,\tau}\}$. The following statements can be found in \cite{PR2003}:
\begin{itemize}
	\item	By some standard blow up arguments, the blow up points  cannot occur in $\Rn \setminus (\cup_{j=1}^{m}{\widetilde{O}}_{l}^{(j)})$ since the energy of $\{u_{l, \tau}\}$ in that region is small using the fact  $u_{l, \tau} \in V_{l}(m, o_{\var_2}(1))$. Hence the blow up points can occur only in $\cup_{j=1}^{m}{\widetilde{O}}_{l}^{(j)}$.
	\item	 Using  Proposition 7.1 in \cite{PR2003} and  the definition of $V_{l}(m,o_{\var_2}(1))$,  there are at most $m$ isolated blow up points, namely, the blow up occurs in $\{\overline{\xi}_{1}, \ldots, \overline{\xi}_{m}\}$ for some $\overline{\xi}_{j} \in \widetilde{O}_{l}^{(j)}$ $(1\leq j\leq m)$.  Here we used the finite energy condition \eqref{31}.
	\item 	Under the flatness condition $(R_2)$, we conclude from \cite[Proposition 6.2]{PR2003}  that an isolated blow up point has to be an isolated simple blow up point. From the structure of functions in $V_{l}(m,o_{\var_2}(1))$ we know that if the blow up does occur, there have to be exactly $m$ isolated simple  blow up points.  		
\end{itemize}

Let us consider this situation only, namely, $\{\overline{\xi}_{1}, \ldots,\overline{\xi}_{m}\}$  is the blow up set and they are all \emph{isolated simple} blow up points. Moreover,  in our situation, $R_{i}=RH^{\tau_i}$ is the sequence of functions in \eqref{blowupeq} with $\Om=\Hn$.  We assume that  the blow up occurs at $u_{i}=u_{l, \tau_{i}}$ and  we suppress the dependence of $l$ in the notation since $l$ is fixed in the blow up analysis. Now we complete the proof of Proposition \ref{prop:5.1} by  checking balance via the   Pohozeav identity \eqref{Pohozaevid}.
\begin{proof}[Proof of Proposition \ref{prop:5.1}]
Let  $u_i$ be the solution of \eqref{blowupeq} with $R_{i}=RH^{\tau_i}$ and  $\Om=\Hn$. Without loss of generality, we may assume  that $\overline{\xi}_1=0$ and $\xi_{i}=(x^{(i)},y^{(i)},t^{(i)})\to 0$ be the sequence as in Definition  \ref{def:blowuppoint}. By \eqref{20} and \eqref{21},  we also assume that $R_i>0$  in $B_1$.  Applying the Pohozaev identity \eqref{Pohozaevid}   to  $u_i$, we obtain
	\begin{align}
		\int_{\pa {B_\si (\xi_i)}}B(\si , \xi_i, u_i, \nabla_{\H}u_i) =& \Big( \frac{Q}{p_i+1} - \frac{Q-2}{2}\Big) \int_{B_\si (\xi_i)}R_iu_i^{p_i+1}\notag\\& + \frac{1}{p_i +1}\int_{B_\si (\xi_i)} \mathcal{X}_i(R_i)u_i^{p_i+1}-\frac{1}{p_i +1}\int_{\pa  B_\si (\xi_i)} R_i u_i^{p_i+1} \mathcal{X}_i \cdot \nu,\label{identity}
	\end{align}
	where  $\nu$ is the outward unit normal vector with respect to $ \pa B_\si (\xi_i)$ and $$
	B(\si , \xi_i, u_i, \nabla_{\H}u_i) =
	\frac{Q-2}{2}(A \nabla u_i \cdot \nu) u_i-\frac{1}{2}|\nabla_{\H} u_i|^2 \mathcal{X}_i \cdot \nu +(A \nabla u_i \cdot \nu) \mathcal{X}_i (u_i)
	$$
	with $$\mathcal{X}_i=\sum_{j=1}^n\Big((x-x^{(i)})_j \frac{\pa}{\pa x_j}+(y-y^{(i)})_j \frac{\pa}{\pa y_j}\Big)+2(t-t^{(i)}+2(x^{(i)} \cdot y-y^{(i)} \cdot x)) \frac{\pa}{\pa t}.
	$$

We are going to derive a contradiction to \eqref{identity}, by showing that for small $\si>0$,
\be\label{identity1}
\liminf_{i \to  \infty}u_i(\xi_i)^{2} \times \text{RHS of \eqref{identity}}\geq 0	
\ee
and
\be\label{identity2}
\liminf_{i \to  \infty}u_i(\xi_i)^{2}\int_{\pa {B_\si (\xi_i)}}B(\si , \xi_i, u_i, \nabla_{\H}u_i)<0.
\ee
Hence Proposition \ref{prop:5.1} will be established.
	
 Note that  $\frac{Q}{p_i+1} - \frac{Q-2}{2}\geq 0$, we have
$$
\liminf_{i \to  \infty}u_i(\xi_i)^{2} \Big( \frac{Q}{p_i+1} - \frac{Q-2}{2}\Big) \int_{B_\si (\xi_i)}R_iu_i^{p_i+1} \geq 0.
$$
Using  Corollary 5.15 in \cite{PR2003}  we know
 $$
 \liminf_{i \to  \infty} \frac{u_i(\xi_i)^{2}}{p_i +1}\int_{B_\si (\xi_i)} \mathcal{X}_i(R_i)u_i^{p_i+1}= 0.
 $$
It follows from \cite[Proposition 4.3]{PR2003} that
	$$
	0\leq   \int_{\pa  B_\si (\xi_i)} R_i u_i^{p_i+1} \mathcal{X}_i \cdot \nu = O(u_i(\xi_i)^{-p_i-1}),
	$$
	which leads to
	$$\lim_{i \to  \infty}-\frac{u_i(\xi_i)^2}{p_i +1}\int_{\pa  B_\si (\xi_i)} R_i u_i^{p_i+1} \mathcal{X}_i \cdot \nu=0.$$
Thus, we complete the proof of \eqref{identity1}.  It  remains to prove \eqref{identity2}.

In a small punctured disc centered at $0$, we derive from the Bôcher type Lemma in \cite[Proposition 5.7]{PR2003} that
$$
	\lim_{i \to  \infty}u_i(\xi_i)u_i(\xi) = a |\xi|^{2-Q} +b +\al  (\xi),
$$
	where $a, b>0$ are  two constants and $\al  (\xi)$ is a smooth function near $0$ with $\al  (0) = 0$. It follows from Lemma \ref{lem:positivepoho} that, when $\si>0$ is  small,
	$$
	\liminf_{i \to  \infty}u_i(\xi_i)^{2}\int_{\pa {B_\si (\xi_i)}}B(\si , \xi_i, u_i, \nabla_{\H}u_i) = \liminf_{i \to  \infty}\int_{\pa {B_\si (\xi_i)}}B(\si , \xi_i, h_i, \nabla_{\H}h_i)<0,
	$$
where  $h_i(\xi): = u_i(\xi_i)u_i(\xi)$. This gives the proof of  \eqref{identity2}.

In conclusion, from the above arguments we know that there will be no blow up occur under
the hypotheses of Proposition \ref{prop:5.1}. We complete the proof.
\end{proof}
\subsection{Final arguments}
We are  ready to complete the proofs of the main results in this paper.
\begin{proof}[Proof of Theorem \ref{thm:1}]
Let $q_0\in \S^{2n+1}$ be the south pole, we write \eqref{perturbed} as the form \eqref{maineq1} by  using the  CR equivalence.
	Under the hypotheses of  Theorem \ref{thm:1} we know that $R(\xi)$ satisfies
$$
	\| R \|_{L^\infty(\Hn)} \leq A_1,\quad
	R \in C^0(\Hn \backslash B_S), \quad
	\lim_{|\xi| \to  \infty}R(\xi) = R_\infty,
$$
where $A_1 >0$, $S>1$ and $R_\infty >0$ are some constants.
Let $\psi(\xi) \in C^{\infty}(\Hn)$ satisfy   $(R_2)$ and
$$
  \| \psi \|_{C^2(\Hn)}< \infty,\quad
  \lim_{|\xi| \to  \infty}\psi(\xi) =: \psi_{\infty}>0,\quad
    \mathcal{X}(\psi)<0, \quad \forall\, \xi \neq 0,
$$
where $\mathcal{X}$ is the vector field defined by \eqref{vectorfield2}.
It follows from the Kazdan-Warner type condition \eqref{KW2} that
$$
  -\Delta_{\H} u=\psi |u|^{4/(Q-2)}u \quad \text{ in }\,\Hn
$$
has no nontrivial solution in $E$.

For any $\var \in (0,1)$, $k \geq 1$ and $m\geq 2$, let $\overline{k}$ be an integer  such that for any $2\leq s\leq m$ it holds $C_{\over{k}}^s\geq k$, where $C_{\over{k}}^s$ is a combination number. Then we choose $e_1,  \ldots , e_{\overline{k}} \in \pa  B_1$ to be $\overline{k}$ distinct points. Let
$$
  A_{S} = \max_{|\xi|\geq S}|R(\xi)-R_\infty| +\max_{|\xi|\geq S}|\psi(\xi)-\psi_\infty|, \quad S>1,
$$
and $\widetilde{\Om }_l^{(i)}$ be the connected component of
$$
\{ \xi: \var (\psi((e_i)^{-l}\circ \xi)-\psi_\infty)+ R_\infty - A_{\sqrt{l}}>R(\xi) \}
$$
which contains $ (e_i)^l$. Define
$$
  S_l^{(i)} = \min_{1\leq i  \leq m}\sup\{ |(e_i)^{-l}\circ \xi|: \xi \in  \widetilde{\Om }_l^{(i)}\}
$$
and
$$
  R_{\var , k, m, l}(\xi)= \begin{cases} \var (\psi((e_i)^{-l}\circ \xi)-\psi_\infty)+ R_\infty - A_{\sqrt{l}} & \text{ if }\,x \in \widetilde{\Om }_l^{(i)}, \\ R(\xi) & \text { otherwise. }\end{cases}
$$
It is easy to prove that  $\operatorname{diam}(\widetilde{\Om }_l^{(i)}) \leq \sqrt{l}$ for large $l$ and $\lim_{l \to  \infty}S_l^{(i)} = \infty$.

With the function $R_{\var , k, m, l}$ defined above,  we claim that for large $l$, the equation
\be \label{153}
  -\Delta_{\H}u = R_{\var , k, m, l}(\xi)u^{(Q+2)/(Q-2)}, \quad u>0 \quad \text{ in }\,\Hn
\ee
has at least $k$ solutions with $s$ bumps in $E$. To verify it,  let $\{e_{j_1}, \ldots , e_{j_s}\}$ be any distinct $s$ points among $\{e_1, \ldots , e_{\overline{k}}\}$. For $1\leq i\leq s$, we define
\begin{gather*}
  \xi_l^{(i)} = (e_{j_i})^l, \\
  O_l^{(i)} = B_1(\xi_l^{(i)}), \quad \widetilde{O}_l^{(i)} =  B_2(\xi_l^{(i)}), \\
  R_\infty^{(i)} = \var (\psi - \psi_\infty) + R_\infty, \\
  a^{(i)} = \var (\psi(0) - \psi_\infty) + R_\infty.
\end{gather*}
By using  Proposition \ref{prop:5.1},  we  conclude that there exists at least one positive  solution in $V_l(s, \var )$ for large $l$.  Obviously, if we choose a different set of $s$ points among $\{e_1, \ldots , e_{\overline{k}}\}$, we get different solutions since their mass are distributed in different regions by the definition of $V_l(s, \var )$. Due to the choice of $\overline{k}$, \eqref{153} has at least $k$ positive solutions for large $l$.

Finally,  we fix $l$ large enough to make the above arguments work for all $2 \leq s \leq m$, and set $R_{\var , k, m} = R_{\var , k, m, l}$. Evidently, there exist at least $k$ positive  solutions with $s$ $(2 \leq s \leq m)$ bumps to the  equation \eqref{maineq1} with $R=R_{\var , k, m}$.
Theorem \ref{thm:1} is proved  by using the inverse of  Calay transform \eqref{inverse}.
\end{proof}
\begin{proof}[Proof of Corollary \ref{cor:1}]
	One can see from the proof in Theorem \ref{thm:1} that if $\bar{R}\in C^{\infty}(\Hn)$, then
	$\bar{R}_{\var , k, m}-\bar{R}$ can also be achieved.
\end{proof}

\begin{proof}[Proof of Theorem \ref{thm:2}]
We prove it by contradiction argument. Suppose not,  then for some  $\overline{\var }>0$ and $k \geq 2$, there exists a sequence of integers $I_l^{(1)}, \ldots , I_l^{(k)}$ such that
$$
  \lim_{l \to  \infty}|I_l^{(i)}| = \infty, \quad
  \lim_{l \to  \infty}|I_l^{(i)}- I_l^{(j)}| = \infty, \quad i\neq j,
$$
but \eqref{maineq1} has no solution in $V(k, \overline{\var }, B_{\overline{\var }}(\xi_l^{(1)}), \ldots ,B_{\overline{\var }}(\xi_l^{(k)}))$ satisfying $kc-\overline{\var } \leq I_R \leq kc+\overline{\var }$, where $c = (R(\xi^*))^{(2-Q)/2}(S_n)^Q/Q$ and  $\xi_l^{(i)} = (\hat{\xi})^{I_l^{(i)}} \circ\xi^*$.

For  $\var >0$ small, define
\begin{gather*}
R_l(\xi)=R(\xi), \\
O_l^{(i)}=B_{\var }(\xi_l^{(i)}), \quad \widetilde{O}_l^{(i)} = B_{2\var }(\xi_l^{(i)}), \\
S_l = \min_{i \neq j}\big\{ \sqrt{|I_l^{(i)}|}, \sqrt{|I_l^{(i)}- I_l^{(j)}|} \big\},\\
R_\infty^{(i)}(\xi)= R_\infty(\xi) = \lim_{l \to  \infty} R((\hat{\xi})^l\circ \xi),\\
a^{(i)} =R(\xi^*).
\end{gather*}
Obviously,  $R_\infty$ is periodic in $\hat{\xi}$ with respect to left translation and satisfies $(R_2)$ and $R_\infty(\xi^*) = \sup_{\xi \in \Hn}R_\infty(\xi)>0$.
Let $u$ be the positive solution of \eqref{maineq1} with $R(\xi)=R_\infty(\xi)$.
It follows from \cite[Theorem 2.1]{PR2003} that $u$ has no more than one blow up point.
Furthermore, Corollary \ref{cor:A.2} tells us one point blow up may not occur either. Nevertheless, by Proposition  \ref{prop:5.1}, we immediately derive a contradiction.
\end{proof}

\begin{proof}[Proof of Theorem \ref{thm:3}]
The proof is similar to the proof of Theorem \ref{thm:2}, we omit it here.
\end{proof}

\appendix

\section{Refined analysis of blow up profile}\label{app:A}
This appendix is a continuation of  the   blow up analysis studied in \cite{PR2003}. Herein, we present a more detailed characterization of the blow up phenomenon.   We keep using the notation $(z, t) \in \Cn \times \R$ or $(x, y, t) \in \Rn \times \Rn \times \R$ to denote some element $\xi$ of $\Hn$.

\begin{prop}\label{prop:anotherblowup}
Suppose that $\{\bar{R}_i\} \subset \Gamma_{2+\al}(\S^{2n+1})$ with uniform $C^1$ modulo of continuity and satisfies for some point $q_0 \in \S^{2n+1}$, $\var_0>0$, $A_1>0$ independent of $i$ and $2\leq \beta<n$, $$\text{$\{\bar{R}_i\}$ is bounded in $C^{[\beta], \beta-[\beta]}(B(q_0,\var_0))$, \quad $\bar{R}_i(q_0) \geq A_1$}$$  and
	$$
	\bar{R}_i(\xi)=\bar{R}_i(0)+Q_i^{(\beta)}(\xi)+\bar{R}_i(\xi), \quad|\xi| \leq \var_0,
	$$
	where $\xi$ is some pseudo-Hermitian normal coordinates system centered at $q_0$, $Q_i^{(\beta)}(\xi)$ satisfies $Q_i^{(\beta)}(\delta_{\lam}(\xi))=\lam^\beta Q_i^{(\beta)}(\xi)$, $\forall\, \lam>0$, $\xi \in \Hn$, and $R_i(\xi)$ satisfies  $$\sum_{s=0}^{[\beta]}|\nabla^s R_i(\xi)||\xi|^{-\beta+s} \to 0$$ uniformly in $i$ as $\xi \to 0$.
	
	Suppose also that $Q_i^{(\beta)} \to  Q^{(\beta)}$ in $C^1(\S^{2n+1})$ and for some constant $A_2>0$ that
\be\label{5.13}
	A_2|\xi|^{\beta-1} \leq|\nabla Q^{(\beta)}(\xi)|, \quad|\xi| \leq \var_0,
\ee
	and
\be\label{5.14}
\left(\begin{aligned}
	\int \widetilde{X} Q^{(\beta)}(\hat{\xi} \circ \xi)w_{0,1}^{2Q/(Q-2)} \\
\int Q^{(\beta)}(\hat{\xi} \circ \xi)w_{0,1}^{2Q/(Q-2)}
\end{aligned}\right) \neq 0, \quad \forall\, \hat{\xi} \in \Hn,
\ee
where   $\widetilde{X}:=(X_1,\ldots,X_n,Y_1,\ldots,Y_n,T)$.
 Let $v_i$ be positive solutions of \eqref{maineq} with $\bar{R}=\bar{R}_i$. 		If $q_0$ is an isolated simple blow up point of $v_i$, then $v_i$ has to have at least another blow up point.
\end{prop}
\begin{proof}
 Suppose the contrary: $q_0$ is the only blow up point of $v_i$. We first make a Cayley transform  with $q_0$ being the north pole  with inverse $\mathcal{C}$, then  equation \eqref{maineq} with $\bar{R}=\bar{R}_i$ is equivalent to
\be\label{5.15}
-\Delta_{\H} u_i=R_i(\xi)u_i^{(Q+2)/(Q-2)},\quad u_i>0\quad \text{ in }\, \Hn,
\ee
where
 $$
 u_i(\xi)=\Big(\frac{2^{2 n+2}}{((1+|z|^2)^2+t^2)^{n+1}}\Big)^{\frac{Q-2}{2Q}} v_i(\mathcal{C}(\xi))\quad  \text{ and }\quad  R_i(\xi) =\bar{R}_i(\mathcal{C}(\xi)) .
 $$
 It is easy to see that our hypotheses hold in the Heisenberg coordinates.

 Let $\xi_i \to  0$ be the local maximum of $u_i$. It follows from \cite[Lemma 5.12]{PR2003} that
 $$
|\nabla R_i(\xi_i)|=O\big(u_i(\xi_i)^{-2}+u_i(\xi_i)^{-(2 / (Q-2))([\beta]-1+ \beta-[\beta]) / (\beta-1)}\big) .
 $$	
  First we establish that
\be\label{5.16}
 |\xi_i|=O\big(u_i(\xi_i)^{-2 /(Q-2)}\big) .
\ee
  Since we have assumed that $v_i$ has no other blow up point other than $q_0$, it follows from \cite[Proposition 5.7]{PR2003} and the Harnack inequality that   $u_i(\xi) \leq C(\var)u_i(\xi_i)^{-1}|\xi|^{2-Q} $ for $|\xi|\geq \var>0$.

Let $X$ be any left invariant vector field in \eqref{vectorfields}. It follows from the Kazdan-Warner type condition \eqref{KW1} that
\be\label{5.17}
 \int  X R_i u_i^{2 Q /(Q-2)} =0 .
\ee
Then  for $\var>0$ small we have
 $$
	 \Big|\int_{B_\var} \nabla_{\H} R_i(\xi_i\circ\xi) u_i(\xi_i\circ\xi)^{2 Q /(Q-2)}\Big| \leq  C(\var) u_i(\xi_i)^{-2 Q /(Q-2)}.
 $$
  Using our hypotheses $\nabla Q^{(\beta)}$ and $R_i$, we have
 $$
  \Big|\int_{B_\var}(1+o_{\var}(1))\nabla_{\H} Q_i^{(\beta)}(\xi_i\circ\xi) u_i(\xi_i\circ\xi)^{2 Q /(Q-2)}\Big|\leq  C(\var ) u_i(\xi_i)^{-2 Q /(Q-2)}.
 $$Multiplying the above by $m_i^{(2/(Q-2))(\beta-1)}$ with $m_i=u_i(\xi_i)$ we have
$$
\Big|\int_{B_{\var}}(1+o_{\var}(1)) \nabla_{\H} Q_i^{(\beta)}(\tilde{\xi}_i\circ \delta_{m_i^{2Q/(Q-2)}}(\xi)) u_i(\xi_i\circ\xi)^{2 Q /(Q-2)}\Big|
\leq C(\var) u_i(\xi_i)^{(2 /(Q-2 ))(\beta-1-n)},
$$
 where $\tilde{\xi}_i=u_i(\xi_i)^{2 / (Q-2)} \xi_i$.
 Suppose that \eqref{5.16} is false, namely, $|\tilde{\xi}_i| \to  \infty$ along a subsequence. Then it follows from \cite[Proposition 5.2]{PR2003} (we may choose $S_i\leq |\tilde{\xi}_i|/4$) that
\begin{align*}
& \Big|\int_{|\xi| \leq S_i m_i^{-2 /(Q-2 )}}(1+o_{\var}(1)) \nabla_{\H} Q_i^{(\beta)}(\tilde{\xi}_i\circ \delta_{m_i^{2Q/(Q-2)}}(\xi)) u_i(\xi_i\circ\xi)^{2 Q /(Q-2)}\Big| \\
=&\Big|\int_{|\hat{\xi}| \leq R_i}(1+o_{\var}(1)) \nabla_{\H} Q_i^{(\beta)}(\tilde{\xi}_i\circ\hat{\xi})(m_i^{-1} u_i(\xi_i\circ\delta_{m_i^{2/(Q-2)}}(\hat{\xi}) ))^{2 Q /(Q-2 )}\Big| \sim|\tilde{\xi}_i|^{\beta-1} .
\end{align*}

On the other hand, it follows from \cite[Lemma 5.10]{PR2003} that
\begin{align*}
& \Big|\int_{S_i m_i^{-2 /(Q-2 )} \leq|\xi| \leq \var}(1+o_{\var}(1)) \nabla_{\H} Q_i^{(\beta)}(\tilde{\xi}_i \circ \delta_{m_i^{2/(Q-2)}}(\xi)) u_i(\xi_i\circ \xi)^{2 Q /(Q-2)}\Big| \\
 \leq& C\Big|\int_{S_i m_i^{-2 /(Q-2)} \leq|\xi| \leq \var}(|\delta_{m_i^{2 /(Q-2)}}(\xi) |^{\beta-1}+|\tilde{\xi}_i|^{\beta-1}) u_i(\xi_i\circ \xi)^{2Q /(Q-2)}\Big|\\ \leq& o(1)|\tilde{\xi}_i|^{\beta-1} .
\end{align*}
 It follows that
 $$
 |\tilde{\xi}_i|^{\beta-1} \leq  C(\var) u_i(\xi_i)^{(2 /(Q-2))(\beta-1-Q)} ,
 $$
 which implies that
 $$
|\xi_i| \leq  C(\var) m_i^{-(2 Q /(Q-2))(Q/(\beta-1))}=o(m_i^{-2 /(Q-2)}) .
 $$
 This contradicts to $|\tilde{\xi}|\to \infty$. Thus \eqref{5.16} holds.

 We are going to find some $\xi_0$ such that \eqref{5.14} fails.

For $1\leq j\leq n$, define the vector fields
 $$
 \overline{X}_j=\frac{\pa}{\pa x_j}-2 y_j \frac{\pa}{\pa t}\quad \text{ and }\quad  \overline{Y}_j=\frac{\pa}{\pa y_j}-2 x_j \frac{\pa}{\pa t}.
 $$
Multiplying \eqref{5.15} by $\overline{X}_j$, $\overline{Y}_j$, $T$  and integrate by parts, together with the Kazdan-Warner condition \eqref{KW2} we have
 $$
 \int \Big(\sum_{j=1}^n (x_j X_j+y_j Y_j)+2 t T\Big) R_i(\xi_i \circ \xi) u_i(\xi_i\circ \xi)^{2Q /(Q-2)}=0 .
 $$
 Since $q_0$ is an isolated simple blow up point and the only blow up point of $v_i$, we have for any $\var>0$,
 $$
 \Big|\int_{B_\var} \Big(\sum_{j=1}^n (x_j X_j+y_j Y_j)+2 t T\Big) R_i(\xi_i \circ \xi) u_i(\xi_i\circ \xi)^{2Q /(Q-2)} \Big| \leq  C(\var) u_i(\xi_i)^{2Q /(Q-2)} .
 $$
 It follows from  \cite[Lemma 5.10]{PR2003} and our hypotheses on $R_i$ that
 \begin{align*}
 &\Big| \int_{B_\var} \Big(\sum_{j=1}^n (x_j X_j+y_j Y_j)+2 t T\Big)  Q_i^{(\beta)}(\xi_i\circ \xi) u_i(\xi_i\circ \xi)^{2 Q /(Q-2)}\Big| \\
  \leq&  C(\var ) u_i(\xi_i)^{-2 Q /(Q-2)} +o_\var(1) \int_{B_\var}|\xi||\xi_i\circ \xi |^{\beta-1} u_i(\xi_i\circ \xi)^{2 Q /(Q-2)}\\&+o_\var(1) \int_{B_\var}|\xi|^2|\xi_i\circ \xi |^{\beta-2} u_i(\xi_i\circ \xi)^{2 Q /(Q-2)}\\
 \leq  & C(\var ) u_i(\xi_i)^{-2 Q /(Q-2)}  +o_\var(1) \int_{B_{\var }}(|\xi|^\beta+|\xi||\xi_i |^{\beta-1}) u_i(\xi_i\circ \xi)^{2 Q /(Q-2)} \\\leq &C(\var) u_i(\xi_i)^{-2 Q /(Q-2)}+o_{\var}(1) u_i(\xi_i)^{-2 \beta /(Q-2 )} ,
 \end{align*}
where we used \eqref{5.16} in the last inequality.

Multiplying the above by $u_i(\xi_i)^{2 \beta /(Q-2)}$, due to $\beta<n$ we obtain
\be\label{5.19}
 \lim _{i \to  \infty} u_i(\xi_i)^{2 \beta / (Q-2)}\Big|\int_{B_\var} \Big(\sum_{j=1}^n (x_j X_j+y_j Y_j)+2 t T\Big)  Q_i^{(\beta)}(\xi_i\circ \xi)  u_i(\xi_i\circ \xi)^{2Q /(Q-2)}\Big| =o_{\var}(1) .
 \ee
 Let $S_i \to \infty$ as $i \to \infty$. We assume that $r_i:=R_i u_i(\xi_i)^{-2 /(Q-2)} \to 0$.  By  \cite[Lemma 5.10]{PR2003}, we have
 \begin{align}
 & u_i(\xi_i)^{2 \beta /(Q-2)}\Big|\int_{r_i \leq |\xi| \leq  \var }\Big(\sum_{j=1}^n (x_j X_j+y_j Y_j)+2 t T\Big)  Q_i^{(\beta)}(\xi_i\circ \xi) u_i(\xi_i\circ \xi)^{2 Q /(Q-2)} \Big| \notag\\
  \leq &\lim _{i \to \infty} u_i(\xi_i)^{2 \beta /(Q-2)}\Big|\int_{r_i \leq|\xi| \leq \var}(|\xi|^\beta+|\xi||\xi_i|^{\beta-1}) u_i(\xi_i\circ \xi)^{2 Q /(Q-2 )}\Big| \to 0\label{5.20}
 \end{align}
 as $i \to \infty$. Combining \eqref{5.19} and \eqref{5.20}, we conclude that
 $$
 \lim _{i \to \infty} u_i(\xi_i)^{2 \beta /(Q-2)}\Big|\int_{B_{r_i}}\Big(\sum_{j=1}^n (x_j X_j+y_j Y_j)+2 t T\Big)  Q_i^{(\beta)}(\xi_i\circ \xi)u_i(\xi_i\circ \xi)^{2 Q /(Q-2)}\Big|=o_{\var}(1) .
 $$
 It follows from the change of variable $\bar{\xi}=(\bar{z},\bar{t})=u_i(\xi_i)^{2 / Q-2)} \xi$, applying \cite[Proposition 5.2]{PR2003} and then letting $\var\to 0$ that
\be\label{5.21}
 \Big| \int  \Big(\sum_{j=1}^n (\bar{x}_j X_j+\bar{y}_j Y_j)+2 \bar{t} T\Big) Q^{(\beta)}(\xi_0\circ\bar{\xi})\Lam^{Q}w_{0,\Lam}^{2Q/(Q-2)}\, \d \bar{z}\, \d\bar{t}  \Big|=o_{\var}(1),
\ee
	 where  $\xi_0=\lim _{i \to  \infty} u_{i}(\xi_i)^{2/(Q-2)}\xi_i$ and $\Lam=\lim _{i \to  \infty} \sqrt{\frac{R_i(\xi_i)}{2 Q(Q-2)}}$.

 On the other hand, from \eqref{5.17},
$$
\int X R_i(\xi_i\circ\xi) u_i(\xi_i\circ\xi)^{2 Q /(Q-2)}=0 .
$$
 Arguing as above, we will have
 \be\label{5.23}
 \int X Q^{(\beta)}(\xi_0\circ \xi))\Lam^{Q} w_{0,\Lam}^{2Q/(Q-2)} =0.
 \ee
 It follows from \eqref{5.21} and \eqref{5.23} that
 $$
 \int  Q^{(\beta)}(\xi_0\circ \xi)w_{0,\Lam}^{2Q/(Q-2)}  =\beta^{-1} \int  (\xi_0\circ \xi)\cdot \widehat{X} Q^{(\beta)}(\xi_0\circ \xi) w_{0,\Lam}^{2Q/(Q-2)}=0 ,
 $$
 where $\widehat{X}=(X_1,\ldots,X_n,Y_1,\ldots,Y_n,2T)$.
Therefore, \eqref{5.14} does not hold for $\hat{\xi}=\delta_{\Lam}(\xi_0)$. Proposition \ref{prop:anotherblowup} is established.
\end{proof}

\begin{cor}\label{cor:A.1}
 Suppose that $\{\bar{R}_i\} \in C^1(\S^{2n+1})$ with uniform $C^1$ modulo of continuity and satisfies for some $q_0 \in \S^{n+1}$, $\var_0>0$, $A_1>0$ independent of $i$ and $2 \leq \beta<Q$,
$$
\bar{R}_i  \in C^{[\beta]-1,1}(B(q_0,\var_0)),\quad  \bar{R}_i(q_0) \geq 1 / A_1, $$
and
$$
\bar{R}_i(\xi)  =\bar{R}_i(0)+Q_i^{(\beta)}(\xi)+P_i(\xi),   \quad |\xi| \leq \var_0,
$$
where $\xi$ is some pseudo-Hermitian normal coordinates system centered at $q_0$, $R_i(\xi)$ satisfies $$\sum_{s=0}^{[\beta]}|\nabla^s R_i(\xi)||\xi|^{-\beta+s} \to  0$$ uniformly for $i$ as $\xi \to  0$,  $Q_i^{(\beta)}(\xi)=\sum_{j=1}^n (a_j(i) |x_j|^{\beta-1} x_j+b_j(i) |y_j|^{\beta-1} y_j)+c_i |t|^{\frac{\beta}{2}-1}t$, $a_j(i) \to  a_j$, $b_j(i) \to  b_j$, $c_i\to c$ as $i \to  \infty$, $a_j, b_j, c \neq 0$,  $\forall\, 1 \leq j \leq n$.
Let $v_i$ be positive solutions of \eqref{maineq} with $\bar{R}=\bar{R}_i$.  Then if $q_0$ is an isolated simple blow up point of $v_i$, $v_i$ has to have at least another blow up point.	
\end{cor}
\begin{proof}
Let $Q^{(\beta)}(\xi)=\sum_{j=1}^n a_j|x_j|^{\beta-1} x_j+\sum_{j=1}^n b_j|y_j|^{\beta-1} y_j+c|t|^{\frac{\beta}{2}-1}t$. We only have to check that \eqref{5.13} and \eqref{5.14} hold under our hypotheses. The first one  \eqref{5.13} is obvious.  It remains to prove \eqref{5.14}.

For any $\xi_0=(x_0^{(1)},\ldots,x_0^{(n)},y_0^{(1)},\ldots,y_0^{(n)},t_0) \in \Hn$, we have
\begin{align*}
\int T Q^{(\beta)}(\xi_0\circ \xi)  w_{0,1}^{2Q/(Q-2)}
= \frac{\beta c}{2}\int \Big|t+t_0+2\sum_{i=1}^{n}(x_iy_0^{(i)}-y_ix_0^{(i)})\Big|^{\frac{\beta}{2}-1}  w_{0,1}^{2Q/(Q-2)} \neq  0 .
\end{align*}
Thus \eqref{5.14} is established under our hypotheses. Corollary \ref{cor:A.1} follows  immediately.
\end{proof}
\begin{cor}\label{cor:A.2}
 Suppose that $\{\bar{R}_i\} \in C^1(\S^{n+1})$ with uniform $C^1$ modulo of continuity and satisfies for some $q_0 \in \S^{n+1}$, $\var_0>0$, $A_1>0$ independent of $i$ and $2 \leq \beta<Q$, that
$$
	\bar{R}_i \in C^{[\beta]-1,1}(B(q_0,\var_0)), \quad  \bar{R}_i(q_0) \geq 1 / A_1, $$
		and
$$		
	\bar{R}_i(\xi)  =\bar{R}_i(0)+Q_i^{(\beta)}(\xi)+P_i(\xi), \quad  |\xi| \leq \var_0,
$$
where $\xi$ is some pseudo-Hermitian normal coordinates system centered at $q_0$, $R_i(y)$ satisfies $$\sum_{s=0}^{[\beta]}|\nabla^s R_i(\xi)||\xi|^{-\beta+s} \to  0$$ uniformly for $i$ as $\xi \to  0$,  $Q_i^{(\beta)}(\xi)=\sum_{j=1}^n (a_j(i)|x_j|^\beta+b_j(i)|y_j|^\beta)+c(i)|t|^{\frac{\beta}{2}}$, $a_j(i) \to  a_j$, $b_j(i) \to  b_j$, $c(i) \to  c$ as $i \to  \infty$, $a_j,b_j,c \neq 0$, $ \forall\, 1 \leq j \leq n$, and $\sum_{j=1}^n (a_j+b_j)+\kappa c \neq 0$ with
$$\kappa=\frac{\int|x_1|^\beta w_{0,1}^{2Q/(Q-2)}}{\int |t|^{\frac{\beta}{2}}w_{0,1}^{2Q/(Q-2)}}.$$
	Let $v_i$ be positive solutions of \eqref{maineq} with $\bar{R}=\bar{R}_i$. Then if $q_0$ is an isolated simple blow up point of $v_i$, $v_i$ has to have at least another blow up point.
\end{cor}
\begin{proof}
 Let $Q^{(\beta)}(\xi)=\sum_{j=1}^n (a_j|x_j|^\beta+b_j|y_j|^\beta)+c|t|^{\frac{\beta}{2}}$. We only have to check that \eqref{5.13} and \eqref{5.14} hold under our hypotheses. The first one \eqref{5.13} is obvious.  It remains to prove \eqref{5.14}.

For any $\xi_0=(x_0^{(1)},\ldots,x_0^{(n)},y_0^{(1)},\ldots,y_0^{(n)},t_0) \in \Hn$, we have
\begin{align*}
 X_j Q^{(\beta)}(\xi_0\circ \xi)  =&
\beta a_j|x_j+x_0^{(j)}|^{\beta-2}(x_j+x_0^{(j)})\\&+\beta c\Big|t+t_0+2\sum_{i=1}^{n}(x_iy_0^{(i)}-y_ix_0^{(i)})\Big|^{\frac{\beta}{2}-2}\Big(t+t_0+2\sum_{i=1}^{n}(x_iy_0^{(i)}-y_ix_0^{(i)})\Big)y_0^{(j)}\\&+\beta cy_j\Big|t+t_0+2\sum_{i=1}^{n}(x_iy_0^{(i)}-y_ix_0^{(i)})\Big|^{\frac{\beta}{2}-2}\Big(t+t_0+2\sum_{i=1}^{n}(x_iy_0^{(i)}-y_ix_0^{(i)})\Big),\\Y_j Q^{(\beta)}(\xi_0\circ \xi)  =&
\beta b_j|y_j+y_0^{(j)}|^{\beta-2}(y_j+y_0^{(j)})\\&-\beta c\Big|t+t_0+2\sum_{i=1}^{n}(x_iy_0^{(i)}-y_ix_0^{(i)})\Big|^{\frac{\beta}{2}-2}\Big(t+t_0+2\sum_{i=1}^{n}(x_iy_0^{(i)}-y_ix_0^{(i)})\Big)x_0^{(j)}\\&-\beta cx_j\Big|t+t_0+2\sum_{i=1}^{n}(x_iy_0^{(i)}-y_ix_0^{(i)})\Big|^{\frac{\beta}{2}-2}\Big(t+t_0+2\sum_{i=1}^{n}(x_iy_0^{(i)}-y_ix_0^{(i)})\Big),
\end{align*}
and $$TQ^{(\beta)}(\xi_0\circ \xi)  =\frac{\beta c}{2}\Big|t+t_0+2\sum_{i=1}^{n}(x_iy_0^{(i)}-y_ix_0^{(i)})\Big|^{\frac{\beta}{2}-2}\Big(t+t_0+2\sum_{i=1}^{n}(x_iy_0^{(i)}-y_ix_0^{(i)})\Big).$$
It is straightforward to verify that
$$
\int (1+|\xi|^2)^{-n} X Q^{(\beta)}(\xi_0\circ \xi) =0 \quad \text{ for each }\,X\in \{X_1\ldots,X_n,Y_1,\ldots,Y_n,T\} \quad \text{ iff }\, \xi_0=0.
$$
Next we have
\begin{align*}
& \int Q^{(\beta)}(\xi)w_{0,1}^{2Q/(Q-2)}  \\
=& \int \Big(\sum_{j=1}^n (a_j|x_j|^\beta+b_j|y_j|^\beta+c|t|^{\frac{\beta}{2}}\Big)w_{0,1}^{2Q/(Q-2)}  \\
=&\Big(\sum_{j=1}^n (a_j+b_j)\Big) \int|x_1|^\beta w_{0,1}^{2Q/(Q-2)}+c\int |t|^{\frac{\beta}{2}}w_{0,1}^{2Q/(Q-2)}  \neq 0 .
\end{align*}
Thus \eqref{5.14} is established under our hypotheses.  Corollary \ref{cor:A.2} follows  immediately.
\end{proof}

\subsection*{Conflict of interest statement} On behalf of all authors, the corresponding author states that there is no conflict of interest.

\subsection*{Data availability statement} Data sharing not applicable to this article as no datasets were generated or analysed during the current study.

\medskip

\noindent Z. Tang

\noindent School of Mathematical Sciences, Laboratory of Mathematics and Complex Systems, MOE\\
Beijing Normal University, Beijing 100875, People's Republic of China\\
Email: \textsf{tangzw@bnu.edu.cn}

\medskip
\noindent H. Wang

\noindent School of Mathematical Sciences, Laboratory of Mathematics and Complex Systems, MOE\\
Beijing Normal University, Beijing 100875, People's Republic of China\\
Email: \textsf{hmw@mail.bnu.edu.cn}

\medskip

\noindent B. Zhang

\noindent School of Mathematical Sciences, Laboratory of Mathematics and Complex Systems, MOE\\
Beijing Normal University, Beijing 100875, People's Republic of China\\
Email: \textsf{bingweizhang@mail.bnu.edu.cn}

\end{document}